\documentclass[11pt,a4paper,reqno]{amsart}
\usepackage{amsfonts,amssymb,amsmath,amsthm,graphicx,color,amscd,xspace,verbatim,dsfont}
\usepackage{scalerel}
\makeatletter

\usepackage{hyperref}
\makeatletter
\def\@tocline#1#2#3#4#5#6#7{\relax
  \ifnum #1>\c@tocdepth % then omit
  \else
    \par \addpenalty\@secpenalty\addvspace{#2}%
    \begingroup \hyphenpenalty\@M
    \@ifempty{#4}{%
      \@tempdima\csname r@tocindent\number#1\endcsname\relax
    }{%
      \@tempdima#4\relax
    }%
    \parindent\z@ \leftskip#3\relax \advance\leftskip\@tempdima\relax
    \rightskip\@pnumwidth plus4em \parfillskip-\@pnumwidth
    #5\leavevmode\hskip-\@tempdima
      \ifcase #1
       \or\or \hskip 1em \or \hskip 2em \else \hskip 3em \fi%
      #6\nobreak\relax
      \dotfill
      \hbox to\@pnumwidth{\@tocpagenum{#7}}
    \par
    \nobreak
    \endgroup
  \fi}
\makeatother

\newtheorem{theorem}{Theorem}[section]
\newtheorem{lemma}[theorem]{Lemma}
\newtheorem{proposition}[theorem]{Proposition}

\theoremstyle{definition}
\newtheorem{definition}[theorem]{Definition}

\newcommand{\N}{{\mathbb N}}
\newcommand{\R}{{\mathbb R}}

\newcommand{\BB}{\mathbb{B}}

\newcommand{\tr}{\mathrm{tr}^*}

\newcommand{\beqn}{\begin{eqnarray}}
\newcommand{\eeqn}{\end{eqnarray}}   % equation with number
\newcommand{\beq}{\begin{eqnarray*}}
\newcommand{\eeq}{\end{eqnarray*}}

\newcommand{\be}{\small\begin{equation}}
\newcommand{\bel}[1]{\small\begin{equation}\label{#1}}
\newcommand{\ee}{\end{equation}\normalsize}%%
\newcommand{\BA}{\begin{array}}
\newcommand{\EA}{\end{array}}
\newcommand{\BAN}{\renewcommand{\arraystretch}{1.2}
\setlength{\arraycolsep}{2pt}\begin{array}}
\newcommand{\BAV}[2]{\renewcommand{\arraystretch}{#1}
\setlength{\arraycolsep}{#2}\begin{array}}
\newcommand{\BSA}{\begin{subarray}}
\newcommand{\ESA}{\end{subarray}}
\newcommand{\BAL}{\begin{aligned}}
\newcommand{\EAL}{\end{aligned}}

\newcommand{\forevery}{\quad \forall}

\newcommand{\abs}[1]{\left |#1\right |}

\newcommand{\norm}[1]{\left \|#1\right \|}%% adjustable norm
\newcommand{\supp}{\mathrm{supp}\,}
\newcommand{\dist}{\mathrm{dist}\,}
\newcommand{\sign}{\mathrm{sign}}
\newcommand{\diam}{\mathrm{diam}\,}
\newcommand{\prt}{\partial}

\newcommand{\sbs}{\subset}

\def\dist{\mathrm{dist}}

\def \dd {\mathrm{d}}

\def\ga{\alpha}            
       \def\gd{\delta}      \def\ge{\epsilon}
                         
\def\gf{\phi}           
            \def\gl{\lambda}
\def\gm{\mu}        \def\gn{\nu}         
            
\def\gs{\sigma}       \def\gt{\tau}
      \def\gw{\omega}
                
     \def\Gd{\Delta}      

\def\Gl{\Lambda}          
\def\Gw{\Omega}              

\def\CS{{\mathcal S}}      
      
\def\CA{{\mathcal A}}   \def\CB{{\mathcal B}}   \def\CC{{\mathcal C}}
\def\CD{{\mathcal D}}      \def\CF{{\mathcal F}}
      
      \def\CL{{\mathcal L}}

   \def\BBB {\mathbb B}    
       
\def\BBG {\mathbb G}       
   \def\BBK {\mathbb K}    
   \def\BBN {\mathbb N}    
\def\BBP {\mathbb P}   \def\BBR {\mathbb R}    
\def\BBT {\mathbb T}       
\def\BBW {\mathbb W}

\def\GTM {\mathfrak M}

\def\tr{\mathrm{tr}}

\newcommand{\ei}{{\phi_{\xm }}}
\newcommand{\xa}{\alpha}
\newcommand{\xb}{\beta}
\newcommand{\xg}{\gamma}
\newcommand{\xG}{\Gamma}
\newcommand{\xd}{\delta}
\newcommand{\xD}{\Delta}
\newcommand{\xe}{\varepsilon}
\newcommand{\xz}{\zeta}

\newcommand{\xl}{\lambda}

\newcommand{\xm}{\mu}
\newcommand{\xn}{\nu}

\newcommand{\xr}{\rho}

\newcommand{\xS}{\Sigma}
\newcommand{\xf}{\phi}
\newcommand{\xF}{\Phi}

\newcommand{\xo}{\omega}
\newcommand{\xO}{\Omega}

\newcommand{\myfrac}[2]{{\displaystyle \frac{#1}{#2} }}
\newcommand{\myint}[2]{{\displaystyle \int_{#1}^{#2}}}
\def \dd {\mathrm{d}}

\def \dS {\mathrm{d}S}

\newcommand{\ap}{{\xa_{\scaleto{+}{3pt}}}}
\newcommand{\am}{{\xa_{\scaleto{-}{3pt}}}}

\newcommand\1{{\ensuremath {\mathds 1} }}
\def\bal#1\eal{\small\begin{align*}#1\end{align*}\normalsize}
\def\ba#1\ea{\small\begin{align}#1\end{align}\normalsize}
\numberwithin{equation}{section}

\begin{document}

\title[Semilinear elliptic Schr\"odinger equations]{Semilinear elliptic Schr\"odinger equations with singular potentials and absorption terms}%%
%%}%% Title to be inserted !!!!!!!!!!
\author{Konstantinos T. Gkikas}
\address{Konstantinos T. Gkikas, Department of Mathematics, National and Kapodistrian University of Athens, 15784 Athens, Greece}
\email{kugkikas@math.uoa.gr}

\author[P.T. Nguyen]{Phuoc-Tai Nguyen}
\address{Phuoc-Tai Nguyen, Department of Mathematics and Statistics, Masaryk University, Brno, Czech Republic}
\email{ptnguyen@math.muni.cz}

\date{\today}

\begin{abstract}
Let $\Omega \subset \mathbb{R}^N$ ($N \geq 3$) be a $C^2$ bounded domain and  $\Sigma \subset \Omega$ be a compact, $C^2$ submanifold without boundary, of dimension $k$ with $0\leq k < N-2$. Put $L_\mu = \Delta + \mu d_\Sigma^{-2}$ in $\Omega \setminus \Sigma$, where $d_\Sigma(x) = \dist(x,\Sigma)$ and $\mu$ is a parameter. We investigate the boundary value problem (P) $-L_\mu u + g(u) = \tau$ in $\Omega \setminus \Sigma$ with condition $u=\nu$ on $\partial \Omega \cup \Sigma$, where $g: \mathbb{R} \to \mathbb{R}$ is a nondecreasing, continuous function, and $\tau$ and $\nu$ are positive measures. The complex interplay between the competing effects of the inverse-square potential $d_\Sigma^{-2}$, the absorption term $g(u)$ and the measure data $\tau,\nu$ discloses different scenarios in which problem (P) is solvable. We provide sharp conditions on the growth of $g$ for the existence of solutions. When $g$ is a power function, namely $g(u)=|u|^{p-1}u$ with $p>1$, we show that problem (P) admits several critical exponents in the sense that singular solutions exist in the subcritical cases (i.e. $p$ is smaller than a critical exponent) and singularities are removable in the supercritical cases (i.e. $p$ is greater than a critical exponent). Finally, we establish various necessary and sufficient conditions expressed in terms of appropriate capacities for the solvability of (P).
\medskip


\medskip


\end{abstract}

\maketitle
\tableofcontents
\section{Introduction}
\subsection{Background and aim}
Let $\Omega \subset \R^N$ ($N \geq 3$) be a $C^2$ bounded domain and $\Sigma\subset\xO$ be a compact, $C^2$ submanifold in $\R^N$ without boundary, of dimension $k$ with $0 \leq k < N-2$. Denote $d(x)=\dist(x,\partial\xO)$ and $d_\Sigma(x) = \dist(x,\Sigma)$. For $\mu \in \R$, let $L_\mu$ be the Schr\"odinger operator with the inverse-square potential $d_\Sigma^{-2}$
\bal L_\mu = L_\mu^{\Omega,\Sigma}:=\Delta + \frac{\mu}{d_\Sigma^2}
\eal
in $\Omega \setminus \Sigma$. 
%This type of operators has received a lot of attention in the literature because of its important role in quantum mechanics. 
The study of $L_\mu$ is closely connected to the optimal Hardy constant $\CC_{\Omega,\Sigma}$ and the fundamental exponent $H$ given below
\be \label{H} \CC_{\Omega,\Sigma}:=\inf_{\varphi \in H^1_0(\Omega)}\frac{\int_\Omega |\nabla \varphi|^2\dd x}{\int_\Omega d_\Sigma^{-2}\varphi^2 \dd x} \quad \text{and} \quad H:=\frac{N-k-2}{2}.
\ee

Obviously, $H \leq \frac{N-2}{2}$ and $H=\frac{N-2}{2}$ if and only if $\Sigma$ is a singleton. It is well known that ${\mathcal C}_{\xO,\Sigma}\in (0,H^2]$ (see D\'avila and Dupaigne \cite{DD1, DD2} and Barbatis, Filippas and Tertikas \cite{BFT}) and  ${\mathcal C}_{\Omega,\{0\}}=\left(\frac{N-2}{2} \right)^2$. Moreover, ${\mathcal C}_{\Omega,\Sigma}=H^2$ provided that $-\Delta d_\Sigma^{2+k-N} \geq 0$ in the sense of distributions in $\Omega \setminus \Sigma$ or if $\Omega=\Sigma_\beta$ with $\beta$  small enough (see \cite{BFT}), where
\bal
\Sigma_\beta :=\{ x \in \R^N \setminus \Sigma: d_\Sigma(x) < \beta \}.
\eal
For $\mu \leq H^2$, let $\am$ and $\ap$ be the roots of the algebraic equation $\ga^2 - 2H\ga + \mu=0$, i.e.
\be \label{apm}
\am:=H-\sqrt{H^2-\mu}, \quad \ap:=H+\sqrt{H^2-\mu}.
\ee
We see that $\am\leq H\leq\ap \leq 2H$, and $\am \geq 0$ if and only if $\mu \geq 0$.

By \cite[Lemma 2.4 and Theorem 2.6]{DD1} and \cite[page 337, Lemma 7, Theorem 5]{DD2},
\bal\lambda_\mu:=\inf\left\{\int_{\Gw}\left(|\nabla u|^2-\frac{\xm }{d_\Sigma^2}u^2\right)\dd x: u \in C_c^1(\Omega), \int_{\Gw} u^2 \dd x=1\right\}>-\infty.
\eal
Note that $\lambda_\mu$ is the first eigenvalue associated to $-L_\mu$ and its corresponding eigenfunction $\phi_\mu$, with normalization $\| \phi_\mu \|_{L^2(\Omega)}=1$, satisfies two-sided estimate $\phi_\mu \approx d\,d_\Sigma^{-\am}$ in $\Omega \setminus \Sigma$ (see subsection \ref{subsec:eigen} for more detail). The sign of $\lambda_\mu$ plays an important role in the study of $L_\mu$. If $\mu<\CC_{\Omega,\Sigma}$ then $\lambda_\xm>0$. However, in general, this does not hold true.  Under the assumption $\lambda_\mu>0$, the authors of the present paper obtained the existence and sharp two-sided estimates of the Green function $G_\mu$ and Martin kernel $K_\mu$ associated to $-L_\mu$ (see \cite{GkiNg_linear}) which are crucial tools in the study of the boundary value problem with measures data for linear equations involving $L_\mu$
\be \label{eq:linear} \left\{ \begin{aligned}
-L_\mu u &= \tau \quad &&\text{in } \Omega \setminus \Sigma, \\
\tr(u) &= \nu, &&\,
\end{aligned} \right.
\ee
where $\tau \in \GTM(\Omega;\phi_\mu)$ (i.e. $\int_{\Omega \setminus \Sigma}\phi_\mu \dd |\tau|<\infty$) and $\nu \in \GTM(\partial \Omega \cup \Sigma)$ (i.e. $\int_{\partial \Omega \cup \Sigma}\dd |\nu| < \infty$).

In \eqref{eq:linear}, $\tr(u)$ denotes the \textit{boundary trace} of $u$ on $\partial \Omega \cup \Sigma$ which was defined in \cite{GkiNg_linear} in terms of harmonic measures of $-L_\mu$  (see Subsection \ref{subsec:boundarytrace}). A highlighting property of this notion is $\tr(\BBG_\mu[\tau]) = 0$ for any $\tau \in \GTM(\Omega \setminus \Sigma;\phi_\mu)$ and  $\tr(\BBK_\mu[\tau]) = \nu$ for any $\nu \in \GTM(\partial \Omega \cup \Sigma)$, where
\bal
\BBG_\mu[\tau](x) &= \int_{\Omega \setminus \Sigma}G_\mu(x,y)\, \dd\tau(y), \quad \tau \in \GTM(\Omega \setminus \Sigma;\phi_\mu), \\
\BBK_\mu[\nu](x) &= \int_{\Omega \setminus \Sigma}K_\mu(x,y)\, \dd\nu(y), \quad \nu \in \GTM(\partial \Omega \cup \Sigma).
\eal
Note that for a positive measure $\tau$, $\BBG_\mu[\tau]$ is finite in $\Omega \setminus \Sigma$ if and only if $\tau \in \GTM(\Omega \setminus \Sigma; \ei)$.

It was shown in \cite{GkiNg_linear} that $\BBG_\mu[\tau]$ is the unique solution of \eqref{eq:linear} with $\nu=0$, and $\BBK_\mu[\nu]$ is the unique solution of \eqref{eq:linear} with $\tau=0$. As a consequence of the linearity, the unique solution to \eqref{eq:linear} is of the form
\bal
u = \BBG_\mu[\tau] + \BBK_\mu[\nu] \quad \text{a.e. in } \Omega \setminus \Sigma.
\eal
Further results for linear problem \eqref{eq:linear} are presented in Subsection \ref{subsec:linear}.

Semilinear equations driven by $L_\mu$ with an absorption term have been treated in some particular cases of $\Sigma$. In the free-potential case, namely $\mu=0$ and $\Sigma=\emptyset$, the study of the boundary value problem for such equations in measure frameworks has been a research objective of  numerous mathematicians, and greatly pushed forward by a series of celebrated papers of Marcus and V\'eron (see the excellent monograph \cite{MVbook} and references therein). The singleton case, namely $\Sigma =\{0\} \subset \Omega$, has been investigated in different directions, including the work of Guerch and V\'eron \cite{GuV} on the local properties of solutions to the stationary Schr\"odinger equations in $\R^N$,  interesting results by C\^irstea \cite{Cir} on isolated singular solutions, and recent study of Chen and V\'eron \cite{CheVer} on the existence and stability of solutions with zero boundary condition.

In the present paper, we study the boundary value problem for semilinear equation with an absorption term of the form
\be\label{NLin} \left\{ \BAL
- L_\gm u+g(u)&=\tau\qquad \text{in }\;\Gw\setminus \Sigma,\\
\tr(u)&=\nu,
\EAL \right. \ee
where $\Sigma$ is of dimension $0 \leq k < N-2$, $g: \R\to \R$ is a nondecreasing continuous function such that $g(0)=0$, $\tau\in\GTM(\Omega \setminus \Sigma;\ei)$ and $\nu \in \GTM(\partial\Omega \cup \Sigma)$. A typical model of the absorption term to keep in mind is $g(t)=|t|^{p-1}t$ with $p>1$.

Problem \eqref{NLin} has the following features.
\begin{itemize}
\item The potential $d_\Sigma^{-2}$ blows up on $\Sigma$ and is bounded on $\partial \Omega$. Hence, considering $\partial \Omega \cup \Sigma$ simply as the `whole boundary' does not provide profound enough understanding of the effect of the potential. Therefore, we have to take care of $\partial \Omega$ and $\Sigma$ separably.
\item The dimension of  $\Sigma$, the value of the parameter $\mu$ and the concentration of the measures $\nu,\tau$ give rise to several critical exponents.
\item Heuristically, in measure framework, the growth of $g$ plays an important role in the solvability of \eqref{NLin}.
\end{itemize}
The complex interplay between the above features yields substantial difficulties and reveals new aspects of the study of \eqref{NLin}. We aim to perform a profound analysis of the interplay to establish the existence, nonexistence, uniqueness and a prior estimates for solutions to \eqref{NLin}.

%In this paper, we study the boundary value problem for the nonlinear equation with an absorption of the form
%\begin{equation} \label{NLE}
%	-L_\mu u + g(u) = 0 \quad \text{in } \Omega \setminus \Sigma,
%\end{equation}
%where $g: \R\to \R$ be a nondecreasing continuous function such that $g(0)=0$.

\subsection{Main results}
Let us assume throughout the paper that
\be \label{assump1}
\mu \leq H^2 \quad \text{and} \quad \lambda_\mu > 0.	
\ee
Under the above assumption, a theory for linear problem \eqref{eq:linear} was developed (see Subsection \ref{subsec:linear}), which forms a basis for the study of \eqref{NLin}.

Before stating our main results, we clarify the sense of solutions we will deal with in the paper.
\begin{definition} \label{def:weak-sol}
	 A function $u$ is a \textit{weak solution} of \eqref{NLin} if $u\in L^1(\Omega;\ei)$, $g(u) \in L^1(\Omega;\ei)$  and
	\be \label{nlinearweakform}
	- \int_{\xO}u L_{\xm }\zeta \, \dd x+ \int_{\xO}g(u)\zeta \, \dd x=\int_{\xO \setminus \Sigma} \zeta \, \dd \tau - \int_{\Gw} \mathbb{K}_{\xm}[\xn]L_{\xm }\zeta \, \dd x
	\qquad\forall \zeta \in\mathbf{X}_\xm(\xO\setminus \Sigma),
	\ee
where the \textit{space of test function} ${\bf X}_\mu(\Gw\setminus \Sigma)$ is defined by
\ba \label{Xmu} {\bf X}_\mu(\Gw\setminus \Sigma):=\{ \zeta \in H_{loc}^1(\Omega \setminus \Sigma): \phi_\mu^{-1} \zeta \in H^1(\Gw;\phi_\mu^{2}), \, \phi_\mu^{-1}L_\mu \zeta \in L^\infty(\Omega)  \}.
\ea
\end{definition}

%\begin{remark}
The space ${\bf X}_\mu(\Omega \setminus \Sigma)$ was introduced in \cite{GkiNg_linear} to study linear problem \eqref{eq:linear}. From \eqref{Xmu}, it is easy to see that the first term on the left-hand side of \eqref{nlinearweakform} is finite. By \cite[Lemma 7.3]{GkiNg_linear}, for any $\zeta \in {\bf X}_\mu(\Omega \setminus \Sigma)$, we have $|\zeta| \lesssim \phi_\mu$, hence the second term on the left-hand side  and the first term on the right-hand side of \eqref{nlinearweakform} are finite. Finally, since $\BBK_{\mu}[\nu] \in L^1(\Omega;\ei)$, the second term on the right-hand side of \eqref{nlinearweakform} is also finite. 	
%\end{remark}

By Theorem \ref{linear-problem}, $u$ is a weak solution of \eqref{NLin} if and only if
\bal u + \BBG_\mu[g(u)] = \BBG_\mu[\tau] + \BBK_\mu[\nu] \quad \text{in } \Omega \setminus \Sigma.
\eal

 \begin{definition} \label{goodmeasure}
A couple $(\tau,\nu) \in \GTM(\Omega \setminus \Sigma; \ei) \times \GTM(\partial \Omega \cup \Sigma)$ is called $g$-\textit{good couple} if problem \eqref{NLin} has a solution. When $\tau =0$, a measure $\nu \in \GTM(\partial \Omega \cup \Sigma)$ is called $g$-\textit{good measure} if problem \eqref{NLin} has a solution. When there is no confusion, we simply say `a good couple' (resp. `a good measure') instead of `a $g$-good couple' (resp. `a $g$-good measure').	
\end{definition}

Note that if $(\tau,\nu)$ is a good couple then the solution is unique.

Our first result provides a sufficient condition for a couple of measures to be good.
\begin{theorem} \label{existGK}
Assume $\mu \leq H^2$ and $g$ satisfies
	\bel{g(GK)} g(-\BBG_\mu[\tau^-] - \BBK_{\mu}[\nu^-]),  g(\BBG_\mu[\tau^+] + \BBK_{\mu}[\nu^+]) \in L^1(\Omega;\ei).
	\ee
Then any couple $(\tau,\nu) \in \GTM(\Omega \setminus \Sigma; \ei) \times \GTM(\partial \Omega \cup \Sigma)$ is a $g$-good couple. Moreover, the solution $u$ satisfies
	\be \label{U12}
	-\BBG_\mu[\tau^-] - \BBK_{\mu}[\nu^-] \leq u \leq \BBG_\mu[\tau^+] + \BBK_{\mu}[\nu^+] \quad \text{in } \Omega \setminus \Sigma.
	\ee
\end{theorem}

The existence part of Theorem \ref{existGK} is based on sharp weak Lebesgue estimates on the Green kernel and Martin kernel (Theorems \ref{lpweakgreen}--\ref{lpweakmartin1}) and the sub and super solution theorem (see Theorem \ref{existencesubcr}). The uniqueness is derived from Kato inequalities (see Theorem \ref{linear-problem}).

When $g$ satisfies the so-called \textit{subcritical integral condition}
\be \label{subcricondintro} \int_1^\infty  s^{-q-1} (g(s)-g(-s)) \, \dd s<\infty
\ee
for suitable $q>0$, we can show that condition \eqref{g(GK)} holds (see Lemma \ref{subcrcon}) and consequently, $(\tau,\nu)$ is a good couple.
\begin{theorem} \label{exist-subGK} Assume $\mu <(\frac{N-2}{2})^2$ and $g$ satisfies \eqref{subcricondintro} with
\bal
q=\min\left\{\frac{N+1}{N-1},\frac{N-\am}{N-\am-2}\right\},
\eal
where $\am$ is defined in \eqref{apm}.
Then any couple $(\tau,\nu) \in \GTM(\Omega \setminus \Sigma; \ei) \times \GTM(\partial \Omega \cup \Sigma)$ is a $g$-good couple. Moreover, the solution $u$ satisfies \eqref{U12}.
\end{theorem}

The value of $q$ in condition \eqref{subcricondintro} under which problem \eqref{NLin} with $\tau=0$, namely problem
\be \label{pro:g-tau=0} \left\{ \BAL -L_\mu u + g(u)  &= 0 \quad \text{in } \Gw\setminus \Sigma, \\
\tr(u) &= \xn,
\EAL \right. \ee
admits a unique solution, can be enlarged according to the concentration of the boundary measure data. The case when $\nu$ is concentrated in $\partial \Omega$ is treated in the following theorem.

\begin{theorem} \label{measureO} Assume $\mu \leq H^2$ and $g$ satisfies \eqref{subcricondintro} with $q=\frac{N+1}{N-1}$. Then any measure $\nu \in \GTM(\partial \Omega \cup \Sigma)$ with compact support in $\partial \Omega$ is a $g$-good measure. Moreover, the solution $u$ satisfies
	\be \label{K<u<K}
	- \BBK_{\mu}[\nu^-] \leq u \leq \BBK_{\mu}[\nu^+] \quad \text{in } \Omega \setminus \Sigma.
	\ee	 	
\end{theorem}

%\begin{remark}
%	In comparison to Theorem \ref{exist-subGK}, we see that the exponent in \eqref{subcricondintro} is $\frac{N+1}{N-1}$ which is greater than $p_{\am}$ due to the fact that $\nu$ is concentrated in $\partial \Omega$.
%\end{remark}

It is worth mentioning that, without requiring condition \eqref{subcricondintro}, one can show that any $L^1$ datum concentrated in $\partial \Omega$ is $g$-good. (see Theorem \ref{solL1} for more detail).

When $\nu$ is concentrated in $\Sigma$, it is $g$-good  under the condition \eqref{subcricondintro} with $q=\frac{N-\am}{N-\am-2}$ if $\mu <\left( \frac{N-2}{2} \right)^2.$ However, if $k=0$ and $\mu = \left( \frac{N-2}{2} \right)^2$, which implies that $\am=\frac{N-2}{2},$ condition \eqref{subcricondintro} with $q=\frac{N+2}{N-2}$ is not enough to ensure that $\xn$ is $g$-good. In this case we need to impose a slightly stronger condition on $g$.  This is stated in the following theorem.

\begin{theorem} \label{exist-measureK} ~~
	
(i) Assume $\mu < \left( \frac{N-2}{2} \right)^2$ and $g$ satisfies \eqref{subcricondintro} with $q=\frac{N-\am}{N-\am-2}$. Then any measure $\nu \in \GTM(\partial \Omega \cup \Sigma)$ with compact support in $\Sigma$ is a $g$-good measure. Moreover, the solution $u$ satisfies \eqref{K<u<K}.
	
(ii) Assume $k=0$, $\Sigma=\{0\}$, $\mu = \left( \frac{N-2}{2} \right)^2$ and $g$ satisfies
	\bel{subcriticalg-ln} \int_1^\infty s^{-\frac{N+2}{N-2}-1}(\ln s)^{\frac{N+2}{N-2}}g(s) \, \dd s < \infty.
	\ee
	Then for any $\xr>0$, $\nu = \xr \delta_0$ is $g$-good. Here $\delta_0$ is the Dirac measure concentrated at $0$.
\end{theorem}

When $g$ is a power function, namely $g(t) = |t|^{p-1}t$ with $p>1$, condition \eqref{subcricondintro} with $q=\frac{N+1}{N-1}$ is fulfilled if and only if $1<p<\frac{N+1}{N-1}$, while condition \eqref{subcricondintro} with $q=\frac{N-\am}{N-\am-2}$ is satisfied if and only if $1<p<\frac{N-\am}{N-\am-2}$. In these ranges of $p$, by Theorem \ref{measureO} and Theorem \ref{exist-measureK}, problem \eqref{pro:g-tau=0} admits a unique solution. In particular,  in these ranges of $p$, existence results hold when $\nu$ is a Dirac measure. We will point out below that in case $p \geq \frac{N+1}{N-1}$ or $p \geq \frac{N-\am}{N-\am-2}$ according to the concentration of the boundary data, isolated singularities are removable. This justifies the fact that the values  $\frac{N+1}{N-1}$ and $\frac{N-\am}{N-\am-2}$ are \textit{critical exponents}.

To this purpose, we introduce a weight function which allows to normalize the value of solutions near $\Sigma$. Let $\beta_0$ be the constant in Subsection \ref{assumptionK} and $\eta_{\beta_0}$ be a smooth function such that $0 \leq \eta_{\beta_0} \leq 1$, $\eta_{\beta_0}=1$ in $\overline{\Sigma}_{\frac{\xb_0}{4}}$ and $\supp \eta_{\beta_0} \subset \Sigma_{\frac{\beta_0}{2}}$. We define
\bal W(x):=\left\{ \BAL &d_\Sigma(x)^{-\ap}\qquad&&\text{if}\;\mu <H^2, \\
&d_\Sigma(x)^{-H}|\ln d_\Sigma(x)|\qquad&&\text{if}\;\mu =H^2,
\EAL \right. \quad x \in \Omega \setminus \Sigma,
\eal
and
\bel{tildeW}
\tilde W:=1-\eta_{\beta_0}+\eta_{\beta_0}W \quad \text{in } \Omega \setminus \Sigma.
\ee

It was proved in \cite{GkiNg_linear} that for any $h\in C(\partial\xO\cup\xS)$, the problem
\be \label{linearintro} \left\{ \BAL
L_{\mu}v&=0\qquad \text{in}\;\;\xO\setminus \Sigma\\
v&=h\qquad \text{on}\;\;\partial\xO\cup \Sigma,
\EAL \right. \ee
admits a unique solution $v$. Here the boundary value condition in \eqref{linearintro} is understood as
\bal
\lim_{x \in \Omega \setminus \Sigma,\; x \to y }\frac{v(x)}{\tilde W(x)}=h(y) \quad \text{uniformly w.r.t. } y \in \partial \Omega \cup \Sigma.
\eal

%\begin{theorem} \label{isolatedK}
%	Assume either $0 \in \Sigma$ and $1< p < \frac{N-\am}{N-\am-2}$, or either $0 \in \Sigma$ and $1< p < \frac{N+1}{N-1}$. For $\varrho > 0$, let $u_{0,\varrho}$ be the solution of
%	\bel{Pc} \left\{ \BAL -L_\mu u + \abs{ u}^{p-1}u  &= 0 \quad \text{in } \Gw\setminus \Sigma \\
%	\tr(u) &= \varrho\gd_0.
%	\EAL \right. \ee
%	Then
%	\bel{BA.1} \lim_{|x| \to 0}\frac{u_{0,\varrho}(x)}{K_{\mu}^\Gw(x,0)}=\varrho. \ee
%	Furthermore the mapping $\varrho \mapsto u_{0,\varrho}$ is increasing.
%\end{theorem}

%The behavior of the limit $u_{0,\infty}:=\lim_{\varrho \to \infty}u_{0,\varrho}$ is investigated in Section \ref{Sec:isolated}.

\begin{theorem} \label{remov-1}
Assume $\mu \leq H^2$ and $p\geq \frac{2+\ap}{\ap}$. If $u \in C(\overline \Omega \setminus \Sigma)$ is a nonnegative solution of
	\be \label{eq:power-a}
	-L_\mu u+|u|^{p-1}u=0 \quad \text{in } \Omega \setminus \Sigma
	\ee
in the sense of distributions in $\Omega \setminus \Sigma$	such that
	\be\label{as1intro}
	\lim_{x\in\xO\setminus\xS,\;x\rightarrow\xi}\frac{u(x)}{\tilde W(x)}=0\qquad\forall \xi\in \partial\xO,
	\ee
	locally uniformly in $\partial\xO$, then $u\equiv 0.$
\end{theorem}

The idea of the proof of Theorem \ref{remov-1} is to construct a function $v$ dominating $u$ by using to the Keller-Osserman type estimate (see Lemma \ref{lem:KO-1}). Then, by making use of the Representation Theorem \ref{th:Rep} and a subtle argument based on the maximum principle, we are able to deduce $v \equiv 0$, which implies $u \equiv 0$.

When $\frac{N-\am}{N-\am-2}\leq p< \frac{2+\ap}{\ap}$, an additional condition on the behavior of solutions near $\Sigma$ is required to obtain a removability result.
\begin{theorem} \label{remove-2}
Assume $\mu \leq H^2$, $z\in \Sigma$ and $\frac{N-\am}{N-\am-2}\leq p< \frac{2+\ap}{\ap}$. If $u\in C(\xO\setminus \Sigma)$ is a nonnegative solution of  \eqref{eq:power-a} in the sense of distributions in $\xO\setminus \Sigma$ such that
	\be\label{as2intro}
	\lim_{x\in\xO\setminus\xS,\;x\rightarrow\xi}\frac{u(x)}{\tilde W(x)}=0\qquad\forall \xi\in \partial\xO\cup \Sigma\setminus \{z\},
	\ee
	locally uniformly in $\partial\xO\cup \Sigma\setminus \{z\}$, then $u\equiv 0.$
\end{theorem}

The technique used in the proof of Theorem \ref{remove-2} is different from that of Theorem \ref{remov-1}. In the range $\frac{N-\am}{N-\am-2}\leq p< \frac{2+\ap}{\ap}$, by employing appropriate test functions and Keller-Osserman type estimate (see Lemma \ref{lem:KO-1}), we can show that the solution $u$, which may admit an isolated singularity at $z$, belongs to $L^p(\Omega)$. Then by using a delicate argument based on the properties of the boundary trace, we assert that $u$ cannot have positive mass at $z$, which implies that the isolated singularity is removable and hence $u \equiv 0$.

Next, we introduce an appropriate capacity framework which enables us to obtain the solvability for
\be\label{problempower-1} \left\{ \BAL -L_\mu u + \abs{ u}^{p-1}u  &= 0 \quad \text{in } \Gw\setminus \Sigma \\
\tr(u) &= \xn.
\EAL \right. \ee
A measure $\nu \in \GTM(\partial \Omega \cup \Sigma)$ for which problem \eqref{problempower-1} admits a (unique) solution is called \textit{$p$-good measure}.

For $\ga\in\BBR$ we defined the Bessel kernel of order $\ga$ by  $ \CB_{d,\ga}(\xi):=\CF^{-1}\left((1+|.|^2)^{-\frac{\ga}{2}} \right)(\xi)$, where $\CF$ is the Fourier transform in space $\CS'(\R^d)$ of moderate distributions in $\BBR^d$. For $\kappa>1$, the Bessel space $L_{\ga,\kappa}(\BBR^d)$ is defined by
\bal
L_{\ga,\kappa}(\BBR^d):=\{f=\CB_{d,\alpha}\ast g:g\in L^{\kappa}(\BBR^d)\},
\eal
with norm
\bal
\|f\|_{L_{\ga,\kappa}}:=\|g\|_{L^\kappa}=\|\CB_{d,-\ga}\ast f\|_{L^\kappa}.
\eal
The Bessel capacity $\mathrm{Cap}^{\BBR^d}_{\ga,\kappa}$ is defined for compact subsets
$K \subset\BBR^d$ by
\bal
\mathrm{Cap}^{\BBR^d}_{\ga,\kappa}(K):=\inf\{\|f\|^\kappa_{L_{\ga,\kappa}}, f\in\CS'(\BBR^d),\,f\geq \1_K \}.
\eal
See Section \ref{sec:goodmeasure} for further discussion on the Bessel spaces and capacities.
\begin{definition} \label{abscont}
	Let $\nu \in \GTM^+(\partial \Omega \cup \Sigma)$. We say that $\nu$ is \textit{absolutely continuous} with respect to the Bessel capacity $\mathrm{Cap}_{\alpha,\kappa}^{\R^d}$ if 	
	\bal\BA {lll}
	\forall E \subset \partial \Omega \cup \Sigma,\,E\text{ Borel}, \mathrm{Cap}^{\BBR^{d}}_{\alpha,\kappa}(E)=0\Longrightarrow \xn(E)=0.
	\EA\eal
\end{definition}

When $\frac{N-\am}{N-\am-2}\leq p<\frac{2+\ap}{\ap}$ anf $\nu$ is concentrated in $\Sigma$, a sufficient condition expressed in terms of a suitable Bessel capacity for a measure to be $p$-good is provided in the next theorem.
\begin{theorem}\label{supcrK} Assume $k \geq 1$, $\mu \leq H^2$, $\frac{N-\am}{N-\am-2}\leq p<\frac{2+\ap}{\ap}$ and $\xn\in\mathfrak M^+(\partial\Gw\cup \Sigma)$ with compact support in $\Sigma$. Put
\be \label{gamma} \vartheta: = \frac{2-(p-1)\ap}{p}.
\ee
If $\nu$ is absolutely continuous with respect to  $\mathrm{Cap}^{\BBR^{k}}_{\vartheta,p'}$, where $p'=\frac{p}{p-1}$, then $\nu$ is $p$-good.
\end{theorem}

A pivotal ingredient in the proof of Theorem \ref{supcrK} is a sophisticated potential estimate on the Martin kernel (see Theorem \ref{potest}) inspired by \cite{MV-Pisa}, which allows us to implement an approximation procedure to derive the existence of a solution to \eqref{problempower-1}.

In case $p \geq \frac{N+1}{N-1}$ and $\nu$ is concentrated in $\partial \Omega$, we show that  the absolute continuity of $\nu$ with respect to a suitable Bessel capacity is not only a sufficient condition, but also a necessary condition for $\nu$ to be $p$-good.
\begin{theorem}\label{supcromega} Assume $\mu \leq H^2$, $p \geq \frac{N+1}{N-1}$ and $\xn\in\mathfrak M^+(\partial\Gw\cup \Sigma)$ with compact support in $\partial \xO$. Then $\nu$ is a $p$-good measure  if and only if it is absolutely continuous with respect to  $\mathrm{Cap}^{\BBR^{N-1}}_{\frac{2}{p},p'}$.
\end{theorem} \medskip

\noindent \textbf{Organization of the paper.} In Section \ref{pre}, we present main properties of the submanifold $\Sigma$ and recall important facts about the first eigenpair, Green kernel and Martin kernel of $-L_\mu$. In Section \ref{sec:BVP}, we prove the sub and super solution theorem (see Theorem \ref{existencesubcr}), which is an important tool in the prove of Theorem \ref{existGK} and Theorem \ref{exist-subGK}. Section \ref{sec:BVP-partialO} and Section \ref{sec:BVP-Sigma} are devoted to the proof of Theorem \ref{measureO} and Theorem \ref{exist-measureK} respectively. Next we establish Keller-Osserman estimates in Section \ref{sec:KOestimate}, which is a crucial ingredient in the proof of Theorem \ref{remov-1} and Theorem \ref{remove-2} in Section \ref{sec:removable}. Then we provide the proof of Theorems \ref{supcrK}--\ref{supcromega} in Section \ref{sec:goodmeasure}. Finally, in Appendix, we construct a barrier function and demonstrate some useful estimates.  \medskip

\subsection{Notations} \label{subsec:notations} We list below notations that are frequently used in the paper.

$\bullet$ Let $\phi$ be a positive continuous function in $\Omega \setminus \Sigma$ and $\kappa \geq 1$. Let $L^\kappa(\Omega;\phi)$ be the space of functions $f$ such that
\bal \| f \|_{L^\kappa(\Omega;\phi)} := \left( \int_{\Omega} |f|^\kappa \phi \, \dd x \right)^{\frac{1}{\kappa}}.
\eal

The weighted Sobolev space $H^1(\Omega;\phi)$ is the space of functions $f \in L^2(\Omega;\phi)$ such that $\nabla f \in L^2(\Omega;\phi)$. This space is endowed with the norm
$$ \| f \|_{H^1(\Omega;\phi)}^2= \int_{\Omega} |f|^2 \phi \,\dd x +  \int_{\Omega} |\nabla f|^2 \phi \,\dd x.
$$
The closure of $C_c^\infty(\Omega)$ in $H^1(\Omega;\phi)$ is denoted by $H_0^1(\Omega;\phi)$.

Denote by $\mathfrak{M}(\Omega;\phi)$ the space of Radon measures $\tau$ in $\Omega$ such that \bal \| \tau\|_{\mathfrak{M}(\Omega;\phi)}:=\int_{\Omega}\phi \, \dd|\tau|<\infty,
\eal
and by $\mathfrak{M}^+(\Omega;\phi)$ its positive cone. Denote by $\GTM(\partial \Omega \cup \Sigma)$ the space of finite measure $\nu$ on $\partial \Omega \cup \Sigma$, namely
$$ \| \nu \|_{\GTM(\partial \Omega \cup \Sigma)}:=|\nu|(\partial \Omega \cup \Sigma) < \infty,
$$
and by $\GTM^+(\partial \Omega \cup \Sigma)$ its positive cone.

$\bullet$ For a measure $\omega$, denote by $\omega^+$ and $\omega^-$ the positive part and negative part of $\omega$ respectively.

$\bullet$ For $\beta>0$, let $ \Omega_{\beta}=\{ x \in \Omega: d(x) < \beta\}$ and $\Sigma_{\beta}=\{ x \in \R^N \setminus \Sigma:  d_\Sigma(x)<\beta \}$.

$\bullet$ We denote by $c,c_1,C...$ the constant which depend on initial parameters and may change from one appearance to another.

$\bullet$ The notation $A \gtrsim B$ (resp. $A \lesssim B$) means $A \geq c\,B$ (resp. $A \leq c\,B$) where the implicit $c$ is a positive constant depending on some initial parameters. If $A \gtrsim B$ and $A \lesssim B$, we write $A \approx B$. \textit{Throughout the paper, most of the implicit constants depend on some (or all) of the initial parameters such as $N,\Omega,\Sigma,k,\mu$ and we will omit these dependencies in the notations (except when it is necessary).}

$\bullet$ For $a,b \in \BBR$, denote $a \wedge b = \min\{a,b\}$, $a \lor b =\max\{a,b \}$.

$\bullet$ For a set $D \subset \R^N$, $\1_D$ denotes the indicator function of $D$.

\medskip
\noindent \textbf{Acknowledgement.} K. T. Gkikas acknowledges support by the Hellenic Foundation for Research and Innovation
(H.F.R.I.) under the “2nd Call for H.F.R.I. Research Projects to support Post-Doctoral Researchers” (Project
Number: 59). P.-T. Nguyen was supported by Czech Science Foundation, Project GA22-17403S.
%%%%%%%%%%%%%%%%%%%%%%%%%%%%%%%%%%%%%%%%%%%%%%
%%%%%%%%%%%%%%%%%%%%%%%%%%%%%%%%%%%%%%%%%%%%%%
\section{Preliminaries} \label{pre}

\subsection{Assumptions on $\Sigma$.} \label{assumptionK} Throughout this paper, we assume that $\Sigma \subset \Omega$ is a $C^2$ compact submanifold in $\mathbb{R}^N$ without boundary, of dimension $k$, $0\leq k < N-2$. When $k = 0$ we assume that $\Sigma = \{0\}$.

For $x=(x_1,...,x_k,x_{k+1},...,x_N) \in \R^N$, we write $x=(x',x'')$ where $x'=(x_1,..,x_k) \in \R^k$ and $x''=(x_{k+1},...,x_N) \in \R^{N-k}$. For $\beta>0$, we denote by $B^k(x',\beta)$ the ball  in $\R^k$ with center at $x'$ and radius $\beta.$ For any $\xi\in \Sigma$, we set
\begin{align} \nonumber  \Sigma_\beta &:=\{ x \in \R^N \setminus \Sigma: d_\Sigma(x) < \beta \}, \\
\label{Vxi}
V(\xi,\xb)&:=\{x=(x',x''): |x'-\xi'|<\beta,\; |x_i-\Gamma_i^\xi(x')|<\xb,\;\forall i=k+1,...,N\},
\end{align}
for some functions $\Gamma_i^\xi: \R^k \to \R$, $i=k+1,...,N$.

Since $\Sigma$ is a $C^2$ compact submanifold in $\mathbb{R}^N$ without boundary, we may assume the existence of $\xb_0$ such that the followings hold.

\begin{itemize}
\item $\Sigma_{6\beta_0}\Subset \Omega$ and for any $x\in \Sigma_{6\beta_0}$, there is a unique $\xi \in \Sigma$  satisfies $|x-\xi|=d_\Sigma(x)$.

\item $d_\Sigma \in C^2(\Sigma_{4\beta_0})$, $|\nabla d_\Sigma|=1$ in $\Sigma_{4\beta_0}$ and there exists $\eta \in L^\infty(\Sigma_{4\beta_0})$ such that
\be \label{laplaciand}
\Delta d_\Sigma(x)=\frac{N-k-1}{d_\Sigma(x)}+ \eta(x) \quad \text{in } \Sigma_{4\beta_0} .
\ee
(See \cite[Lemma 2.2]{Vbook} and \cite[Lemma 6.2]{DN}.)

\item For any $\xi \in \Sigma$, there exist $C^2$ functions $\Gamma_i^\xi \in C^2(\R^k;\R)$, $i=k+1,...,N$, such that (upon relabeling and reorienting the coordinate axes if necessary), for any $\beta \in (0,6\beta_0)$, $V(\xi,\beta) \subset \Omega$ and
\be \label{straigh}
V(\xi,\beta) \cap \Sigma=\{x=(x',x''): |x'-\xi'|<\beta,\;  x_i=\Gamma_i^\xi (x'), \; \forall i=k+1,...,N\}.
\ee

\item There exist $m_0 \in \N$ and points $\xi^{j} \in \Sigma$, $j=1,...,m_0$, and $\beta_1 \in (0, \beta_0)$ such that
\be \label{cover}
\Sigma_{2\xb_1}\subset \cup_{j=1}^{m_0} V(\xi^j,\beta_0)\Subset \Omega.
\ee
\end{itemize}

Now for $\xi \in \Sigma$, set
\bel{dist2} \xd_\Sigma^\xi(x):=\left(\sum_{i=k+1}^N|x_i-\Gamma_i^\xi(x')|^2\right)^{\frac{1}{2}}, \qquad x=(x',x'')\in V(\xi,4\beta_0).\ee

Then we see that there exists a constant $C=C(N,\Sigma)$ such that
\be\label{propdist}
d_\Sigma(x)\leq	\xd_\Sigma^{\xi}(x)\leq C \| \Sigma \|_{C^2} d_\Sigma(x),\quad \forall x\in V(\xi,2\beta_0),
\ee
where $\xi^j=((\xi^j)', (\xi^j)'') \in \Sigma$, $j=1,...,m_0$, are the points in \eqref{cover} and
\be \label{supGamma}
\| \Sigma \|_{C^2}:=\sup\{  || \Gamma_i^{\xi^j} ||_{C^2(B_{5\beta_0}^k((\xi^j)'))}: \; i=k+1,...,N, \;j=1,...,m_0 \} < \infty.
\ee
Moreover, $\beta_1$ can be chosen small enough such that for any $x \in \Sigma_{\beta_1}$,
\bel{BinV} B(x,\beta_1) \subset V(\xi,\beta_0),
\ee
where $\xi \in \Sigma$ satisfies $|x-\xi|=d_\Sigma(x)$.

\subsection{Eigenvalue of $-L_\mu$} \label{subsec:eigen} Let $H$ be defined in \eqref{H} and $\am$ and $\ap$ be defined in \eqref{apm}.
We summarize below main properties of the first eigenfunction of the operator $-L_\mu$ in $\Omega \setminus \Sigma$ from \cite[Lemma 2.4 and Theorem 2.6]{DD1} and \cite[page 337, Lemma 7, Theorem 5]{DD2}.

(i) For any $\mu \leq H^2$, it is known that
\be\label{Lin01} \lambda_\mu:=\inf\left\{\int_{\Gw}\left(|\nabla u|^2-\frac{\xm }{d_\Sigma^2}u^2\right)\dd x: u \in H_c^1(\Omega), \int_{\Gw} u^2 \dd x=1\right\}>-\infty.
\ee

\smallskip

(ii) If $\mu < H^2$, there exists a minimizer $\gf_{\xm }$ of \eqref{Lin01} belonging to $H^1_0(\Gw)$. Moreover, it satisfies $-L_\mu \phi_\mu= \lambda_\mu \phi_\mu$  in $\Omega \setminus \Sigma$ and $\phi_{\mu }\approx d_\Sigma^{-\am}$ in $\Sigma_{\beta_0}$.

\smallskip

(iii) If $\xm =H^2$, there is no minimizer of \eqref{Lin01} in $H_0^1(\Gw)$, but there exists a nonnegative function $\phi_{H^2}\in H_{loc}^1(\xO)$  such that $-L_{H^2}\phi_{H^2}=\lambda_{H^2}\phi_{H^2}$ in the sense of distributions in $\Omega \setminus \Sigma$ and
$\phi_{H^2}\approx d_\Sigma^{-H}$ in $\Sigma_{\beta_0}$. In addition, the function $d_\Sigma^{-H}\xf_{H^2} \in H^1_0(\Gw; d_\Sigma^{-2H})$.

From (ii) and (iii) we deduce that
\be \label{eigenfunctionestimates}
\xf_\xm \approx d\,d^{-\am}_\Sigma \quad \text{in } \Omega \setminus \Sigma.
\ee

\subsection{Estimates on Green kernel and Martin kernel} \label{subsec:GreenMartin}
Recall that throughout the paper, we always assume that \eqref{assump1} holds. Let $G_\mu$ and $K_{\mu}$ be the Green kernel and Martin kernel of $-L_\mu$ in $\Omega \setminus \Sigma$ respectively.
Let us recall two-sided estimates on Green kernel.

\begin{proposition}[{\cite[Proposition 4.1]{GkiNg_linear}}]  \label{Greenkernel} ~~
	
	(i) If $\mu< \left( \frac{N-2}{2}\right)^2$ then for any $x,y \in \Omega \setminus \Sigma$, $x \neq y$,
	\bel{Greenesta} \BAL
	G_{\mu}(x,y)&\approx |x-y|^{2-N} \left(1 \wedge \frac{d(x)d(y)}{|x-y|^2}\right)  \left(\frac{|x-y|}{d_\Sigma(x)}+1\right)^\am
	\left(\frac{|x-y|}{d_\Sigma(y)}+1\right)^\am \\
	&\approx |x-y|^{2-N} \left(1 \wedge \frac{d(x)d(y)}{|x-y|^2}\right) \left(1 \wedge \frac{d_\Sigma(x)d_\Sigma(y)}{|x-y|^2} \right)^{-\am}.
	\EAL \ee
	
	(ii) If $k=0$, $\Sigma=\{0\}$ and $\mu = \left( \frac{N-2}{2}\right)^2$ then for any $x,y \in \Omega \setminus \Sigma$, $x \neq y$,
	\bel{Greenestb} \BAL
	G_{\mu}(x,y) &\approx |x-y|^{2-N} \left(1 \wedge \frac{d(x)d(y)}{|x-y|^2}\right) \left(\frac{|x-y|}{|x|}+1\right)^{\frac{N-2}{2}}
	\left(\frac{|x-y|}{|y|}+1\right)^{\frac{N-2}{2}}\\
	& \quad +(|x||y|)^{-\frac{N-2}{2}}\left|\ln\left(1 \wedge \frac{|x-y|^2}{d(x)d(y)}\right)\right| \\
	&\approx |x-y|^{2-N} \left(1 \wedge \frac{d(x)d(y)}{|x-y|^2}\right) \left(1 \wedge \frac{|x||y|}{|x-y|^2} \right)^{-\frac{N-2}{2}}\\
	& \quad +(|x||y|)^{-\frac{N-2}{2}}\left|\ln\left(1 \wedge \frac{|x-y|^2}{d(x)d(y)}\right)\right|.
	\EAL \ee
	
	The implicit constants in \eqref{Greenesta} and \eqref{Greenestb} depend on $N,\Omega,\Sigma,\mu$.
\end{proposition}

\begin{proposition}[{\cite[Theorem 1.2]{GkiNg_linear}}] \label{Martin} ~~
	
	(i) If $\mu< \left( \frac{N-2}{2}\right)^2$ then
	\be \label{Martinest1}
	K_{\mu}(x,\xi) \approx\left\{
	\BAL
	&\frac{d(x)d_\Sigma(x)^{-\am}}{|x-\xi|^N}\quad &&\text{if } x \in \Omega \setminus \Sigma,\;  \xi \in \partial\xO \\
	&\frac{d(x)d_\Sigma(x)^{-\am}}{|x-\xi|^{N-2-2\am}} &&\text{if } x \in \Omega \setminus \Sigma,\; \xi \in \Sigma.
	\EAL \right.
	\ee
	
	(ii) If  $k=0$, $\Sigma=\{0\}$ and $\mu= \left( \frac{N-2}{2}\right)^2$ then
	\be\label{Martinest2}
	K_{\mu}(x,\xi) \approx\left\{
	\BAL
	&\frac{d(x)|x|^{-\frac{N-2}{2}}}{|x-\xi|^N} \quad &&\text{if } x \in \Omega \setminus \{0\},\; \xi \in \partial\xO \\
	&d(x)|x|^{-\frac{N-2}{2}}\left|\ln\frac{|x|}{\CD_\Omega}\right| &&\text{if } x \in \Omega \setminus \{0\},\; \xi=0,
	\EAL \right.
	\ee
	where $\CD_\Omega:=2\sup_{x \in \Omega}|x|$.
	
	The implicit constant depends on $N,\Omega,\Sigma,\mu,p$.
\end{proposition}

The Green operator and Martin operator are respectively
\ba \label{Grt}
\BBG_\mu[\tau](x)=\int_{\xO \setminus \Sigma}G_{\mu}(x,y) \, \dd\tau(y), \quad \tau \in \GTM(\Omega \setminus \Sigma;\phi_\mu), \\
\label{Mrt} \mathbb{K}_\mu[\gn](x)=\int_{\partial\xO \cup \Sigma}K_{\mu}(x,y) \, \dd\xn(y), \quad \gn\in \mathfrak{M}(\partial\xO\cup \Sigma).
\ea

Next we recall the Representation theorem.
\begin{theorem}[{\cite[Theorem 1.3]{GkiNg_linear}}] \label{th:Rep} For any $\nu \in \GTM^+(\partial \Omega \cup \Sigma)$, the function $\BBK_{\mu}[\nu]$ is a positive $L_\mu$-harmonic function (i.e. $L_\mu \BBK_{\mu}[\nu]=0$ in the sense of distributions in $\Omega \setminus \Sigma$). Conversely, for any positive $L_\mu$-harmonic function $u$ (i.e. $L_\mu u = 0$ in the sense of distribution in $\Omega \setminus \Sigma$), there exists a unique measure $\nu \in \GTM^+(\partial \Omega \cup \Sigma)$ such that $u=\BBK_{\mu}[\nu]$.
\end{theorem}

%%%%%%%%%%%%%%%%%%%%%%%%%%%%%%%%%%%%%%%%%%%%%%%%%%
%%%%%%%%%%%%%%%%%%%%%%%%%%%%%%%%%%%%%%%%%%%%%%%%%%

\subsection{Notion of boundary trace} \label{subsec:boundarytrace}

Let $z \in \Omega \setminus \Sigma$ and $h\in C(\partial\Omega \cup \Sigma)$ and denote $\CL_{\mu ,z}(h):=v_h(z)$ where $v_h$ is the unique solution of the Dirichlet problem
\be \label{linear} \left\{ \BAL
L_{\mu}v&=0\qquad \text{in}\;\;\xO\setminus \Sigma\\
v&=h\qquad \text{on}\;\;\partial\xO\cup \Sigma.
\EAL \right. \ee
Here the boundary value condition in \eqref{linear} is understood in the sense that
\bal
\lim_{\dist(x,F)\to 0}\frac{v(x)}{\tilde W(x)}=h \quad \text{for every compact set } \; F\subset \partial \Omega \cup \Sigma.
\eal
The mapping $h\mapsto \CL_{\mu,z}(h)$ is a linear positive functional on $C(\partial\Omega \cup \Sigma)$. Thus there exists a unique Borel measure on $\partial\Omega \cup \Sigma$, called {\it $L_{\mu}$-harmonic measure in $\partial \Omega \cup \Sigma$ relative to $z$} and  denoted by $\omega_{\Omega \setminus \Sigma}^{z}$, such that
\bal
v_{h}(z)=\int_{\partial\Omega\cup \Sigma}h(y) \, \dd\omega_{\Omega \setminus \Sigma}^{z}(y).
\eal
Let $x_0 \in \Omega \setminus \Sigma$ be a fixed reference point. Let $\{\xO_n\}$ be an increasing sequence of bounded $C^2$ domains  such that
\ba\label{Omegan}  \overline{\xO}_n\subset \xO_{n+1}, \quad \cup_n\xO_n=\xO, \quad \mathcal{H}^{N-1}(\partial \Omega_n)\to \mathcal{H}^{N-1}(\partial \Omega),
\ea
where $\mathcal{H}^{N-1}$ denotes the $(N-1)$-dimensional Hausdorff measure in $\R^N$.
Let $\{\Sigma_n\}$ be a decreasing sequence of bounded $C^2$ domains  such that
\ba\label{Kn} \Sigma \subset \Sigma_{n+1}\subset\overline{\Sigma}_{n+1}\subset \Sigma_{n}\subset\overline{\Sigma}_{n} \subset\Omega_n, \quad \cap_n \Sigma_n=\Sigma.
\ea
For each $n$, set $O_n=\xO_n\setminus \Sigma_n$  and assume that $x_0 \in O_1$. Such a sequence $\{O_n\}$ will be called a {\it $C^2$ exhaustion} of $\Gw\setminus \Sigma$.

Then $-L_\mu$ is uniformly elliptic and coercive in $H^1_0(O_n)$ and its first eigenvalue $\lambda_\mu^{O_n}$ in $O_n$ is larger than its first eigenvalue $\lambda_\mu$ in $\Omega \setminus \Sigma$.

For $h\in C(\prt O_n)$, the following problem
\be\label{sub12} \left\{ \BAL
-L_{\xm } v&=0\qquad&&\text{in } O_n\\
v&=h\qquad&&\text{on } \prt O_n,
\EAL \right.
\ee
admits a unique solution which allows to define the $L_{\xm }$-harmonic measure $\omega_{O_n}^{x_0}$ on $\prt O_n$
by
\be\label{redu2}
v(x_0)=\myint{\prt O_n}{}h(y) \,\dd\gw^{x_0}_{O_n}(y).
\ee

Let $G^{O_n}_\xm(x,y)$ be the Green kernel of $-L_\mu$ on $O_n$.  Then $G^{O_n}_\xm(x,y)\uparrow G_\mu(x,y)$ for $x,y\in\xO\setminus \xS, x \neq y$.

We recall below the definition of boundary trace which is defined in a \textit{dynamic way}.

\begin{definition}[Boundary trace] \label{nomtrace}
	A function $u\in W^{1,\kappa}_{loc}(\xO\setminus\xS)$ for some $\kappa>1,$ possesses a \emph{boundary trace}  if there exists a measure $\nu \in\GTM(\partial \Omega \cup \Sigma)$ such that for any $C^2$ exhaustion  $\{ O_n \}$ of $\Omega \setminus \Sigma$, there  holds
	\be\label{trab}
	\lim_{n\rightarrow\infty}\int_{ \partial O_n}\phi u\, \dd \omega_{O_n}^{x_0}=\int_{\partial \Omega \cup \Sigma} \phi \,\dd \nu \quad\forall \phi \in C(\overline{\Omega}).
	\ee
	The boundary trace of $u$ is denoted by $\tr(u)$.
\end{definition}

\begin{proposition}[Proposition 1.8 in \cite{GkiNg_linear}] \label{traceKG} ~~
	
	(i) For any $\nu \in \GTM(\partial \Omega \cup \Sigma)$, $\tr(\BBK_{\mu}[\nu])=\nu$.
	
	(ii) For any $\tau \in \GTM(\Omega \setminus \Sigma;\ei)$, $\tr(\BBG_\mu[\tau])=0$.
\end{proposition}

\subsection{Boundary value problem for linear equations}
\label{subsec:linear}
\begin{definition}
 Let $\tau\in\mathfrak{M}(\xO\setminus \Sigma;\ei)$ and $\nu \in \mathfrak{M}(\partial\xO\cup \Sigma)$. We will say that $u$ is a weak solution of
\be\label{NHL} \left\{ \BAL
- L_\gm u&=\tau\qquad \text{in }\;\Gw\setminus \Sigma,\\
\tr(u)&=\xn,
\EAL \right.
\ee
if $u\in L^1(\xO\setminus \Sigma;\ei)$ and $u$ satisfies
\be \label{lweakform}
	- \int_{\Gw}u L_{\xm }\xi \, \dd x=\int_{\Gw \setminus \Sigma} \xi \, \dd \tau - \int_{\Gw} \mathbb{K}_{\xm}[\xn]L_{\xm }\xi \, \dd x
	\qquad\forall \xi \in\mathbf{X}_\xm(\xO\setminus \Sigma).
	\ee
\end{definition}

\begin{theorem}[{\cite[Theorem 1.8]{GkiNg_linear}}] \label{linear-problem}
Let $\tau,\rho\in\mathfrak{M}(\xO\setminus \Sigma;\ei)$, $\xn \in \mathfrak{M}(\partial\xO\cup \Sigma)$ and $f\in L^1(\xO;\ei)$.  Then there exists a unique weak solution $u\in L^1(\xO;\ei)$ of \eqref{NHL}. Furthermore
\be \label{reprweaksol}
u=\mathbb{G}_{\mu}[\tau]+\mathbb{K}_{\xm}[\xn]
\ee
and for any $\zeta \in \mathbf{X}_\xm(\xO\setminus \Sigma),$ there holds
\ba \label{esti2}
\|u\|_{L^1(\Omega;\ei)} \leq \frac{1}{\lambda_\mu}\| \tau \|_{\GTM(\Omega \setminus \Sigma;\ei)} + C \| \nu \|_{\GTM(\partial\Omega \cup \Sigma)},
\ea
where $C=C(N,\Omega,\Sigma,\mu)$.
In addition, if $\dd \tau=f\dd x+\dd\rho$ then, for any $0 \leq \zeta \in \mathbf{X}_\xm(\xO\setminus \Sigma)$, the following estimates are valid
\be\label{poi4}
-\int_{\Gw}|u|L_{\xm }\zeta \, \dd x\leq \int_{\Gw}\sign(u)f\zeta\, \dd x +\int_{\Gw \setminus \Sigma}\zeta \, \dd|\rho|-
\int_{\Gw}\mathbb{K}_{\xm}[|\xn|] L_{\xm }\zeta \, \dd x,
\ee
\be\label{poi5}
-\int_{\Gw}u^+L_{\xm }\zeta \, \dd x\leq \int_{\Gw} \sign^+(u)f\zeta\, \dd x +\int_{\Gw \setminus \Sigma}\zeta\, \dd\rho^+-
\int_{\Gw}\mathbb{K}_{\xm}[\nu^+]L_{\xm }\zeta \,\dd x.
\ee
%\be\label{poi6}
%-\int_{\Gw}u^- L_{\xm }\zeta \, \dd x\leq \int_{\Gw} \sign^-(u)f\zeta\, \dd x +\int_{\Gw \setminus \Sigma}\zeta\, \dd\rho^- -
%\int_{\Gw}\mathbb{K}_{\xm}[\nu^-]L_{\xm }\zeta \,\dd x.
%\ee
%where
%$$ \sign^+(u) =  \frac{\sign(u) + 1}{2}, \quad \sign^-(u) = \frac{\sign(u)-1}{2}.
%$$
\end{theorem}

\subsection{Weak Lebesgue estimates on Green kernel and Martin kernel}
In this subsection, we present sharp weak Lebesgue estimates for the Green kernel and Martin kernel.

We first recall the definition of weak Lebesgue spaces (or Marcinkiewicz spaces). Let $D \subset \R^N$ be a domain. Denote by $L^\kappa_w(D;\tau)$, $1 \leq \kappa < \infty$, $\tau \in \GTM^+(D)$, the
weak Lebesgue space (or Marcinkiewicz space) defined as follows: a measurable function $f$ in $D$
belongs to this space if there exists a constant $c$ such that
\bel{distri} \gl_f(a;\tau):=\tau(\{x \in D: |f(x)|>a\}) \leq ca^{-\kappa},
\forevery a>0. \ee
The function $\gl_f$ is called the distribution function of $f$ (relative to
$\tau$). For $\kappa \geq 1$, denote
\bal
L^\kappa_w(D;\tau)=\{ f \text{ Borel measurable}:
\sup_{a>0}a^\kappa\gl_f(a;\tau)<\infty\},
\eal
\bel{semi}
\norm{f}^*_{L^\kappa_w(D;\tau)}=(\sup_{a>0}a^\kappa\gl_f(a;\tau))^{\frac{1}{\kappa}}. \ee
The $\norm{.}_{L^\kappa_w(D;\tau)}^*$ is not a norm, but for $\kappa>1$, it is
equivalent to the norm
\bel{normLw} \norm{f}_{L^\kappa_w(D;\tau)}=\sup\left\{
\frac{\int_{A}|f|\dd\tau}{\tau(A)^{1-\frac{1}{\kappa}}}: A \sbs D, A \text{
	measurable},\, 0<\tau(A)<\infty \right\}. \ee
More precisely,
\bel{equinorm} \norm{f}^*_{L^\kappa_w(D;\tau)} \leq \norm{f}_{L^\kappa_w(D;\tau)}
\leq \frac{\kappa}{\kappa-1}\norm{f}^*_{L^\kappa_w(D;\tau)}. \ee

When $\dd\tau=\varphi \, \dd x$ for some positive continuous function $\varphi$, for simplicity, we use the notation $L_w^\kappa(D;\varphi)$.  Notice that
\bel{LpLpweak}
L_w^\kappa(D;\varphi) \sbs L^{r}(D;\varphi) \quad \text{for any } r \in [1,\kappa).
\ee
From \eqref{semi} and \eqref{equinorm}, one can derive the following estimate which is useful in the sequel. For any $f \in L_w^\kappa(D;\varphi)$, there holds
\bel{ue} \int_{\{x \in D: |f(x)| \geq s\} }\varphi \, \dd x \leq s^{-\kappa}\norm{f}^\kappa_{L_w^\kappa(D;\varphi)}.
\ee

Recall that $\am$ is defined in \eqref{apm}. Put
\be \label{p5}
p_{\am}:=\min\left\{\frac{N-\am}{N-2-\am},\frac{N+1}{N-1}\right\}.
\ee
Notice that if $\mu>0$ then $\am>0$, hence $p_{\am}=\frac{N+1}{N-1}$.
\begin{theorem}[Theorem 3.8 and Theorem 3.9 in \cite{GkiNg_source}] \label{lpweakgreen}
	There holds
	\bel{estgreen}
	\norm{\BBG_\mu[\gt]}_{L_w^{p_{\am}}(\Gw\setminus \Sigma;\ei)} \lesssim \norm{\gt}_{\mathfrak{M}(\xO\setminus \Sigma;\ei)}, \quad \forall \tau\in \mathfrak{M}^+(\xO\setminus \Sigma;\ei).
	\ee
	The implicit constant depends on $N,\Omega,\Sigma,\mu$.
\end{theorem}

%\begin{theorem} \label{lpweakgreen} ~~
%
%(i) If $0 < \mu \leq H^2$ then
%	\ba \label{estgreen}
%	\norm{\BBG_\mu[\gt]}_{L_w^{\frac{N+1}{N-1}}(\Gw\setminus \Sigma;\ei)} \lesssim \norm{\gt}_{\mathfrak{M}(\xO\setminus \Sigma;\ei)}, \quad \forall \tau\in \mathfrak{M}^+(\xO\setminus \Sigma;\ei).
%	\ea
%	The implicit constant depends on $N,\Omega,\Sigma,\mu$.
%	
%	(ii) If $\mu \leq 0$ then
%	\ba \label{estgreen2}
%	\norm{\BBG_\mu[\gt]}_{L_w^{p_{*}}(\Gw\setminus \Sigma;\ei)} \lesssim \norm{\gt}_{\mathfrak{M}(\xO\setminus \Sigma;\ei)}, \quad \forall \tau\in \mathfrak{M}^+(\xO\setminus \Sigma;\ei)
%	\ea
%where
%\bal
%p_{*}:=\min\left\{\frac{N-\am}{N-\am-2}, \frac{N+1}{N-1}\right\}.
%\eal
%\end{theorem}

\begin{theorem}[Theorem 3.10 in \cite{GkiNg_source}]\label{lpweakmartin1} ~~
	
\noindent	{\sc I.} Assume $\mu \leq H^2$ and $\gn\in \mathfrak{M}(\partial\xO\cup \Sigma)$ with compact support in $\partial\xO.$ Then
	\bel{estmartin1}
	\norm{\mathbb{K}_\mu[\nu]}_{L_w^{\frac{N+1}{N-1}}(\Gw\setminus \Sigma;\ei)} \lesssim \|\nu\|_{\mathfrak{M}(\partial\Omega)}.
	\ee
	
\noindent	{\sc II.} Assume $\gn\in \mathfrak{M}(\partial\xO\cup \Sigma)$ with compact support in $\Sigma$.
	
\noindent	(i) If $\mu < \left( \frac{N-2}{2} \right)^2$ then
	\bel{estmartin2}
	\norm{\mathbb{K}_{\mu}[\nu]}_{L_w^{\frac{N-\am}{N-\am-2}}(\Gw\setminus \Sigma;\ei)} \lesssim \norm{\nu}_{\mathfrak{M}(\Sigma)}.
	\ee
	
\noindent 	(ii) If $k=0$, $\Sigma=\{0\}$ and $\mu = \left( \frac{N-2}{2} \right)^2$ then for any $1<\theta<\frac{N+2}{N-2}$,
	\bel{estmartin2cr}
	\norm{\mathbb{K}_{\mu}[\nu]}_{L_w^{\theta}(\Gw\setminus \{0\};\ei )} \lesssim \norm{\nu}_{\mathfrak{M}(\Sigma)}.
	\ee
In addition, for $\lambda>0$, set
\ba \label{69a}
\tilde{A}_\xl(0):=\Big\{x\in \xO\setminus \{0\}:\;  \mathbb{K}_{\mu}[\xd_0](x)>\xl \Big \}, \quad \tilde{m}_{\xl}&:=\int_{\tilde{A}_\xl(0)}d(x)|x|^{-\frac{N-2}{2}} \dd x,
\ea
where $\xd_0$ is the Dirac measure concentrated at $0$. Then,
     \ba\label{54a}
\tilde{m}_{\xl}\lesssim (\xl^{-1}\ln\xl)^{\frac{N+2}{N-2}}, \quad \forall \xl>e.
\ea
The implicit constant depends on $N,\Omega,\Sigma,\mu$ and $\theta$.
\end{theorem}

\section{Boundary value problem for semilinear equations} \label{sec:BVP}
In the sequel, we assume that $g: \R\to \R$ is a nondecreasing continuous function such that $g(0)=0$.

\subsection{Sub and super solutions theorem}
We start with the definition of subsolutions and supersolutions of \eqref{NLin}.
\begin{definition} \label{def:subsupersol}
	% Let $\tau\in\mathfrak{M}(\xO\setminus \Sigma;\ei)$ and $\xn \in \mathfrak{M}(\partial\xO\cup \Sigma).$ A function $u$ is a weak solution of
	%\begin{equation}\label{NLin} \left\{ \BAL
	%- L_\gm u+g(u)&=\tau\qquad \text{in }\;\Gw\setminus \Sigma,\\
	%\tr(u)&=\nu,
	%\EAL \right. \end{equation}
%if $u\in L^1(\Omega;\ei)$, $g(u) \in L^1(\Omega;\ei)$  and $u$ satisfies
%\be \label{nlinearweakform}
%	- \int_{\xO}u L_{\xm }\zeta \, \dd x+ \int_{\xO}g(u)\zeta \, \dd x=\int_\xO \zeta \, \dd \tau - \int_{\Gw} \mathbb{K}_{\xm}[\xn]L_{\xm }\zeta \, \dd x
%	\qquad\forall \zeta \in\mathbf{X}_\xm(\xO\setminus \Sigma).
%	\ee
%
A function $u$ is a weak subsolution  (resp. supersolution) of \eqref{NLin} if $u\in L^1(\Omega;\ei)$, $g(u) \in L^1(\Omega;\ei)$  and
\be \label{nonlinearsubsupweakform}
- \int_{\xO}u L_{\xm }\zeta \, \dd x+ \int_{\xO}g(u)\zeta \, \dd x\leq (\text{resp.} \geq)\int_{\Omega \setminus \Sigma} \zeta \, \dd \tau - \int_{\Gw} \mathbb{K}_{\xm}[\xn]L_{\xm }\zeta \, \dd x
\qquad\forall 0\leq\zeta \in\mathbf{X}_\xm(\xO\setminus \Sigma).
\ee
\end{definition}

\begin{lemma}\label{weaksubsupersolution}
(i) Let $u\in L^1(\Omega;\ei)$ be a weak supersolution of \eqref{NLin}. Then there exist $\tau_u \in \mathfrak{M}^+(\xO\setminus \Sigma;\ei)$ and $\xn_u \in \mathfrak{M}^+(\partial\xO\cup \Sigma)$ such that $u$ is a weak solution of

\be\label{weaksup} \left\{ \BAL
- L_\gm u+g(u)&=\tau + \tau_u\qquad \text{in }\;\Gw\setminus \Sigma,\\
\tr(u)&=\nu+ \nu_u.
\EAL \right.
\ee
(ii) Let $u\in L^1(\Omega;\ei)$ be a weak subsolution of \eqref{NLin}. Then there exist $\tau_u\in \mathfrak{M}^+(\xO\setminus \Sigma;\ei)$ and $\xn_u\in \mathfrak{M}^+(\partial\xO\cup \Sigma)$ such that $u$ is a weak solution of

\be\label{weaksub} \left\{ \BAL
- L_\gm u+g(u)&=\tau-\tau_u\qquad \text{in }\;\Gw\setminus \Sigma,\\
\tr(u)&=\nu- \nu_u.
\EAL \right.
\ee
\end{lemma}
\begin{proof} (i) Let $w$ be the unique solution of

\be\label{weaksup1} \left\{ \BAL
- L_\gm w+g(u)&=\tau\qquad \text{in }\;\Gw\setminus \Sigma,\\
\tr(u)&=\xn.
\EAL \right.
\ee
Then
\be \label{z11}
	 -\int_{\xO}(w-u) L_{\xm }\zeta \, \dd x\leq 0
	\qquad\forall 0\leq\zeta \in\mathbf{X}_\xm(\xO\setminus \Sigma).
\ee
Let $\eta \in\mathbf{X}_\xm(\xO\setminus \Sigma) $ be such that $-L_\xm\eta=\sign^+(w-u)\ei$. Then by using $\eta$ as a test function in \eqref{z11}, we obtain that $w \leq u$ in $\Omega \setminus \Sigma$.

Set $v=u-w$ then $v \geq 0$ in $\Omega \setminus \Sigma$ and $-L_\mu v\geq0$ in the sense of distributions in $\xO\setminus \Sigma$. This implies the existence of a nonnegative Radon measure $\tau_u$ in $\xO\setminus \Sigma$ such that $-L_\xm v=\tau_u$ in the sense of distribution. By \cite[Corollary 1.2.3]{MVbook}, $v\in W^{1,\kappa}_{loc}(\xO\setminus \Sigma)$ for some $\kappa>1$.
Let $\{O_n\}$ be a smooth exhaustion of $\xO\setminus \Sigma$ and $\xz_n$ be the weak solution of
\be\label{zetan} \left\{ \BAL
- L_\gm \xz_n&=0\qquad &&\text{in }\;O_n,\\
\xz_n&=v \qquad &&\text{on }\;\partial O_n.
\EAL \right.
\ee
Therefore $v=\mathbb{G}_{\mu}^{O_n}[\tau_u]+\xz_n$. Since $\tau_u,\xz_n$ are nonnegative and $G_{\mu}^{O_n}(x,y)\nearrow G_{\mu}(x,y)$ for any $x\neq y$ and $x,y\in \xO\setminus \Sigma$, we obtain $0 \leq \mathbb{G}_{\mu}[\tau_u]\leq v$  a.e. in  $\Omega \setminus \Sigma$. In particular, $0 \leq \mathbb{G}_{\mu}[\tau_u](x^*) \leq v(x^*)$ for some point $x^* \in \Omega \setminus \Sigma$. This, together with the estimate $G_\mu(x^*,\cdot) \gtrsim \ei$ a.e. in $\Omega$, implies $\tau_u\in \mathfrak{M}(\xO\setminus \Sigma;\ei)$.

Moreover, we observe from above that $v-\mathbb{G}_{\mu}[\tau_u]$ is a nonnegative $L_\mu$-harmonic function in $\xO\setminus \Sigma$. Thus by Theorem \ref{th:Rep} there exists a unique $\nu_u\in \GTM^+(\partial\Omega \cup \Sigma)$ such that
$v-\mathbb{G}_{\mu}[\tau_u]=\mathbb{K}_{\xm}[\xn_u]$   a.e. in  $\Omega \setminus \Sigma$.
This, together with $w + \BBG_\mu[g(u)] = \BBG_\mu[\tau] + \BBK_{\mu}[\nu]$, yields
\bal u+\mathbb{G}_{\mu}[g(u)]=\mathbb{G}_{\mu}[\tau+\tau_u]+\mathbb{K}_{\xm}[\nu+\nu_u],
\eal
which  means that $u$ is a weak solution of \eqref{weaksup}.

(ii) The proof is similar to that of (i) and we omit it.
\end{proof}

The main result of this subsection is the following sub and super solution theorem.
\begin{theorem} \label{existencesubcr}
Assume $\tau \in \GTM(\Omega \setminus \Sigma;\ei)$ and $\nu \in \GTM(\partial \Omega \cup \Sigma)$. Let $v,w \in L^1(\Omega;\ei)$ be weak subsolution and supersolution of \eqref{NLin} respectively such that $v\leq w$  in $\Omega \setminus \Sigma$ and  $g(v), g(w) \in L^1(\Omega;\ei)$. Then problem \eqref{NLin} admits a unique weak solution $u \in L^1(\Omega;\ei)$ which satisfies $v \leq u \leq w$ in $\Omega \setminus \Sigma$.
\end{theorem}
\begin{proof} \textit{Uniqueness.} If $u_1$ and $u_2$ are two solutions of \eqref{NLin} then $u_1-u_2$ satisfies
\bal \left\{ \BAL
- L_\mu (u_1-u_2)+g(u_1)-g(u_2)&=0 \qquad \text{in }\;\Omega\setminus \Sigma,\\
\tr(u_1-u_2)&=0.
\EAL \right.
\eal
Then by using \eqref{poi4} with $u=u_1-u_2$, $f=-(g(u_1)-g(u_2))$, $\rho=0$ and $\nu=0$, we have
\bal -\int_{\Omega}|u_1-u_2| L_\mu \zeta \, \dd x + \int_{\Omega}\sign(u_1-u_2)(g(u_1)-g(u_2))\zeta \,\dd x \leq 0.
\eal
Choosing $\zeta=\ei$ and keeping in mind that $g$ is nondecreasing, we obtain from the above estimate that $u_1=u_2$ in $\Omega \setminus \Sigma$. \medskip

\noindent \textit{Existence.} We follow some ideas of the proof of \cite[Theorem 2.2.4]{MVbook}.
Define
\be \label{gnt}
g_n(t):= \max\{ -n, \min\{g(t),n  \} \}.
\ee
Set
\bal
\tilde g_n(z(x)):=\left\{ \BAL
&g_n(w(x))\quad&&\text{if } z(x) \geq w(x),\\
&g_n(z(x))\quad&&\text{if } v(x)< z(x)< w(x),\\
&g_n(v(x))\quad&&\text{if } z(x)\leq v(x).
\EAL \right.
\eal

Let $u\in L^1(\Omega;\ei)$ and denote by $\mathbb{T}(u)$ the unique solution of
\be\label{fixed} \left\{ \BAL
- L_\gm \varphi +\tilde g_n(u)&=\tau\qquad \text{in }\;\Gw\setminus \Sigma,\\
\tr(\varphi)&=\xn.
\EAL \right.
\ee
Then $\mathbb{T}(u) \in L^1(\Omega;\ei)$ and
\be\label{fixedweaksol}
\mathbb{T}(u)=-\mathbb{G}_{\mu}[\tilde g_n(u)]+\mathbb{G}_{\mu}[\tau]+\mathbb{K}_{\xm}[\xn].
\ee
By \cite[Remark 5.5]{GkiNg_linear}, $\BBG_\mu[1](x) \lesssim d(x)d_\xS(x)^{\min\{\am,0\}}$ for a.e. $x \in \Omega \setminus \Sigma$. Therefore, there exists a constant $C=C(\Omega,\Sigma,N,\mu)>0$ such that
\be
|\mathbb{T}(u)|\leq C nd\,d_\xS^{\min\{\am,0\}}+\mathbb{G}_{\mu}[|\tau|]+\mathbb{K}_{\xm}[|\xn|]. \label{fragmafixed0}
\ee
By Theorems \ref{lpweakgreen} -- \ref{lpweakmartin1}, estimate \eqref{LpLpweak} (with $D=\Omega \setminus \Sigma$ and $\varphi=\ei$), estimate \eqref{eigenfunctionestimates},  and the above inequality we can show that there exists $C_1=C_1(\xO,\Sigma,N,\xm)>0$ such that
\be \label{fragmafixed}
\| \mathbb{T}(u)\|_{L^1(\Omega;\ei)} \leq C_1(n+\norm{\tau}_{\mathfrak{M}(\xO\setminus \Sigma;\ei)}+\norm{\xn}_{\mathfrak{M}(\partial\xO\cup \Sigma)}).
\ee

We will use the Schauder fixed point theorem to prove the existence of a fixed point of $\BBT$ by examining the following criteria.

\emph{The operator $\mathbb{T}: L^1(\Omega;\ei) \to L^1(\Omega;\ei)$ is continuous.} Indeed, let $\{\varphi_m\}$ be a sequence such that $\varphi_m\rightarrow \varphi$ in $L^1(\Omega;\ei)$ as $m \to \infty$. Since $g_n$ is continuous and bounded, we can easily show that $\tilde g_n(\varphi_m)\rightarrow \tilde g_n(\varphi)$ in  $L^1(\Omega;\ei),$ which implies $\mathbb{T}(\varphi_m)\to \mathbb{T}(\varphi)$ as $m \to \infty$  in  $L^1(\Omega;\ei),$ by \eqref{fixedweaksol} and \eqref{estgreen}.

\textit{The operator $\mathbb{T}$ is compact}. Indeed, let $\{\varphi_m\}$ be a sequence in $L^1(\Omega;\ei)$ then by \eqref{fragmafixed} and \cite[Theorem 1.2.2]{MVbook}, $\{\mathbb{T}(\varphi_m)\}$ is uniformly bounded in $W^{1,\kappa}(D)$ for any $1<\kappa<\frac{N}{N-1}$ and any open set $D\Subset \xO\setminus \Sigma$. Therefore there exist $\psi\in W^{1,\kappa}_{loc}(\xO\setminus \Sigma) $ and a subsequence still denoted by $\{\mathbb{T}(\varphi_m)\}$ such that $\mathbb{T}(\varphi_m)\rightarrow \psi$ in $L^\kappa_{loc}(\xO\setminus \Sigma)$ and a.e. in $\Omega \setminus \Sigma$. By \eqref{fragmafixed0} and the dominated convergence theorem, we deduce that $\mathbb{T}(\varphi_m)\rightarrow \psi$ in $L^1(\Omega;\ei)$.

Now set
\bal \CA:=\{ \varphi \in L^1(\Omega;\ei):  \|\varphi\|_{L^1(\Omega;\ei)}\leq C_1(n+\norm{\tau}_{\mathfrak{M}(\xO\setminus \Sigma;\ei)}+\norm{\xn}_{\mathfrak{M}(\partial\xO\cup \Sigma)})  \}.
\eal
Then $\CA$ is a closed, convex subset of $ L^1(\Omega;\ei)$ and $\BBT( \CA) \sbs \CA$.
Thus we can apply Schauder fixed point theorem to obtain the existence of a function $u_n\in \CA$ such that $\mathbb{T}(u_n)=u_n$. This means $u_n$ satisfies
\be\label{fixed2} \left\{ \BAL
- L_\gm u_n+\tilde g_n(u_n)&=\tau\qquad \text{in }\;\Gw\setminus \Sigma,\\
\tr(u_n)&=\xn.
\EAL \right. \ee
%Let $z_n\in L^1_\ei(\xO\setminus \Sigma)$ be the unique solution of \eqref{fixed2}.
Then
\be \label{fragmazn}
|u_n|=|-\mathbb{G}_{\mu}[\tilde g_n(u)]+\mathbb{G}_{\mu}[\tau]+\mathbb{K}_{\xm}[\xn]|\leq\mathbb{G}_{\mu}[ |g(w)|+|g(v)|]+\mathbb{G}_{\mu}[|\tau|]+\mathbb{K}_{\xm}[|\xn|],
\ee
which implies
\be\label{l1fragmazn}
 \norm{u_n}_{L^1(\Omega;\ei)}\leq C_2( \norm{g(w)}_{L^1(\Omega;\ei)}+ \norm{g(v)}_{L^1(\Omega;\ei)}+\norm{\tau}_{\mathfrak{M}(\xO\setminus \Sigma;\ei)}+\|\xn\|_{\mathfrak{M}(\partial\xO\cup \Sigma)}),
\ee
for some positive constant $C_2=C_2(\xO,\Sigma,N,\xm).$

Thus by \cite[Theorem 1.2.2]{MVbook}, $\{u_n\}$ is uniformly bounded in $W^{1,\kappa}(D)$ for any $1<\kappa<\frac{N}{N-1}$ and any open set $D\Subset \xO\setminus \Sigma$. Therefore there exist $u\in W^{1,\kappa}_{loc}(\xO\setminus \Sigma) $ and a subsequence still denoted by $\{u_n\}$ such that $u_n\rightarrow u$ in $L^\kappa_{loc}(\xO\setminus \Sigma)$
and a.e. in $\xO\setminus \Sigma$. By \eqref{fixedweaksol} and the dominated convergence theorem, we deduce that $u_n\rightarrow u$ in $L^1(\Omega;\ei)$.
Taking into account that $|\tilde g_n(u_n)|\leq |g(w)|+|g(v)|,$ we can easily show that $\tilde g_n(u_n)\to \tilde g(u)$ in $L^1(\Omega;\ei)$, where
\be \label{tilg}
\tilde g(u(x))=\left\{ \BAL
&g(w(x))\quad&&\text{if } u(x)\geq w(x),\\
&g(u(x))\quad&&\text{if } v(x)\leq u(x)\leq w(x),\\
&g(v(x))\quad&&\text{if } u(x)\leq v(x).
\EAL \right. \ee
Combining all above we deduce that $u$ is a weak solution of
\be\label{fixed3} \left\{ \BAL
- L_\gm u +\tilde g(u)&=\tau\qquad \text{in }\;\Gw\setminus \Sigma,\\
\tr(u)&=\xn.
\EAL \right.
\ee
Since $w$ is a supersolution of \eqref{NLin}, by Lemma \ref{weaksubsupersolution} there exist measures $\tau_w\in \mathfrak{M}^+(\xO\setminus \Sigma;\ei)$ and $\xn_w\in \mathfrak{M}^+(\partial\xO\cup \Sigma)$ such that $w$ is a weak solution of
\be\label{fixed4} \left\{ \BAL
- L_\gm w+g(w)&=\tau+\tau_w\qquad \text{in }\;\Gw\setminus \Sigma,\\
\tr(w)&=\nu + \xn_w.
\EAL \right. \ee
From \eqref{fixed3} and \eqref{fixed4}, we deduce
\be\label{fixed5} \left\{ \BAL
- L_\gm (u-w) &= -(\tilde g(u) - g(w))-\tau_w\qquad \text{in }\;\Gw\setminus \Sigma,\\
\tr(u-w)&= -\nu_w.
\EAL \right. \ee
Applying \eqref{poi5} for \eqref{fixed5} yields
\bal -\int_{\Omega}(u-w)^+L_\mu \zeta \,\dd x \leq - \int_{\Omega}\sign^+(u-w)(\tilde g(u)- g(w)) \zeta \,\dd x \qquad \forall \zeta \in {\bf X}_\mu(\Omega \setminus \Sigma).
\eal
By taking $\zeta=\ei$ and taking into account the definition of $\tilde g(u)$ in \eqref{tilg}, we derive that
$ \int_{\Omega}(u-w)^+\, \ei \dd x \leq 0$,
which implies $u \leq w$.

Similarly we can show that $u \geq v$ in $\Omega \setminus \Sigma$. Therefore $\tilde g(u) = g(u)$ and thus $u$ is a weak solution of \eqref{NLin}.
\end{proof}

\subsection{Sufficient conditions for existence}
We first prove Theorem \ref{existGK}.

\begin{proof}[\textbf{Proof of Theorem \ref{existGK}}.]
Put $U_1=-\BBG_\mu[\tau^-] - \BBK_{\mu}[\nu^-]$ and $U_2=\BBG_\mu[\tau^+] + \BBK_{\mu}[\nu^+]$. By Theorems \ref{lpweakgreen}--\ref{lpweakmartin1} and \eqref{LpLpweak} (with $D=\Omega \setminus \Sigma$ and $\varphi=\ei$), $U_1, U_2 \in L^1(\Omega;\ei)$ and by the assumption, $g(U_1), g(U_2) \in L^1(\Omega;\ei)$. Moreover, we see that $U_1$ and $U_2$ are subsolution and supersolution of \eqref{NLin} respectively. Therefore, by Theorem \ref{existencesubcr}, there exists a unique solution $u$ of \eqref{NLin} which satisfies \eqref{U12}. The proof is complete.
\end{proof}

%\begin{lemma} \label{subcrcon} Let $g$ be continuous, nondecreasing, which satisfies $g(0)=0$ and
%\begin{equation} \label{subcd0} \int_1^\infty  s^{-p-1}(\ln s)^{q} (g(s)-g(-s))ds<\infty
%\end{equation}
%for some $p>1$ and $q \geq 0$. Let $v$ be a function defined in $\Omega \setminus \Sigma$. For $s>0$, set
%$$E_s(v):=\{x\in \xO\setminus \Sigma:| v(x)|>s\} \quad \text{and} \quad e(s):=\int_{E_s(v)} \ei dx.$$
%Assume that there exists a positive constant $C_0$ such that
%\bel{e}
%e(s) \leq C_0s^{-p}(\ln s)^q, \quad \forall s>1.
%\ee
%Then for any $s_0>e^\frac{2q}{p},$ there holds
%\ba\label{53}\BAL
%\norm{g(|v|)}_{L^1(\Omega;\ei)}&\leq \int_{(\Omega \setminus \Sigma) \setminus E_{s_0}(v)} g(|v|)\ei dx+pC_0\int_{s_0}^\infty  s^{-p-1}(\ln s)^{q} g(s)ds\\
%\norm{g(-|v|)}_{L^1(\Omega;\ei)}&\leq  -\int_{(\Omega \setminus \Sigma) \setminus E_{s_0}(v)} g(-|v|)\ei dx-pC_0\int_{s_0}^\infty  s^{-p-1}(\ln s)^{q} g(-s)ds
%\EAL
%\ea
%\end{lemma}

In order to prove Theorem \ref{exist-subGK}, we need the following result.
\begin{lemma}[{\cite[Lemma 5.1]{GkiNg_source}}] \label{subcrcon} Assume
	\ba \label{subcd0} \int_1^\infty  s^{-q-1}(\ln s)^{m} (g(s)-g(-s)) \,\dd s<\infty
	\ea
	for $q,m \in \R$, $q >1$ and $m \geq 0$. Let $v$ be a function defined in $\Omega \setminus \Sigma$. For $s>0$, set
	\bal E_s(v):=\{x\in \xO\setminus \Sigma:| v(x)|>s\} \quad \text{and} \quad e(s):=\int_{E_s(v)} \ei \,\dd x.
	\eal
	Assume that there exists a positive constant $C_0$ such that
	\ba \label{e}
	e(s) \leq C_0s^{-q}(\ln s)^m, \quad \forall s>e^\frac{2 m}{q}.
	\ea
	Then for any $s_0>e^\frac{2 m}{q}$ there hold
	\ba\label{53}
	\norm{g(|v|)}_{L^1(\Omega;\ei)}&\leq \int_{(\Omega \setminus \Sigma) \setminus E_{s_0}(v)} g(|v|)\ei \,\dd x + C_0 q \int_{s_0}^\infty  s^{-q-1}(\ln s)^{m} g(s) \,\dd s, \\ \label{53-a}
	\norm{g(-|v|)}_{L^1(\Omega;\ei)}&\leq  -\int_{(\Omega \setminus \Sigma) \setminus E_{s_0}(v)} g(-|v|)\ei \, \dd x - C_0 q \int_{s_0}^\infty  s^{-q-1}(\ln s)^{m} g(-s) \,\dd s.
	\ea
\end{lemma}

We are ready to demonstrate Theorem \ref{exist-subGK} and Theorem \ref{exist-measureK}.

\begin{proof}[\textbf{Proof of Theorem \ref{exist-subGK}}.]
	Let $U_1$ and $U_2$ as in Theorem \ref{existGK}. Then by Theorem \ref{lpweakgreen} and Theorem \ref{lpweakmartin1}, $U_1, U_2 \in L_w^{p_{\am}}(\Omega \setminus \Sigma;\ei)$ (recall that $\am$ is defined in \eqref{p5}). Applying Lemma \ref{subcrcon} for $q=\frac{N+1}{N-1}$ and $m=0,$ we deduce $g(U_1), g(U_2) \in L^1(\Omega;\ei)$. Finally, due to Theorem \ref{existGK}, there exists a unique solution $u$ of \eqref{NLin} which satisfies \eqref{U12}. The proof is complete.	
\end{proof}

%\begin{corollary} \label{exist-subGK-1} Assume $p$ satisfies \eqref{p}. Then for any $\tau \in \GTM(\Omega \setminus \Sigma; \ei)$ and $\nu \in \GTM(\partial \Omega \cup \Sigma)$, there exists a unique solution $u$ of
%\begin{equation}\label{NLpower} \left\{ \BAL
%- L_\gm u+|u|^{p-1}u&=\tau\qquad \text{in }\;\Gw\setminus \Sigma,\\
%\tr(u)&=\nu.
%\EAL \right. \end{equation}
%Moreover, $u$ satisfies \eqref{U12}.
%\end{corollary}
%\begin{proof}
%Let $U_1$ and $U_2$ as in the proof of Theorem \ref{existGK}. Then by By Theorems \ref{lpweakgreen}, \ref{lpweakmartin1} and \eqref{LpLpweak}, $U_1, U_2, |U_1|^p, |U_2|^p \in L^1(\Omega;\ei)$. The desired result follows from Theorem \ref{existGK}.	
%\end{proof}
%%%%%%%%%%%%%%%%%%%%%%%%%%%%%%%%%%%%%%%%%%%%%%

%%%%%%%%%%%%%%%%%%%%%%%%%%%%%%%%%%%%%%%%%%%%%%
%%%%%%%%%%%%%%%%%%%%%%%%%%%%%%%%%%%%%%%%%%%%%%

\section{Boundary data concentrated in $\partial \Omega$} \label{sec:BVP-partialO}
In this section, we consider the following problem
\be \label{BVP-O} \left\{ \BAL
-L_\mu u + g(u) &= 0 \quad \text{in } \Omega \setminus \Sigma, \\
\tr(u) &= \nu,
\EAL \right. \ee
where $\nu$ is concentrated in $\partial \Omega$. 
%Solutions of \eqref{BVP-O} are understood in the sense in Definition \ref{def:weak-sol}.

\subsection{Poisson kernel and $L_\mu$-harmonic measure on $\partial\Omega$}

The following result asserts the existence of the Poisson kernel and its properties.

\begin{proposition}  \label{measureboundary}
For any $x \in \xO\setminus \Sigma$,  $G_{\mu}(x,\cdot) \in C^{1,\gamma}(\overline{\Omega} \setminus (\Sigma \cup \{x\})) \cap C^2(\Omega \setminus (\Sigma \cup  \{x\}))$ for all $\xg\in (0,1)$. Let $P_\mu$ be the Poisson kernel defined by
\be \label{Poisson}  P_\mu(x,y):=-\frac{\partial G_\mu}{\partial{\bf n}}(x,y), \quad x \in \Omega \setminus \Sigma, \; y \in \partial \Omega,
\ee
where ${\bf n}$ is the unit outer normal vector of $\partial \xO$. Let $x_0 \in \Omega \setminus \Sigma$ be the fixed reference point.

(i) There holds
\ba\label{KP}
P_\mu(x,y)=P_\mu(x_0,y)K_\mu(x,y), \quad x \in \Omega \setminus \Sigma, \; y \in \partial \Omega.
\ea

(ii) For any $h \in L^1(\partial \Omega \cup \Sigma; \dd \omega^{x_0}_{\xO\setminus\xS})$ with compact support in $\partial \Omega$, there holds
\be \label{Pformula}
\int_{\partial\Omega} h(y) \, \dd \omega^{x_0}_{\xO\setminus\xS}(y) = \BBP_\mu[h](x_0).
\ee
Here
\be \label{BBP} \mathbb{P}_{\mu}[h](x)=\int_{\partial\Omega}P_\mu(x,y)h(y) \, \dd S(y).
\ee
where $S$ is the $(N-1)$-dimensional surface measure on $\partial\Omega$.

Consequently, if $h \in L^1(\partial \Omega \cup \Sigma; \dd \omega^{x_0}_{\xO\setminus\xS})$ with compact support in $\partial \Omega$ then $h \in L^1(\partial \Omega)$. In particular, for any Borel set  $E\subset \partial\xO$ there holds
\be \label{PE} \omega^{x_0}_{\xO\setminus\xS}(E)=\BBP_\mu[\1_E](x_0).
\ee
\end{proposition}
\begin{proof}
For any $x \in \Omega \setminus \Sigma$, the regularity of $G_{\mu}(x,\cdot)$ follows from the standard elliptic theory. Also, we note that
$P_\mu(\cdot,y)$ is $L_\mu$-harmonic in $\xO\setminus\xS$ and
\bal
\lim_{x\in\xO,\;x\rightarrow\xi}\frac{P_\mu(x,y)}{\tilde W(x)}=0\qquad\forall \xi\in \partial\xO\cup \Sigma\setminus \{y\}.
\eal
By the uniqueness of kernel functions with pole at $y$ and
 basis at $x_0$ (\cite[Proposition 6.6]{GkiNg_linear}), we deduce \eqref{KP}.

Now, let $\{\Sigma_n\}$ be  a decreasing sequence of bounded open smooth domains as in \eqref{Kn}. We denote by $\phi_*$ the unique solution of
\be\label{fi} \left\{  \BAL
-L_{\mu }u&=0 \qquad&&\text{in } \xO\setminus \Sigma\\
u&=1\qquad&&\text{on } \prt\Omega\\
u&=0\qquad&&\text{on } \Sigma.
\EAL \right.
\ee

Then by Lemma \cite[Lemma 5.6]{GkiNg_linear}, there exist constants $c_1=c_1(\Omega,\Sigma,\Sigma_n,\mu)$ and $c_2=c_2(\Omega,\Sigma,N,\mu)$ such that
$0<c_1 \leq\phi_*(x)\leq c_2 d_\Sigma(x)^{-\ap}$ for all $x \in \Omega \setminus \Sigma_n$.
By the standard elliptic theory, $\xf_*\in C^2(\Omega \setminus \Sigma) \cap C^{1,\xg}(\overline{\xO}\setminus \Sigma)$ for any $0<\xg<1$.

Let $\tilde \zeta \in C(\partial \Omega \cup \Sigma_n)$, we consider the problem
\be\label{harmonicrel1} \left\{  \BAL
-L_{\mu }v&=0 \qquad&&\text{in } \xO\setminus \Sigma_n\\
v&=\tilde\xz\qquad&&\text{on } \prt\xO\cup \partial \Sigma_n.
\EAL \right.
\ee

We observe that $v$ satisfies \eqref{harmonicrel1} if and only if $w=v/\phi_*$ satisfies
\be\label{harmonicrel2} \left\{  \BAL
-\text{div}(\phi_*^2\nabla w)&=0 \qquad&&\text{in } \Omega \setminus \Sigma_n\\
w&=\frac{\tilde\zeta}{\phi_*}\qquad&&\text{on } \prt\xO\cup \partial \Sigma_n.
\EAL \right.
\ee
We note that for any $\tilde\zeta \in C(\partial \Omega \cup \partial \Sigma_n)$, there exists a unique solution of \eqref{harmonicrel2}. From the above observation, we deduce that there exists a unique solution of \eqref{harmonicrel1}. Thus, for any $n$ and $x \in \Omega \setminus \Sigma$, there exists $L_\mu$-harmonic measure $\omega_n^{x}$  on $\partial \Omega \cup \partial \Sigma_n$. Denote by $v_n$ the solution of \eqref{harmonicrel1}, then
\be \label{wnrep1}
v_n(x)=\int_{\partial\xO\cup \partial \Sigma_n}\tilde\xz(y)\, \dd\xo_n^x(y).
\ee
For any $\xz \in C(\partial\xO),$ we set $\tilde\xz=\xz$ if $x\in \partial\xO,$ $\tilde\xz=0$ otherwise. In view of the proof of \cite[Proposition 6.12]{GkiNg_linear} and \eqref{wnrep1}, we may deduce that $v_n(x)\rightarrow v(x)=\int_{\partial\Omega \cup \Sigma} \zeta(y) \,\dd \omega^{x}_{\xO\setminus\xS}(y).
$

On the other hand, for any $n \in \N$,  the Green function of $-L_\mu$ in $\Omega \setminus \Sigma_n$ exists, denoted by $G^n_{\mu}$. We see that $G^n_{\mu}(x,y)\nearrow G_{\mu}(x,y)$ for any $x\neq y$ and $x,y\in \xO\setminus \Sigma$.

Denote the Poisson kernel of $-L_\mu$ in $\Omega \setminus \Sigma_n$ by
\bal P_\mu^n(x,y)=-\frac{\partial G_\mu^n}{\partial {\bf n}^n}(x,y), \quad x \in \Omega \setminus \Sigma_n, y \in \partial \Omega \cup \partial \Sigma_n,
\eal
where ${\bf n}^n$ is the unit outer normal vector of $\partial \Omega \cup \partial \Sigma_n$. Then we have the representation
\be \label{wnrep2} v_n(x) = \int_{\partial \Omega \cup \partial \Sigma_n}P_\mu^n(x,y)\tilde\zeta(y)\,\dd S(y),
\ee
where $S$ is the $(N-1)$-dimensional surface measure on $\partial\Omega \cup \partial \Sigma_n$. From \eqref{wnrep1} and \eqref{wnrep2} and using the fact that $\tilde\xz$ has compact support in $\partial\xO,$ we obtain

\be \label{zP}
\int_{\partial \Omega}\zeta(y)\dd \omega_n^{x}(y)=\int_{\partial \Omega}P_\mu^n(x,y)\zeta(y)\, \dd S(y).
\ee
Put $\beta=\frac{1}{2}\min\{d(x), \text{dist}(\partial\xO, \Sigma)\}$. Let $\Omega_\beta=\{ x \in \Omega: d(x)<\beta \}$. Then $\{G_{\mu}^n(x,\cdot)\}_n$ is uniformly bounded with respect to $W^{2,\kappa}(\Omega_\beta)$-norm for any $\kappa>1$. Thus, by compact embedding, there exists a subsequence, still denoted by $\{G_{\mu}^n(x,\cdot)\}_n$, which converges to $G_{\mu}(x,\cdot)$ in $C^1(\overline{\Omega_\beta})$ as $n \to \infty$. In particular $P_\mu^n(x,\cdot) \to P_\mu(x,\cdot)$ uniformly on $\partial \Omega$ as $n \to \infty$.

Therefore, by letting $n \to \infty$ in \eqref{zP}, we obtain
\be \label{zP1} \begin{aligned} \int_{\partial\Omega} \zeta(y) \,\dd \omega^{x}_{\xO\setminus\xS}(y)&=\lim_{n\to\infty}\int_{\partial\Omega} \zeta(y) \,\dd \omega^{x}_n(y) \\
&=\lim_{n\to\infty}\int_{\partial\Omega} P_{\mu}^n(x,y)\zeta(y) \, \dd S(y)=\int_{\partial\Omega} P_{\mu}(x,y)\zeta(y) \,\dd S(y).
\end{aligned} \ee
Since $\inf_{y \in \partial \Omega}P_\mu(x_0,y)>0$ and \eqref{zP1} holds for any $\xz\in C(\partial\xO),$ we have that \eqref{PE} is valid, which implies \eqref{Pformula}. The proof is complete.
\end{proof}

\begin{proposition}\label{weaksolution0}
(i) For any $h \in L^1(\partial\xO \cup \Sigma;\dd \omega^{x_0}_{\xO\setminus\xS})$ with support on $\partial \Omega$, there holds
\be \label{weakfor1}
-\int_\xO \BBK_\mu[h\dd \omega^{x_0}_{\xO\setminus\xS}] L_\xm\eta \, \dd x=-\int_{\partial\xO}\frac{\partial \eta}{\partial{\bf n}}(y)h(y)\, \dd S(y),\quad \forall \eta \in \mathbf{X}_\xm(\xO\setminus \Sigma).
\ee

(ii) For any $\xn\in \mathfrak{M}(\prt\xO \cup \Sigma)$ with support on $\partial \Omega$, there holds
\be \label{weakfor2}
-\int_\xO \BBK_\mu[\nu] L_\xm\eta \, \dd x=-\int_{\partial\xO}\frac{\partial \eta}{\partial{\bf n}}(y)\frac{1}{P_\mu(x_0,y)}\, \dd \nu(y),\quad \forall \eta \in \mathbf{X}_\xm(\xO\setminus \Sigma),
\ee
where $P_\mu(x_0,y)$ is defined in \eqref{Poisson} and ${\bf X}_\mu(\Gw\setminus \Sigma)$ is defined by \eqref{Xmu}.

\end{proposition}
\begin{proof} (i) Let $\{\Sigma_n\}$ be as in \eqref{Kn}. Let $\eta\in \mathbf{X}_\xm(\xO\setminus \Sigma),$ $\xz\in C(\partial\xO\cup \partial \Sigma_n)$ with compact support in $\partial\xO$ and $v_n$ be the solution of \eqref{harmonicrel1}.

In view of the proof of Proposition \ref{measureboundary}, $v_n\in C(\overline{\xO\setminus \Sigma_n})$ and
\bal
v_n(x)=\int_{\partial\xO} \xz(y) \, \dd\xo^{x}_n(y) =\int_{\partial\Omega} P_{\mu}^n(x,y)\zeta(y) \, \dd S(y).
\eal
Put
\bal
v(x)=\int_{\partial\xO} \zeta(y) \, \dd\xo^{x}(y) \quad \text{and} \quad w(x)=\int_{\partial\xO} |\xz(y)| \, \dd\xo^{x}(y).
\eal
Then $v_n(x)\rightarrow v(x)$ and $|v_n(x)|\leq w(x)$. By \cite[Proposition 1.3.7]{MVbook},

\be
-\int_{\xO\setminus \Sigma_n}v_n L_\xm Z \, \dd x=-\int_{\partial\xO}\xz\frac{\partial Z}{\partial{\bf n}} \, \dd S,\quad \forall Z \in C^2_0(\xO\setminus \Sigma_n).
\ee
By approximation, the above equality is valid for any $Z \in C^{1,\xg}(\overline{\xO\setminus \Sigma_n}),$ for some  $\xg\in(0,1)$ and $\xD Z\in L^\infty.$ Hence, we may choose $Z=\eta_n$, where $\eta_n$ satisfies
\bal \left\{  \BAL
-L_{\mu }\eta_n&=-L_\xm\eta \qquad&&\text{in } \xO\setminus \Sigma_n\\
\eta_n &=0\qquad&&\text{on } \prt\xO\cup \partial \Sigma_n,
\EAL \right.
\eal
we obtain
\be
-\int_{\xO \setminus \Sigma_n} v_n L_\xm \eta_n \, \dd x=-\int_{\partial\xO}\xz\frac{\partial \eta_n}{\partial{\bf n}} \, \dd S.
\ee
We note that $\eta_n\rightarrow \eta$ a.e. in $\xO\setminus \Sigma$ and in $C^1(\overline{\xO\setminus \Sigma_1})$. Therefore by the dominated convergence theorem, we obtain
\be
-\int_{\xO} v L_\xm \eta \, \dd x=-\int_{\partial\xO}\xz\frac{\partial \eta}{\partial{\bf n}}\, \dd S.\label{weak1}
\ee

Now let $h \in L^1(\partial \Omega \cup \Sigma;\dd \omega^{x_0}_{\xO\setminus\xS})$ with support on $\partial \Omega$ and $\{h_n\}$ be a sequence of functions in $C(\partial\xO \cup \Sigma)$ with support on $\partial \Omega$ such that $h_n\rightarrow h$ in $L^1(\partial\xO \cup \Sigma;\dd \omega^{x_0}_{\xO\setminus\xS})$, i.e.
%Then
%$$\lim_{n \to \infty}\int_{\partial\xO}K_{\mu}(x,y)f_n(y) \dd\omega_{\Omega \setminus \Sigma}^{x_0}(y) = \int_{\partial\xO}K_{\mu}(x,y)f(y) \dd\omega_{\Omega \setminus \Sigma}^{x_0}(y)$$
%locally uniformly in $\xO\setminus K$. In particular, by taking $x=x_0$ and noting that $K_\mu(x_0,y)=1$, we deduce
\be \label{conver} \lim_{n \to \infty}\int_{\partial\Omega}|h_n(y) - h(y)| \,\dd\omega_{\Omega \setminus \Sigma}^{x_0}(y) = 0.
\ee
This, together with \eqref{Pformula} with $h$ replaced by $|h_n-h|$ and the fact $P_\mu(x_0,\cdot) \in C(\partial \Omega)$, yields
\bal \lim_{n \to \infty}\int_{\partial\Omega}P_\mu(x_0,y)|h_n(y)-h(y)|\,\dd S(y) = \lim_{n \to \infty}\int_{\partial\Omega}|h_n(y) - h(y)|\, \dd\omega_{\Omega \setminus \Sigma}^{x_0}(y) = 0.
\eal
As a consequence, $h_n \to h$ in $L^1(\partial \Omega)$ due to the fact that $\inf_{y \in \partial \Omega}P_\mu(x_0,y)>0$.

Put
\bal u_n(x)=\int_{\partial \Omega} K_\mu(x,y)h_n(y)\,\dd \omega^{x_0}_{\xO\setminus\xS}(y), \quad x \in \Omega \setminus \Sigma.
\eal
By \eqref{conver} and the fact that $K_\mu(\cdot,y)$ is bounded in any compact subset of $\Omega \setminus \Sigma$ (the bound depends on the distance from the compact subset to $\partial \Omega$ and $\Sigma$), we deduce that $u_n \to u$ locally uniformly in $\Omega \setminus \Sigma$ where \bal u(x)=\int_{\partial\xO}K_{\mu}(x,y)h(y) \,\dd\omega_{\Omega \setminus \Sigma}^{x_0}(y).
\eal
Therefore, up to a subsequence, $u_n \to u$  in $\Omega \setminus \Sigma$.

Again, since $K_\mu(x,\cdot), h_n \in C(\partial \Omega)$, by \eqref{Pformula}, we derive
\bal u_n(x)=\int_{\partial \Omega}K_\mu(x,y)P_\mu(x_0,y)h_n(y)\,\dd S(y).
\eal
By Theorem \ref{lpweakmartin1} and \eqref{LpLpweak} and the fact that $0<\max_{y \in \partial \Omega}P_\mu(x_0,y)<\infty$ and $\| h_n\|_{L^1(\partial \Omega)} \leq C\|h\|_{L^1(\partial \Omega)}$, we deduce that for any $1<\kappa<\frac{N+1}{N-1}$, there exists a positive constant $C=C(N,\Omega,\Sigma,\mu, \kappa)$ such that
$\| u_n \|_{L^\kappa(\Omega;\ei)}  \leq C\|h\|_{L^1(\partial \Omega)}$  for all $n \in \N$.
This in turn implies that $\{u_n\}$ is equi-integrable in $L^1(\Omega;\ei)$. Therefore, by Vitali's convergence theorem, up to a subsequence, $u_n \to u$ in $L^1(\Omega;\ei)$.

Next applying \eqref{weak1} with $v=u_n$ and $\zeta=h_n$, we obtain
\be \label{unfn} - \int_{\Omega} u_n L_\mu \eta \,\dd x = - \int_{\partial \Omega} h_n \frac{\partial \eta}{\partial {\bf n}}\,\dd S.
\ee
Since $u_n \to u$ in $L^1(\Omega;\ei)$, $h_n \to h$ in $L^1(\partial \Omega)$ and $|\frac{\partial \eta}{\partial {\bf n}}|$ is bounded on $\partial \Omega$, by letting $n \to \infty$ in \eqref{unfn}, we conclude \eqref{weakfor1}. \medskip

(ii) Let $\{h_n\}$ be a sequence in $C(\partial\xO)$ converging weakly to $\nu$, i.e.
\be  \label{weakcon}  \int_{\partial\xO}\xz h_n \, \dd S\rightarrow \int_{\partial\xO}\xz \dd \nu \quad \forall \zeta \in C(\partial\xO),
\ee
and $\|h_n\|_{L^1(\partial\xO)} \leq C\|\xn\|_{\mathfrak{M}(\prt\xO)}$ for every $n \geq 1$. Put
\bal u_n(x)=\int_{\partial\xO}K_{\mu}(x,y)\frac{h_n(y)}{P_\mu(x_0,y)} \,\dd\omega_{\Omega \setminus \Sigma}^{x_0}(y).
\eal
Since $P_\mu(x_0,\cdot),\;K_{\mu}(x,\cdot)\in C(\partial\Omega)$ and $\inf_{y\in\partial\Omega} P_\mu(x_0,y)>0$, by \eqref{Pformula} and \eqref{weakcon}, we have
\bal u_n(x)=\int_{\partial \Omega}K_{\mu}(x,y)h_n(y)\,\dd S(y)
\rightarrow\int_{\partial\xO}K_{\mu}(x,y)\,\dd \nu(y)=u(x).
\eal
Therefore $u_n \to u$ a.e. in $\Omega \setminus \Sigma$.

On the other hand, by Theorem \ref{lpweakmartin1} and \eqref{LpLpweak}, for any $1<\kappa<\frac{N+1}{N-1},$ there exists a positive constant $C=C(N,\Omega,\Sigma,\mu,\kappa)$ such that
$\norm{u_n}_{L^\kappa(\xO;\ei)}\leq C\norm{\xn}_{\mathfrak{M}(\prt\xO)}$.
By a similar argument as in the proof of (i), we can show that $u_n \to u $  in $L^1(\xO;\ei)$.
Hence by applying \eqref{weak1} with $v=u_n$ and $\zeta=h_n/P_\mu(x_0,\cdot)$, and then letting $n \to \infty$, we conclude \eqref{weakfor2}.
\end{proof}

\subsection{Existence and uniqueness}
We start with a result on the solvability in $L^1$ setting.
\begin{theorem} \label{solL1}
	Assume $\mu \leq H^2$ and $h \in L^1(\partial\xO \cup \Sigma;\dd\omega_{\Omega \setminus \Sigma}^{x_0})$ with compact support in $\partial \Omega$. Then there exists a unique weak solution of \eqref{pro:g-tau=0} and $\dd \nu=h\,\dd \omega^{x_0}_{\xO\setminus\xS}$. Furthermore there holds
	\be
	-\int_{\Omega} u L_\xm\eta \,\dd x+\int_{\Omega} g(u)\eta \, \dd x=-\int_{\partial\xO}\frac{\partial \eta}{\partial{\bf n}}(y)h(y) \, \dd S(y),\quad \forall \eta \in \mathbf{X}_\xm(\xO\setminus \Sigma)\label{weakl1}
	\ee
	and
	\be \label{K=P}
	u+\mathbb{G}_{\mu}[g(u)]=\mathbb{K}_{\xm}[h\,\dd\omega_{\Omega \setminus \Sigma}^{x_0}] =\mathbb{P}_{\xm}[h],
	\ee
	where $\BBP_\mu(x,y)$ is defined in \eqref{BBP}.
\end{theorem}

\begin{proof} The uniqueness is obtained by a similar argument as in the proof of Theorem \ref{existencesubcr}.
	
Next we prove the existence. First we assume that $h \in C(\partial\xO)$ and $h \geq 0$ on $\partial \Omega$.  Let $g_n$ be the function defined in \eqref{gnt} then $g_n\in L^\infty(\mathbb{R})\cap C(\mathbb{R})$. Put $v_h=\BBK_{\mu}[h\,\dd \omega^{x_0}_{\xO\setminus\xS}]$, by Theorem \ref{lpweakmartin1} and \eqref{LpLpweak}, $v_h \in L^1(\Omega; \ei)$. Moreover, by Proposition \ref{measureboundary} and Proposition \ref{Martin}, for $x \in \Omega \setminus \Sigma$,
\be \label{vf-est} \begin{aligned} 0 \leq v_h(x) &= \int_{\partial \Omega}K_{\mu}(x,y)P_\mu(x_0,y)h(y)\,\dd S(y) \\
&\lesssim  \| h \|_{L^\infty(\partial \Omega)} d_\Sigma(x)^{-\am}\int_{\partial \Omega} d(x)|x-y|^{-N} \,\dd S(y) \lesssim d_\Sigma(x)^{-\am}.
\end{aligned} \ee

Since $v_h$ and $0$ are supersolution and subsolution of \eqref{BVP-O} with $g=g_n$ and $\dd\nu= h\,\dd\omega_{\Omega \setminus \Sigma}^{x_0}$ and $0$ respectively, by Theorem \ref{existencesubcr}, there exists a unique weak solution $u_n\in L^1(\Omega; \ei)$ of
\be\label{NLin2} \left\{ \BAL
- L_\mu u+g_n(u)&=0\qquad \text{in }\;\Gw\setminus \Sigma,\\
\tr(u)&=h\,\dd\omega_{\Omega \setminus \Sigma}^{x_0},
\EAL \right. \ee
such that $0\leq u_n\leq v_h$ in $\Omega \setminus \Sigma.$ By Proposition \ref{weaksolution0} (i), $u_n$ satisfies
\be \label{weakl2}
-\int_{\Omega} u_n L_\xm\eta \, \dd x+\int_{\Omega} g_n(u_n)\eta \, \dd x=- \int_{\Omega} v_h L_{\xm }\eta \,\dd x = - \int_{\partial \Omega} \frac{\partial \eta}{\partial {\bf n} }h \, \dd S, \quad \forall \eta \in \mathbf{X}_\xm(\xO\setminus \Sigma).
\ee

By applying \eqref{poi4} with $\zeta=\ei$, $f=-g_n(u_n)$, $\rho=0$, $\dd \nu=h\,\dd \omega^{x_0}_{\xO\setminus\xS}$ and using Theorem \ref{lpweakmartin1} and \eqref{LpLpweak}, we assert that
\be\label{unest1}
\| u_n\|_{L^1(\Omega;\ei)} +\| g_n(u_n) \|_{L^1(\Omega;\ei)} \lesssim \| h \|_{L^1(\partial \Omega \cup \Sigma;\dd \omega^{x_0}_{\xO\setminus\xS})}.
\ee

Owing to standard local regularity, $\{u_n\}$ is uniformly bounded in $W^{1,\kappa}(D)$ for any $1<\kappa<\frac{N}{N-1}$ and any open $D\Subset \Omega \setminus \Sigma$. By a compact embedding, there exist a subsequence, say $\{u_n\}$, and a nonnegative function $u$ such that $u_n \to u$ a.e. in $\Omega \setminus \Sigma$. Since $|u_n|\leq v_h \in L^1(\Omega;\ei)$, by the dominated convergence theorem we have that $u_n \to u \in L^1(\Omega;\ei)$.  We also note that  $g_n(u_n) \to g(u)$ and $0 \leq g_n(u_n) \leq g(v_h)$ a.e. in $\Omega \setminus \Sigma$. From \eqref{vf-est}, we see that $g(v_h) \in L^1(\Omega \setminus \Sigma_\beta;\ei)$ for every $\beta \in (0,\beta_0)$. Therefore, by the dominated convergence theorem, we derive  $g_n(u_n)\to g(u)$ in $L^1(\Omega \setminus \Sigma_\beta;\ei)$ for every $\beta \in (0,\beta_0)$. By \eqref{unest1} and Fatou's lemma, $g(u)\in  L^1(\Omega;\ei)$. In addition, by letting $n \to \infty$ in \eqref{weakl2}, we derive that \eqref{weakl1} holds true
for all $\eta \in \mathbf{X}_\xm(\xO\setminus \Sigma)$ with $\supp \eta \Subset \overline \Omega \setminus \Sigma$.

We note that $u+\mathbb{G}_{\mu}[g(u)]$ is a nonnegative $L_\mu$-harmonic function in $\xO\setminus \Sigma$, hence by Theorem \ref{th:Rep}, there exists a unique measure $\nu \in \mathfrak{M}^+(\partial\xO\cup \Sigma)$ such that
\be \label{urep} u+\mathbb{G}_{\mu}[g(u)]=\mathbb{K}_{\xm}[\xn].
\ee
This, combined with the fact that $g(u) \in L^1(\Omega;\ei)$ and Proposition \ref{traceKG}, implies $\tr(u)=\nu$.

By choosing $\phi \in C(\overline \Omega)$ such that $0 \leq \phi \leq 1$ in $\overline \Omega$,  $\phi=0$ in $\overline \Omega_{\beta_0}$ and $\phi=1$ in $\overline \Sigma_{\beta_0}$ in Definition \ref{nomtrace}, we deduce
\be \label{contra1} \lim_{n \to \infty}\int_{\partial \Sigma_n} u \,\dd \omega_{O_n}^{x_0} = \int_{\Sigma}\dd \nu=\nu(\Sigma).
\ee
Here we choose the sequence $\{\Sigma_n\}$ such that $\dist(\Sigma_n,\Sigma)=\frac{1}{n}$.

Next we show that $\nu$ has compact support in $\partial \Omega$. Suppose by contradiction that $\nu(\Sigma)>0$. If $\mu<H^2$, then from the estimate $u(x) \leq v_h(x)\leq C d_\Sigma(x)^{-\am}$ for any $x\in\Omega \setminus \Sigma$, the definition of $\tilde W$ in \eqref{tildeW} and  \cite[Proposition 6.12]{GkiNg_linear} (with $\phi$ chosen as above) and \eqref{contra1}, we have
\be \label{contra2} \begin{aligned} \int_{\Sigma}\, \dd \omega^{x_0}_{\xO\setminus\xS}(x)&=\lim_{n \to \infty}\int_{\partial \Sigma_n} d_\Sigma(x)^{-\ap} \, \dd \omega_{O_n}^{x_0}(x) \\
&= \lim_{n \to \infty} n^{\ap-\am} \int_{\partial \Sigma_n}d_\Sigma(x)^{-\am} \, \dd \omega_{O_n}^{x_0}(x) \\
&\gtrsim \lim_{n \to \infty} n^{\ap-\am} \int_{\partial \Sigma_n} u(x) \, \dd \omega_{O_n}^{x_0}(x) = +\infty,
\end{aligned} \ee
which yields a contradiction since $\omega^{x_0}_{\xO\setminus\xS} \in \GTM^+(\partial \Omega \cup \Sigma)$ (note that $\ap - \am>0$). If $\mu=H^2$ then by a similar argument, we obtain
\bal \int_{\Sigma}\dd \omega^{x_0}_{\xO\setminus\xS}(x) &=\lim_{n \to \infty}\int_{\partial \Sigma_n} d_\Sigma(x)^{-H} |\ln d_\Sigma(x)| \, \dd \omega_{O_n}^{x_0}(x) \\
&= \lim_{n \to \infty} \ln(n) \int_{\partial \Sigma_n}d_\Sigma(x)^{-H} \, \dd \omega_{O_n}^{x_0}(x) \\
&\gtrsim \lim_{n \to \infty} \ln(n) \,\nu(\Sigma) = +\infty,
\eal
which is a contradiction.
Therefore $\nu$ has compact support in $\partial \Omega$.

Since $u$ satisfies \eqref{urep}, by using Proposition \ref{weaksolution0} (ii), we obtain
\be \label{weakl4}
-\int_{\Omega} u L_\xm\eta \, \dd x+\int_{\Omega} g(u)\eta \, \dd x=  - \int_{\Omega}\BBK_{\mu}[\nu] L_\mu \eta \, \dd x = -\int_{\partial\xO}\frac{\partial \eta}{\partial{\bf n}}(y)\frac{1}{P_\xm(x_0,y)}\, \dd \nu(y),
\ee
for all $\eta \in \mathbf{X}_\xm(\Omega \setminus \Sigma)$.
Combining \eqref{weakl1} (which holds for all $\eta \in \mathbf{X}_\xm(\xO\setminus \Sigma)$ with $\supp \eta \Subset \overline \Omega \setminus \Sigma$) and \eqref{weakl4} yields
\be \label{comparePf}
-\int_{\partial\xO}\frac{\partial \eta}{\partial{\bf n}}(y)\frac{1}{P_\xm(x_0,y)}\, \dd \nu(y)=-\int_{\partial\xO}\frac{\partial \eta}{\partial{\bf n}}(y)h(y)\, \dd S(y),
\ee
for all $\eta \in {\bf X}_\mu(\Omega \setminus \Sigma)$ with $\supp \eta \Subset \overline \Omega \setminus \Sigma$.

Let $\eta \in {\bf X}_\mu(\Omega \setminus \Sigma)$ and $\xf$ be the cut-off function above \eqref{contra1}. Using the test function $\tilde\eta=(1-
\xf)\eta$  in \eqref{comparePf}, we can show that \eqref{comparePf} holds for all $\eta \in {\bf X}_\mu(\Omega \setminus \Sigma)$. This in turn implies that \eqref{weakl1} holds for any $\eta \in {\bf X}_\mu(\Omega \setminus \Sigma)$. Combining \eqref{weakl1} and Proposition \ref{weaksolution0} (i), we deduce that
\bal
-\int_{\Omega} u L_\xm\eta \, \dd x+\int_{\Omega} g(u)\eta \, \dd x=- \int_{\Omega} \BBK_{\mu}[h\dd \omega^{x_0}_{\xO\setminus\xS}] L_{\xm }\eta \,\dd x,
\eal
which means  $u$ is a weak  solution of \eqref{BVP-O} with $\dd \nu=h\dd \omega^{x_0}_{\xO\setminus\xS}$.

Next we still assume that $h \in C(\partial \Omega)$, but drop  the assumption that $h \geq 0$ on $\partial \Omega$. Let $u_n$ and $\tilde u_n$ are weak solutions of \eqref{NLin2} with boundary datum $h\dd \omega^{x_0}_{\xO\setminus\xS}$ and $|h|\dd \omega^{x_0}_{\xO\setminus\xS}$ respectively. Then by \eqref{poi5}, $|u_n| \leq \tilde u_n$ in $\Omega \setminus \Sigma$. Moreover, by local regularity results, $\{u_n\}$ is uniformly bounded in $W^{1,\kappa}(D)$ for any $1<\kappa<\frac{N}{N-1}$ and $D \Subset \Omega \setminus \Sigma$. By the compact embedding, up to a subsequence, $u_n \to u$ a.e. in $\Omega \setminus \Sigma$. As a consequence, $g_n(u_n) \to g(u)$ a.e. in $\Omega \setminus \Sigma$ and $|g_n(u_n)| \leq g_n(\tilde u_n)-g_n(-\tilde u_n)$ a.e. in $\Omega \setminus \Sigma$. Therefore $u_n \to u$ and $g_n(u_n) \to g(u)$ in $L^1(\Omega;\ei)$. Consequently $u$ is a weak  solution of \eqref{BVP-O} with $\dd \nu=h\dd \omega_{\Omega \setminus \Sigma}^{x_0}$.

If $h \in L^1(\partial \Omega; \dd \omega_{\Omega \setminus \Sigma}^{x_0})$, let $\{h_n\}\subset C(\partial\Omega)$ such that $h_n \to h$ in $L^1(\partial\Omega;\dd \omega_{\Omega \setminus \Sigma}^{x_0})$ and $u_n$ be the respective solution with boundary datum $h_n \dd \omega_{\Omega \setminus \Sigma}^{x_0}$. By \eqref{poi4}, Theorem \ref{lpweakmartin1} and \eqref{LpLpweak}, there exists a positive constant $C$ such that
\be \label{un-um} \| u_n -u_l \|_{L^1(\Omega;\ei)} + \| g(u_n) - g(u_l) \|_{L^1(\Omega;\ei)} \leq C \| h_n-h_l \|_{L^1(\partial \Omega; \dd \omega_{\Omega \setminus \Sigma}^{x_0})}.
\ee
This implies that $\{u_n\}$ and $\{g(u_n)\}$ are Cauchy sequences in $L^1(\Omega;\ei)$, hence there exists $u\in L^1(\Omega;\ei)$ such that $u_n \to u$  and $g(u_n) \to g(u)$ in $L^1(\Omega;\ei)$. Thus $u$ is a weak solution of \eqref{BVP-O} with $\dd \nu=h\dd \omega_{\Omega \setminus \Sigma}^{x_0}$.

Formula \eqref{weakl1} follows from  formula \eqref{nlinearweakform} with $\dd \nu=h\dd \omega_{\Omega \setminus \Sigma}^{x_0}$ and Proposition \ref{weaksolution0} (i).

The first equality in \eqref{K=P} follows from \eqref{reprweaksol} with $\dd \nu=h\dd \omega_{\Omega \setminus \Sigma}^{x_0}$. The second equality in \eqref{K=P} follows from Proposition \ref{KP}.
\end{proof}

\begin{proof}[\textbf{Proof of Theorem \ref{measureO}}.] 	
Put $U_1=-\BBK_{\mu}[\nu^-]$ and $U_2=\BBK_{\mu}[\nu^+]$. Then by Theorem \ref{lpweakmartin1}, $U_1, U_2 \in L^1(\Omega;\ei)$. Moreover, from Theorem \ref{lpweakmartin1} and Lemma \ref{subcrcon} with $m=0$ and $q=\frac{N+1}{N-1}$, we have $g(U_1),$ $g(U_2) \in L^1(\Omega;\ei)$. We also note that $U_1$ and $U_2$ are subsolution and supersolution with $U_1 \leq 0 \leq U_2$. By applying Theorem \ref{existencesubcr}, we deduce that there exists a unique weak solution $u$ of \eqref{BVP-O} which satisfies \eqref{K<u<K}.
\end{proof}

\section{Boundary data concentrated in $\xS$} \label{sec:BVP-Sigma}
In this Section, we consider the case where the measure data are concentrated in $\Sigma$. Below is a regularity result in weak Lebesgue spaces.
\begin{lemma}\label{weaklpsmooth2}
	Assume $1 \leq k <N-2$ and $S_\Sigma$ is the $k$-dimensional surface measure on $\Sigma$.
	
	\noindent (i) If $\mu<H^2$ then $\mathbb{K}_{\mu}[S_\Sigma]\in L_w^{\frac{N-k-\am}{\ap}}(\Gw\setminus \Sigma;\ei)$.
	
	\noindent (ii) If $\mu=H^2$ then $\mathbb{K}_{\mu}[S_\Sigma]\in L_w^{\theta}(\Gw\setminus \Sigma;\ei)$ for all $1<\theta<\frac{N-k+2}{N-k-2}$. In addition, for $\lambda>0$, set
	\ba
	\tilde{A}_\xl(0):=\Big\{x\in \xO\setminus \{0\}:\;  \mathbb{K}_{\mu}[S_\Sigma](x)>\xl \Big \}, \quad \tilde{m}_{\xl}&:=\int_{\tilde{A}_\xl(0)}d(x)|x|^{-\frac{N-2}{2}} \, \dd x.\label{69b}
	\ea
	Then
	\ba\label{54b}
	\tilde{m}_{\xl}\lesssim (\xl^{-1}\ln\xl)^{\frac{N+k+2}{N+k-2}}, \quad \forall \xl>e.
	\ea
	The implicit constant depends on $N,\Omega,\Sigma,\mu$ and $\theta$.
\end{lemma}

\begin{proof}
	By \eqref{Martinest1}, we have, for $x \in \Omega \setminus \Sigma$,
	\ba \label{martinKest1}
	\BBK_{\mu}[S_\Sigma](x) = \int_{\Sigma}K_{\mu}(x,y)\dS_\Sigma(y)\lesssim d_\Sigma(x)^{-\am}\int_{\Sigma}|x-y|^{-(N-2-2\am)}\dS_\Sigma(y).
	\ea
	%\noindent \textbf{Case 1: $x\in \xO\setminus \Sigma_{\frac{\beta_1}{2}}$.} Then $|x-y| \geq \beta_1/2$ for any $y \in \Sigma$ and hence
	%\bel{martinKest2}
	%\int_{K}|x-y|^{-(N-2-2\am)}\dS_\Sigma(y) \lesssim 1.
	%\ee
	%
	%\noindent \textbf{Case 2: $x\in \Sigma_{\frac{\beta_1}{2}}$.} Then there exists  $\xi \in \Sigma$ such that $|x-\xi|=d_\Sigma(x) \leq \beta_1/2$.
	%
	%We note that $|y-\xi| \geq \beta_1$ for all $y \in \Sigma \setminus V(\xi,\beta_0)$ and consequently by triangle inequality, $|x-y| \geq \beta_1/2$. Therefore
	%\be
	%\int_{ \Sigma \setminus V(\xi,\beta_0)}|x-y|^{-(N-2-2\am)}\dS_\Sigma(y)\lesssim 1.\label{martinKest3}
	%\ee
	%
	%We observe that $x \in V(\xi,\beta_0)$. By \eqref{straigh} and \eqref{propdist}, $|x'-y'| \leq 2\beta_0$ and $\delta_\Sigma^\xi(x) \lesssim |x-y| $ for every $y \in V(\xi,\beta_0)$ (We recall that $\delta_\Sigma^\xi$ is defined in \eqref{dist2}). As a consequence,
	%\begin{align}\nonumber
	%\int_{\Sigma\cap V(\xi,\beta_0)}|x-y|^{-(N-2-2\am)}\dS_\Sigma(y)&\lesssim \int_{|x'-y'|\leq 2\beta_0}\left(|x'-y'|+\delta_\Sigma^\xi(x)\right)^{-(N-2-2\am)}dy'\\ \nonumber
	%&\lesssim \int_{|x'-y'|\leq 2\beta_0}\left(|x'-y'|+d_\Sigma(x)\right)^{-N+2\am+2}dy'\\
	%&\leq \int_0^{2\beta_0}(r+d_\Sigma(x))^{k-N+2\am+1}dr.\label{martinKest4}
	%\end{align}
	
	(i) If $\mu <H^2$ then $\am<H$.
	% and from \eqref{martinKest4}, we get
	%$$ \int_{\Sigma\cap V(\xi,\beta_0)}|x-y|^{-(N-2-2\am)}\dS_\Sigma(y) \lesssim d_\Sigma(x)^{2\am-(N-k-2)}.
	%$$
	From \eqref{martinKest1}, we obtain $\BBK_{\mu}[S_\Sigma] \lesssim d_\Sigma^{-\ap}$ in $\Omega \setminus \Sigma$. Then we can proceed as in the proof of \cite[Theorem 3.5 (i)]{GkiNg_source} to derive
	$\mathbb{K}_{\mu}[S_\Sigma]\in L_w^{\frac{N-k-\am}{\ap}}(\Gw\setminus \Sigma;\ei).
	$
	
	(ii) If $\mu=H^2$ then $\am=H$. From \eqref{martinKest1} we can show that $\BBK_{\mu}[S_\Sigma] \lesssim d_\Sigma^{-H}|\ln \frac{d_\Sigma}{\CD_\xO}|$, where $\CD_\Omega = 2\sup_{x \in \Omega}|x|$.
	%Set
	%\begin{align*}
	%A_\xl:=\Big\{x\in \xO\setminus \Sigma:\;\; \BBK_{\mu}[S_\Sigma](x)>\xl \Big \} \quad \text{and} \quad m_{\xl}:=\int_{A_\xl(y)}d(x)d_\Sigma(x)^{-H} \dx.
	%\end{align*}
	%
	%Then by \eqref{martinKest5} we have
	%$$A_\lambda \subset \Big\{x\in(\xO\setminus \Sigma):\;\; d_\Sigma(x)^{-H}|\ln \frac{d_\Sigma(x)}{\CD_\xO}|> c\lambda \Big \}=:\tilde A_\lambda$$
	%Therfore
	%\begin{align} \label{martinKest6}
	% m_{\xl}\lesssim \int_{\tilde A_\xl}d_\Sigma(x)^{-H} \dx \lesssim  \int_{\{d_\Sigma(x)\leq c\lambda^{-1}|\ln \lambda|)^{-\frac{1}{H}} \}}d_\Sigma(x)^{-H} \dx
	% \lesssim (\xl^{-1}|\ln\xl|)^{\frac{N-k+2}{N-k-2}}.
	%\end{align}
	%It follows that, for any $1<p<\frac{N-k+2}{N-k-2}$,
	%\bel{martinKest7} m_\lambda \leq C \lambda^{-p}
	%\ee
	%for all $\lambda \geq 1$. Then we can show that \eqref{martinKest7} holds true for all $\lambda>0$. Thus by  the definition of the $L^p$ weak space and \eqref{eigenfunctionestimates},
	Then by proceeding as in the proof of \cite[Theorem 3.6]{GkiNg_source}, we may obtain the desired result.
\end{proof}

\begin{theorem} \label{bdr-Sigma-1} (i) Assume $\mu < H^2$ and $g$ satisfies \eqref{subcd0} with $q= \frac{N-k-\am}{\ap}$ and $m=0$. Then for any $h \in L^1(\partial \Omega \cup \Sigma;\dd S_\Sigma)$ with compact support in $\Sigma$, problem \eqref{BVP-O} with $\dd \nu=h\, \dd S_\Sigma$ admits a unique weak solution.
	
	(ii) Assume $\mu =H^2$ and $g$ satisfies \eqref{subcd0} with $q=m= \frac{N+k+2}{N-k-2}.$ Then for any $h \in L^1(\partial \Omega \cup \Sigma;\dd S_\Sigma)$ with compact support in $\Sigma$, problem \eqref{BVP-O} with $\dd \nu=h\,\dd S_\Sigma$ admits a unique weak solution.
\end{theorem}

\begin{proof}
	Let $h \in L^1(\partial \Omega \cup \Sigma;\dd S_\Sigma)$ with compact support in $\Sigma$. Let $\{h_n\}\subset L^\infty(\partial \Omega \cup \Sigma)$ with compact support in $\Sigma$ be such that $h_n\rightarrow h$ in $L^1(\Sigma;\dd S_\Sigma)$. For each $n$, set $U_{n,1}=-\BBK_{\mu}[(h_n)^-]$ and $U_{n,2}=\BBK_{\mu}[(h_n)^+]$.
	
	(i) Assume $\mu < H^2$ and $g$ satisfies \eqref{subcd0} with $q= \frac{N-k-\am}{\ap}$ and $m=0$. For $i=1,2$, by Lemma \ref{weaklpsmooth2}, \eqref{ue} and Lemma \ref{subcrcon} for $q=\frac{N-k-\am}{\ap}$ and $m=0,$ we have $g(U_{n,i}) \in L^1(\Omega; \ei)$, $i=1,2$. Moreover, we see that $U_{n,1}$ and $U_{n,2}$ are respectively subsolution and supersolution of \eqref{BVP-O} with $\nu=h_n$  with $U_{n,1} \leq U_{n,2}$ in $\Omega \setminus \Sigma$. Therefore, by Theorem \ref{existencesubcr}, there exists a unique solution $u_n$ of \eqref{BVP-O} with $\nu=h_n$ which satisfies $U_{n,1} \leq u_n \leq U_{n,2}$ in $\Omega \setminus \Sigma$.  Furthermore $|u_n|^p \in L^1(\Omega;\ei)$ and there holds
	\be \label{nlinearweakform11}
	- \int_{\Omega}u_n L_{\xm }\zeta \, \dd x+ \int_{\Omega} |u_n|^{p-1}u_n \zeta \, \dd x=\int_{\Omega \setminus \Sigma} \zeta \, \dd \tau - \int_{\Omega} \BBK_{\mu}[h_n]L_{\xm }\zeta \, \dd x,
	\quad\forall \zeta \in\mathbf{X}_\xm(\xO\setminus \Sigma).
	\ee
	
	In addition, by using a similar argument leading to \eqref{un-um} and Proposition \ref{weaklpsmooth2}, we can show that
	there exists a positive constant $C$ such that
	\bal
	\| u_n -u_l \|_{L^1(\Omega;\ei)} + \| g(u_n) - g(u_l) \|_{L^1(\Omega;\ei)} \leq C \| h_n-h_l \|_{L^1(\Sigma; \dd S_\Sigma)}.
	\eal
	The result follows by using the above inequality and argument following \eqref{un-um}. \medskip
	
	The proof of (ii) is similar and we omit it.
\end{proof}

Similarly we can show that
\begin{theorem} \label{bdr-Sigma-2} (i) Assume $\mu < H^2$ and $g$ satisfies \eqref{subcd0} with $q= \frac{N-k-\am}{\ap}$ and $m=0$. Then for any $h \in L^1(\partial \Omega \cup \Sigma;\xo^{x_0}_{\xO\setminus \xS})$ with compact support in $\Sigma$, problem \eqref{BVP-O} with $\dd \nu=h\,\dd \xo^{x_0}_{\xO\setminus \xS}$ admits a unique weak solution.
	
	(ii) Assume $\mu =H^2$ and $g$ satisfies \eqref{subcd0} with $q=m= \frac{N+k+2}{N-k-2}.$ Then for any $h \in L^1(\partial \Omega \cup \Sigma;\xo^{x_0}_{\xO\setminus \xS})$ with compact support in $\Sigma$, problem \eqref{BVP-O} with $\dd \nu=h\,\dd \xo^{x_0}_{\xO\setminus \xS}$ admits a unique weak solution.
\end{theorem}
\begin{proof}
	By \cite[Lemma 5.6]{GkiNg_linear}, we have that
	\bal
	\BBK_{\mu}[\xo^{x_0}_{\xO\setminus \xS}]\lesssim \left\{
	\BAL
	&d_\Sigma^{-\ap} &&\quad \text{if}\; \mu < H^2,\\
	&d_\Sigma^{-H}|\ln \frac{d_\Sigma}{\CD_\xO}| &&\quad\text{if}\; \mu =H^2.
	\EAL\right.
	\eal
	By the same arguments as in the proof of Theorem \ref{bdr-Sigma-1}, we may deduce the desired result.
\end{proof}

\begin{proof}[\textbf{Proof of Theorem \ref{exist-measureK}}.] (i)  The proof is similar to that of Theorem \ref{measureO} with some minor modification and hence we omit it.
	
(ii) Without loss of generality we assume that $\xn\geq0.$ Put $U_1=0$ and $U_2=\BBK_{\mu}[\nu]$. By \eqref{54a} and Lemma \ref{subcrcon} with $q=m=\frac{N+2}{N-2}$, we have that $g(U_2)\in L^1(\Omega;\ei).$ Proceeding as in the proof of Theorem \ref{measureO}, we can obtain the desired result.
\end{proof}

\section{Keller-Osserman estimates in the power case}
\label{sec:KOestimate}
In this Section, we prove Keller-Osserman type estimates on nonnegative solutions to equations with a power nonlinearity.
\begin{lemma} \label{lem:KO-1} Assume $p>1$. Let $u\in C(\overline{\xO}\setminus \Sigma)$ be a nonnegative solution of
	\be \label{eq:power1}
	-L_\mu u+|u|^{p-1}u=0
	\ee
	 in the sense of distributions in $\Omega \setminus \Sigma$. Assume that
	 \be \label{zero-bdr} \lim_{x\in\xO,\;x\to\xi}u(x)=0,\quad\forall \xi \in \partial\xO.
	 \ee
	Then there exists a positive constant $C=C(\xO,\Sigma,\xm,p)$ such that
	\be
	0 \leq u(x) \leq Cd(x)d_\Sigma(x)^{-\frac{2}{p-1}},\quad\forall x\in\xO\setminus \Sigma.\label{ko}
	\ee
\end{lemma}
\begin{proof}
Let $\beta_0$ be as in Subsection \ref{assumptionK} and  $\eta_{\beta_0} \in C_c^\infty(\R^N)$ such that
\bal 0\leq\eta_{\xb_0}\leq1, \quad \eta_{\xb_0}=1 \text{ in } \overline{\Sigma}_{\frac{\xb_0}{4}} \quad \text{and} \quad  \supp(\eta_{\beta_0}) \subset \Sigma_{\frac{\xb_0}{2}}.
\eal
Let $\varepsilon \in (0,\frac{\xb_0}{16})$, we define
\bal V_\xe:=1-\eta_{\xb_0}+\eta_{\xb_0}(d_\Sigma-\xe)^{-\frac{2}{p-1}} \quad \text{in } \overline \Omega \setminus \overline{\Sigma}_\varepsilon.
\eal
Then $V_\varepsilon \geq 0$ in $\overline \Omega \setminus \overline{\Sigma_\epsilon}$. It  can be checked that there exists  $C=C(\xO,\Sigma,\xb_0,\xm,p)>1$ such that the function $W_\xe:=CV_\xe$ satisfies
\be \label{Wep} -L_\xm W_\xe+W_\xe^p = C(-L_\mu V_\varepsilon + V_\varepsilon^p) \geq 0 \quad \text{in } \xO\setminus \overline{\Sigma}_\xe.
\ee
Since $u\in C(\xO\setminus \Sigma)$ is a nonnegative solution of equation \eqref{eq:power1}, by standard regularity results, $u\in C^2(\Omega \setminus \Sigma)$. Combining \eqref{eq:power1} and \eqref{Wep} yields
\be \label{u-Wep} -L_\mu (u-W_\varepsilon) + u^{p} - W_\varepsilon^p \leq 0 \quad \text{in } \Omega \setminus \overline{\Sigma}_\varepsilon.
\ee
We see that $(u-W_\xe)^+\in H_0^1(\xO\setminus\overline{\Sigma}_\xe)$ and  $(u-W_\xe)^+$ has compact support in $\xO\setminus\overline{\Sigma}_\xe$. By using $(u-W_\varepsilon)^+$ as a test function for \eqref{u-Wep}, we deduce that
\bal
	0 &\geq \int_{\xO\setminus \Sigma_\xe}|\nabla (u-W_\xe)^+|^2 \dd x - \xm\int_{\xO\setminus \Sigma_\xe}\frac{[(u-W_\xe)^+]^2}{d_\Sigma^2} \dd x+\int_{\xO\setminus \Sigma_\xe}(u^{p}-W_\xe^p)(u-W_\varepsilon)^+ \dd x\\
	&\geq \int_{\xO\setminus \Sigma_\xe}|\nabla (u-W_\xe)^+|^2 \dd x - \xm\int_{\xO\setminus \Sigma_\xe}\frac{[(u-W_\xe)^+]^2}{d_\Sigma^2} \dd x \geq \xl_\xm\int_{\xO\setminus \Sigma_\xe}| (u-W_\xe)^+|^2 \dd x.
\eal
	This and the assumption $\lambda_\mu >0$ imply $(u-W_\xe)^+ = 0$, whence $u\leq W_\xe$ in $\xO\setminus\overline{\Sigma}_\xe$. Similarly we can show that $-W_\xe\leq u$ in $\xO\setminus\overline{\Sigma}_\xe$. Thus $u \leq W_\varepsilon$ in $\xO\setminus\overline{\Sigma}_\xe$. Letting $\xe\to0,$ we obtain
	\be \label{KO-1}
	u \leq Cd_\Sigma^{-\frac{2}{p-1}} \quad \text{in } \xO\setminus \Sigma.
	\ee
	Let $0<\xd_0<\frac{1}{4}\dist(\partial\xO,\Sigma).$ Then by \eqref{KO-1}, $u\leq C(\xd_0,p)$ in  $\xO_{\xd_0}$.
	As a consequence, by  standard elliptic estimates, there exists a constant $C$ depending only on $\xd_0$ and the $C^2$ characteristic of $\xO$ such that
	\be \label{KO-2}
	u\leq Cd \quad \text{in } \xO_{\xd_0}.
	\ee
	Combining \eqref{KO-1} and \eqref{KO-2} gives \eqref{ko}.
\end{proof}

In case of lack of boundary condition on $\partial \Omega$, by adapting the above argument, we can show that
$u \leq Cd^{-\frac{2}{p-1}}$ in  $\Omega_{\delta_0}$.
Combining \eqref{KO-1} and \eqref{KO-2} leads to the following result whose proof is omitted.
\begin{lemma}
	Let $u\in C(\overline{\xO}\setminus \Sigma)$ be a nonnegative solution of \eqref{eq:power1} in the sense of distributions in $\Omega$. Then there exists a positive constant $C=C(\xO,\Sigma,\xm,p)$ such that
	\be \label{ko1}
	u(x)\leq C\left(\min\{d(x),d_\Sigma(x)\}\right)^{-\frac{2}{p-1}},\quad\forall x\in\xO\setminus \Sigma.
	\ee
\end{lemma}

\section{Removable singularities} \label{sec:removable}
In this Section, we show that singularities are removable in supercritical cases.
\begin{proof}[\textbf{Proof of Theorem \ref{remov-1}}.]
	Assume $\xm<H^2$ and $p=\frac{2+\ap}{\ap}$. Let $u$ be a nonnegative solution of \eqref{eq:power-a} satisfying \eqref{as1intro}. Denote $O_n=\xO\setminus \overline{\Sigma}_{\frac{1}{n}}$ and
	\bal V(x)=2C\diam(\xO)\int_{\Sigma}K_{\mu}(x,y)\dd\omega_{\Omega \setminus \Sigma}^{x_0}(y) = 2C\diam(\Omega)\BBK_\mu[\1_{\Sigma}\omega_{\Omega \setminus \Sigma}^{x_0}](x),
	\eal
	where $C$ is the constant in \eqref{ko}. Then by \cite[estimate (5.29)]{GkiNg_linear}, there exists $\tilde \beta>0$ such that
\be \label{V>C} V(x)\geq C\diam(\xO)d_\Sigma(x)^{-\ap} \quad\forall x\in \Sigma_{\tilde \beta}.
\ee
	
	Let $ n_0\in\mathbb{N}$ be large enough such that $\frac{1}{n}\leq \frac{\tilde \beta}{2}$ for any $n\geq n_0$. Let $v_n$ be the solution of
	\be\label{suba} \left\{ \BAL
	-L_{\xm }^{O_n}v_n + v_n^p&=0\qquad&&\text{in } O_n\\
	v_n&=0\qquad&&\text{on } \prt \xO,\\
	v_n&=V\qquad&&\text{on } \prt \Sigma_\frac{1}{n}.
	\EAL \right.
	\ee
	Then by \eqref{ko}, we have that $0 \leq u\leq v_n$ in $O_n$. Furthermore, $\{v_n\}$ is a non-increasing sequence. Let $G_\mu^{O_n}$ and $P_\mu^{O_n}$ be the Green function and Poisson kernel of $-L_\mu$ in $O_n$. Denote by $\BBG_\mu^{O_n}$ and $\BBP_\mu^{O_n}$ the corresponding Green operator and Poisson operator. We extend $V$ by zero on $\partial \Omega$ and use the same notation for the extension. Then, we deduce from \eqref{suba} that
	\be \label{vnexpr} v_n+\BBG_\xm^{O_n}[v_n^p]=\BBP_\mu^{O_n}[V] = V \quad \text{in } O_n.
	\ee
	This implies $v_n\leq V$ in $O_n$ for any $n \in \N$. Therefore $v_n \downarrow v$ locally uniformly and in $L^1(\xO;\ei)$. Using the fact that $ G_\mu^{O_n}\uparrow G_\mu$ and Fatou's Lemma, by letting $n \to \infty$ in \eqref{vnexpr}, we obtain
	$v+\mathbb{G}_\xm[v^p]\leq V$ in $\xO\setminus \Sigma$, which implies that $v\in L^p(\xO;\ei)$.
	
	Since $v+\mathbb{G}_\xm[v^p]$ is a nonnegative $L_\xm$ harmonic in $\xO\setminus \Sigma$, by the Representation Theorem \ref{th:Rep} and the fact that $v+\mathbb{G}_\xm[v^p]\leq V$, there exists  $\xn\in \GTM^+(\partial\xO\cup \Sigma)$ with compact support in $\Sigma$ such that
	\be \label{vGK} v+\mathbb{G}_\xm[v^p]=\BBK_{\mu}[\xn] \quad \text{in } \Omega \setminus \Sigma.
	\ee
	
Let $\tilde O_n=\xO_n\setminus \Sigma_n$ be a smooth exhaustion of $\Gw\setminus \Sigma$. We denote by  $\tilde v_n$ the solution of
	\be \label{tilvn} \left\{ \BAL
	-L_{\xm }^{\tilde O_n}\tilde v_n + \tilde v_n^p&=0\qquad&&\text{in } \tilde O_n\\
	\tilde v_n&=2v\qquad&&\text{on } \prt \tilde O_n.
	\EAL \right.
	\ee
	Then $\tilde v_n\leq 2v\leq 2V$ in $\tilde O_n,$ since $2v$ is a supersolution of \eqref{tilvn}. Hence, there exist a function $\tilde v$ and a subsequence, still denoted by $\{\tilde v_n\}$, such that $\tilde v_n\to \tilde v$ a.e. in $\xO\setminus \Sigma$. Let $G_\mu^{\tilde O_n}$ and $P_\mu^{ \tilde O_n}$ be the Green function and Poisson kernel of $-L_\mu$ in $\tilde O_n$. Denote by $\BBG_\mu^{ \tilde O_n}$ and $\BBP_\mu^{ \tilde O_n}$ the corresponding Green operator and Poisson operator.  From \eqref{tilvn}, we have that
	\bel{tildeu-express} \tilde v_n+\mathbb{G}_\xm^{\tilde O_n}[\tilde v_n^p]=2\BBP_\mu^{\tilde O_n}[v] \quad \text{in }  \tilde O_n.
	\ee
By \eqref{vGK}, we obtain
\bal
\BBP_\mu^{\tilde O_n}[v](x)=\int_{\partial\tilde O_n} v \, \dd x\xo^x_{\tilde O_n}=-\int_{\partial \tilde O_n}\BBG_\mu[v^p] \, \dd \xo^x_{\tilde O_n}+\BBK_\mu[\nu](x).
\eal
Since $\tr(\BBG_\mu[v^p])=0$ (see Proposition \ref{traceKG}), we derive from Definition \ref{nomtrace} and the above expression that $\BBP_\mu^{\tilde O_n}[v] \to \BBK_\mu[\nu]$ a.e. in $\Omega \setminus \Sigma$. Since $\tilde v_n \leq 2v \in L^p(\Omega;\ei),$  by dominated convergence theorem, we have $\mathbb{G}_\xm^{\tilde O_n}[\tilde v_n^p] \to \BBG_\mu[\tilde v^p]$ in $\Omega \setminus \Sigma$.  Letting $n \to \infty$ in \eqref{tildeu-express} yields
	\bal \tilde v+\mathbb{G}_\xm[\tilde v^p]= 2\BBK_{\mu}[\xn]\quad \text{in } \xO\setminus \Sigma.
	\eal

On the other hand, since $0 \leq \tilde v\in C^2(\xO\setminus \Sigma)$ satisfies $-L_\xm \tilde v+ \tilde v^{\frac{2+\ap}{\ap}}=0$, we deduce from Lemma \ref{lem:KO-1} that $\tilde v(x) \leq Cd(x)d_\Sigma(x)^{-\ap}$ for all $x\in\xO\setminus \Sigma$.
This and \eqref{V>C} implies that $\tilde v(x)\leq V(x)$ for all $x\in \partial\xS_{\frac{1}{n}}$. By the maximum principle, $\tilde v\leq v_n$ in $O_n$. Since $v_n\to v$ locally uniformly in $\xO\setminus \Sigma$, we derive that $\tilde v\leq v$ in $\xO\setminus \Sigma$. Consequently, $2\xn=\tr(\tilde v)\leq \tr(v)=\xn$, thus $\xn\equiv0$ and hence, by \eqref{vGK}, $v \equiv 0$. Thus $u \equiv 0$.
	
	When $p>\frac{2+\ap}{\ap}$ or $p=\frac{2+\ap}{\ap}$ if $\xm=H^2,$ the proof is similar  to the above case, hence we omit it.
\end{proof}

\begin{proof}[\textbf{Proof of Theorem \ref{remove-2}}.]
	Without loss of generality, we may assume that $z=0.$ Let $\xz:\mathbb{R}\to[0,\infty)$ be a smooth function such that $0\leq \xz\leq 1,$ $\xz(t)=0$ for $|t|\leq 1$ and $\xz(t)=1$ for $|t|>2$. For $\varepsilon>0$, we set $\xz_\xe(x)=\xz(\frac{|x|}{\xe}).$
	
	Since $u\in C(\xO \setminus \Sigma)$ by standard elliptic theory we have that $u\in C^2(\xO \setminus \Sigma)$ and hence
	\bal L_\mu (\xz_\xe u)=u\xD\xz_{\xe}+\xz_\xe u^p + 2\nabla\xz_\xe\nabla u \quad\text{in }\;\xO \setminus \Sigma.
	\eal

\noindent \textbf{Step 1:} We show that $L_\mu(\zeta_{\varepsilon}u) \in L^1(\Omega;\phi_\mu)$.

We first see that
\be \label{Lzetau}
\int_{\Omega}|L_\mu(\zeta_\varepsilon u)|\phi_\mu \, \dd x \leq \int_\Omega \xz_\xe u^p \ei \, \dd x + \int_\Omega u|\Delta\xz_{\xe}| \ei \, \dd x + 2\int_\Omega |\nabla\xz_\xe||\nabla u|\ei \, \dd x.
\ee
We note that there exists a constant $C>0$ that does not depend on $\xe$ such that
\bal |\nabla \xz_\xe|^2+ |\xD\xz_{\xe}|\leq C\xe^{-2}\1_{\{\xe\leq |x|\leq 2\xe\}}.
\eal
This, together with \eqref{3.4.24}, \eqref{3.4.24*}, \eqref{eigenfunctionestimates}, the estimate $\int_{\Sigma_{\beta}}d_{\Sigma}(x)^{-\alpha}\dd x \lesssim \beta^{N-\alpha}$ for $\alpha<N-k$, and the assumption $p\geq\frac{N-\am}{N-\am-2}$, yields
\be \label{Lzetau1} \begin{aligned}
&\int_\Omega \xz_\xe u^p \ei \, \dd x \lesssim \xe^{-\frac{2p}{p-1}+\am p}\int_{\xO\cap\{|x|>\xe\}}d_\xS(x)^{-(p+1)\am} \, \dd x \lesssim \varepsilon^{-\frac{2p}{p-1}-\am p}, \\
&\int_\Omega u|\Delta\xz_{\xe}| \ei \, \dd x \leq \xe^{-\frac{2}{p-1}+\am-2}\int_{\xO\cap\{|x|<\xe<2|x|\}} d_\xS(x)^{-2\am} \, \dd x \lesssim \varepsilon^{N-\frac{2}{p-1}-\am-2} \lesssim 1, \\
&\int_\Omega |\nabla\xz_\xe||\nabla u|\ei \dd x \lesssim \xe^{-\frac{2}{p-1}+\am-1}\int_{\xO\cap\{|x|<\xe<2|x|\}} d_\xS(x)^{-2\am-1} \, \dd x \lesssim \varepsilon^{N-\frac{2}{p-1}-\am-2} \lesssim 1.
\end{aligned} \ee
Estimates \eqref{Lzetau} and \eqref{Lzetau1} yield $L_\mu(\zeta_{\varepsilon}u) \in L^1(\Omega;\phi_\mu)$. \medskip

\noindent \textbf{Step 2:} We will show that $u \in L^p(\Omega;\ei)$.

By \cite[Lemma 7.4]{GkiNg_linear}, we have
	\bal
-\int_{\xO} \xz_\xe u L_\xm\eta \,\dd x=-\int_{\xO} \left(u\xD\xz_{\xe}+\xz_\xe u^p + 2\nabla\xz_\xe\nabla u\right)\eta \,\dd x,\quad\forall \eta \in \mathbf{X}_\xm(\xO\setminus \Sigma).
\eal
	Taking $\eta=\ei,$ we obtain
\bal
\xl_\xm\int_{\xO} \xz_\xe u \ei \, \dd x+\int_{\xO} \xz_\xe u^p\ei \, \dd x=-\int_{\xO} \left(u\xD\xz_\xe+2\nabla\xz_\xe\nabla u\right)\ei \, \dd x.
\eal
By the last two lines in \eqref{Lzetau1}, we have
	\bal
	\xl_\xm\int_{\xO} \xz_\xe u\ei \,\dd x+\int_{\xO} \xz_\xe u^p\ei \,\dd x\leq  C.\label{lpfragma}
	\eal
	By Fatou's lemma, letting $\xe\to 0,$ we deduce that
	\be \label{uup}
	\xl_\xm\int_{\xO}  u\ei \,\dd x+\int_{\xO}  u^p\ei \,\dd x\leq  C.
	\ee
	This implies that $u\in L^p(\xO;\ei)$. \medskip
	
\noindent	\textbf{Step 3: End of proof.} Let $\{O_n\}$ be a smooth exhaustion of $\Gw\setminus \Sigma.$ From Step 2, we see that $u+ \BBG_\mu[u^p]$ is a nonnegative $L_\mu$ harmonic function and by the Representation theorem, there exists $\rho \geq 0$ such that
	\be \label{btr}
	u+\BBG_\xm[u^p]=\xr K_{\mu}(\cdot,0)\quad \text{in } \xO\setminus \Sigma.
	\ee	
	
	We will show that $\rho=0$. Suppose by contradiction that $\rho>0$. Let $ n_0\in\mathbb{N}$ large enough such that $\frac{1}{n}\leq \frac{\xb_0}{16}$ for any $n\geq n_0$. For $1<M\in\mathbb{N}$, let $v_{M,n}$ be the positive solution of
	\be\label{sub} \left\{ \BAL
	-L_{\xm }^{O_n}v_{M,n} + v_{M,n}^p&=0\qquad&&\text{in } O_n\\
	v_{M,n}&=M u\qquad&&\text{on } \prt O_n.
	\EAL \right.
	\ee
	 Then $u\leq v_{M,n} \leq Mu$ in $O_n$, since $Mu$ is a supersolution of \eqref{sub}. Furthermore, by \eqref{ko}, there exist a function $v_M$ and a subsequence, still denoted by the same notation, such that $v_{M,n}\to v_M$ locally uniformly in $\xO\setminus \Sigma$. Moreover, from \eqref{sub}, we have
	\be \label{vnM} \BAL
v_{M,n}(x)+\mathbb{G}_\xm^{O_n}[v_{M,n}^p](x) = \BBP_\mu^{O_n}[Mu](x)  = \int_{\partial O_n}Mu\, \dd\xo^x_{ O_n} =: h_n(x),
\EAL
\quad\forall x\in O_n.
	\ee
Now, by \eqref{btr},
\bal
h_n(x)=\int_{\partial O_n}Mu\,\dd\xo^x_{ O_n}=-M\int_{\partial O_n}\BBG_\xm[u^p]\,\dd\xo^x_{ O_n}+M\xr K_{\mu}(x,0).
\eal
	Since $\tr(\BBG_\xm[u^p])=0$, by Definition \ref{nomtrace} (with $\phi=1$), it follows that $h_n(x)\to M\xr K_{\mu}(x,0)$ as $n \to \infty$. By dominated convergence theorem, letting $n \to \infty$ in \eqref{vnM}, we obtain
	\be \label{v_M} v_M(x)+\mathbb{G}_\xm[v^p_M](x)=M\rho K_{\mu}(x,0).
	\ee
	We observe that $\{v_M\}_{M=1}^\infty$ is nondecreasing and by \eqref{3.4.24}, it is locally uniformly bounded from above. Therefore, $v_M\to v$ locally uniformly in $\xO\setminus \Sigma$ as $M\to\infty$. For each $M>1$, we have $v_M \leq Mu$ in $\Omega \setminus \Sigma$, which implies that $v_M$ satisfies \eqref{as2intro}. Therefore, by using an argument similar to the one leading to \eqref{uup}, we deduce that $\{ v_M\}$ is uniformly bounded in $L^p(\Omega \setminus \Sigma; \ei)$. By the monotonicity convergence theorem, we deduce that $v_M \to v$ in $L^p(\Omega \setminus \Sigma;\ei)$, whence $\BBG_\mu[v_M^p] \to \BBG_\mu[v^p]$ a.e.  in $\Omega \setminus \Sigma$. Therefore, by letting $M \to \infty$ in \eqref{v_M}, we derive
	$\lim_{M\to\infty}(v_M(x)+\mathbb{G}_\xm[v^p_M](x))=\infty$,
	which is  a contradiction. Thus $\rho=0$ and hence by \eqref{btr}, $u \equiv 0$ in $\Omega \setminus \Sigma$. The proof is complete.
\end{proof}

\section{Good measures} \label{sec:goodmeasure}

In this section we investigate the problem
\be\label{mainproblempower} \left\{ \BAL -L_\mu u + \abs{ u}^{p-1}u  &= 0 \quad \text{in } \Gw\setminus \Sigma ,\\
\tr(u) &= \xn,
\EAL \right. \ee
where $p>1$ and $\nu \in \GTM(\partial \Omega \cup \Sigma)$.
Recall that a measure is called a $p$-good measure if problem \eqref{mainproblempower} admits a (unique) solution.

Let us first remark that if $1<p<\min\left\{\frac{N+1}{N-1},\frac{N-\am}{N-\am-2}\right\}$ then by Theorem \ref{exist-subGK}, problem \eqref{mainproblempower} admits a unique solution for any $\xn\in \mathfrak{M}(\partial\xO\cup \Sigma)$. Furthermore, if $\xn$  has compact support in $\partial\xO$ and $1<p<\frac{N+1}{N-1}$ (resp. $\xn$ has compact support in $\Sigma$ and $1<p<\frac{N-\am}{N-\am-2}$), then \eqref{mainproblempower} admits a unique weak solution by Theorem \ref{measureO} (resp. by Theorem \ref{exist-measureK}).

In order to characterize $p$-good measures, we make use of appropriate capacities. We recall below some notations concerning Besov space (see, e.g., \cite{Ad, Stein}). For $\gs>0$, $1\leq \kappa<\infty$, we denote by $W^{\gs,\kappa}(\BBR^d)$ the Sobolev space over $\BBR^d$. If $\gs$ is not an integer the Besov space $B^{\gs,\kappa}(\BBR^d)$ coincides with $W^{\gs,\kappa}(\BBR^d)$. When $\gs$ is an integer we denote $\Gd_{x,y}f:=f(x+y)+f(x-y)-2f(x)$ and
\bal B^{1,\kappa}(\BBR^d):=\left\{f\in L^\kappa(\BBR^d): \myfrac{\Gd_{x,y}f}{|y|^{1+\frac{d}{\kappa}}}\in L^\kappa(\BBR^d\times \BBR^d)\right\},
\eal
with norm
\bal \|f\|_{B^{1,\kappa}}:=\left(\|f\|^\kappa_{L^\kappa}+\int_{\BBR^d} \int_{\BBR^d}\frac{|\Gd_{x,y}f|^\kappa}{|y|^{\kappa+d}}\,\dd x \, \dd y\right)^{\frac{1}{\kappa}}.
\eal
Then
\bal B^{m,\kappa}(\BBR^d):=\left\{f\in W^{m-1,\kappa}(\BBR^d): D_x^\ga f\in B^{1,\kappa}(\BBR^d)\;\forall\ga\in \BBN^d \text{ such that } |\ga|=m-1\right\},
\eal
with norm
\bal \|f\|_{B^{m,\kappa}}:=\left(\|f\|^\kappa_{W^{m-1,\kappa}}+\sum_{|\ga|=m-1}\int_{\BBR^d} \int_{\BBR^d}\frac{|D_x^\ga\Gd_{x,y}f|^\kappa}{|y|^{\kappa+d}}\,\dd x \, \dd y\right)^{\frac{1}{\kappa}}.
\eal
These spaces are fundamental because they are stable under the real interpolation method  developed by Lions and Petree.
For $\ga\in\BBR$ we defined the Bessel kernel of order $\ga$ in $\R^d$ by $\CB_{d,\ga}(\xi):=\CF^{-1}\left((1+|.|^2)^{-\frac{\ga}{2}}\right)(\xi)$, where $\CF$ is the Fourier transform in the space $\CS'(\R^d)$ of moderate distributions in $\BBR^d$.
For $\kappa>1$, the Bessel space $L_{\ga,\kappa}(\BBR^d)$ is defined by
\bal L_{\ga,\kappa}(\BBR^d):=\{f=\CB_{d,\alpha} \ast g:g\in L^{\kappa}(\BBR^d)\},
\eal
with norm
\bal \|f\|_{L_{\ga,\kappa}}:=\|g\|_{L^\kappa}=\|\CB_{d,-\ga}\ast f\|_{L^\kappa}.
\eal
It is known that if $1<\kappa<\infty$ and $\ga>0$, $L_{\ga,\kappa}(\BBR^d)=W^{\ga,\kappa}(\BBR^d)$ if $\ga\in\BBN.$ If $\ga\notin\BBN$ then the positive cone of their dual coincide, i.e. $(L_{-\ga,\kappa'}(\BBR^d))^+=(B^{-\ga,\kappa'}(\BBR^d))$, always with equivalent norms. The Bessel capacity is defined for compact subsets
$K \subset\BBR^d$ by
\bal \mathrm{Cap}^{\BBR^d}_{\ga,\kappa}(K):=\inf\{\|f\|^\kappa_{L_{\ga,\kappa}}, f\in\CS'(\BBR^d),\,f\geq \1_K \}.
\eal

\begin{lemma}\label{besov}
	Let $k\geq1$, $ \max\left\{1,\frac{N-k-\am}{N-2-\am}\right\}< p<\frac{2+\ap}{\ap}$ and $\xn\in \mathfrak{M}^+(\mathbb{R}^k)$ with compact support in $B^k(0,\frac{R}{2})$ for some $R>0$. Let $\vartheta$ be as in \eqref{gamma}.
	For $x \in \mathbb{R}^{k+1}$, we write $x=(x_1,x') \in \mathbb{R} \times \mathbb{R}^{k}$.
	Then there exists a constant $C=C(R,N,k,\mu,p)>1$ such that
	\be\label{est-Bnu0} \BAL
	&C^{-1}\norm{\xn}^p_{B^{-\vartheta,p}(\mathbb{R}^k)}\\
	&\leq \int_{B^k(0,R)}\int_{0}^R x_{1}^{N-k-1-(p+1)\am}\left(\int_{B^k(0,R)}\left(|x_1|+|x'-y'|\right)^{-(N-2\am-2)}\dd \nu(y')\right)^p\,\dd x_1\,\dd x'\\
	&\leq C \norm{\xn}^p_{B^{-\vartheta,p}(\mathbb{R}^k)}. \EAL
	\ee
\end{lemma}
\begin{proof} The proof is inspired by the idea in \cite[Proposition 2.8]{BHV}. 

\noindent \textbf{Step 1:} We will prove the upper bound in \eqref{est-Bnu0}.

	Let $0<x_1<R$ and $|x'|<R$. In view of the proof of \cite[Lemma 3.1.1]{Ad}, we obtain
	\bal \BAL
	&\int_{B^k(0,R)}\left(x_1+|x'-y'|\right)^{-(N-2\am-2)}\,\dd \nu(y')\leq \int_{B^k(x',2R)}\left(x_1+|x'-y'|\right)^{-(N-2\am-2)}\,\dd \nu(y')\\
	&=(N-2\am-2)\left(\int_0^{2R}\frac{\xn(B^k(x',r))}{(x_1+r)^{N-2\am-2}}\frac{\dd r}{x_1+r}+\frac{\xn(B^k(x',2R))}{(x_1+2R)^{N-2\am-2}}\right)\\
	&\lesssim \int_0^{3R}\frac{\xn(B^k(x',r))}{(x_1+r)^{N-2\am-2}}\frac{\dd r}{x_1+r}
%	%&=C(N,\am)\int_{x_1}^{3R+x_1}\frac{\xn(B^k(x',r-x_1))}{r^{N-2\a%m-2}}\frac{\dd r}{r}\\
	\leq \int_{x_1}^{4R}\frac{\xn(B^k(x',r))}{r^{N-2\am-2}}\frac{\dd r}{r}.
	\EAL
	\eal
	It follows that
	\ba\BAL
	&\int_{0}^R x_{1}^{N-k-1-(p+1)\am}\left(\int_{B^k(0,R)}\left(|x_1|+|x'-y'|\right)^{-(N-2\am-2)}\dd \nu(y')\right)^p\, \dd x_1\\
	&\lesssim \int_{0}^R x_{1}^{N-k-1-(p+1)\am}\left(\int_{x_1}^{4R}\frac{\xn(B^k(x',r))}{r^{N-2\am-2}}\frac{\dd r}{r}\right)^p \, \dd x_1.
	\EAL
	\ea
	Since $p<\frac{2+\ap}{\ap}<\frac{N-k-\am}{\am}$, it follows that $N-k-(p+1)\am>0$. Let $\xe$ be such that $0<\xe<N-k-(p+1)\am$. By H\"older inequality and Fubini's theorem, we have
	\ba\BAL
	\int_{0}^R &x_{1}^{N-k-1-(p+1)\am}\left(\int_{x_1}^{4R}\frac{\xn(B^k(x',r))}{r^{N-2\am-2}}\frac{\dd r}{r}\right)^p \dd x_1\\
	&\leq \int_{0}^R x_{1}^{N-k-1-(p+1)\am}\left(\int_{x_1}^\infty r^{-\frac{\xe p'}{p}}\frac{\dd r}{r}\right)^{\frac{p}{p'}}\int_{x_1}^{4R}\left(\frac{\xn(B^k(x',r))}{r^{N-2\am-2-\frac{\xe}{p}}}\right)^p\frac{\dd r}{r}\, \dd x_1\\
	&=C(p,\xe)\int_{0}^R x_{1}^{N-k-1-(p+1)\am-\xe}\int_{x_1}^{4R}\left(\frac{\xn(B^k(x',r))}{r^{N-2\am-2-\frac{\xe}{p}}}\right)^p\frac{\dd r}{r}\, \dd x_1\\
	&\leq C(p,\xe,N,k,\am,R)\int_0^{4 R} \left(\frac{\xn(B^k(x',r))}{r^{N-2\am-2-\frac{N-k-(p+1)\am}{p}}}\right)^p\frac{\dd r}{r}.
	\EAL
	\ea
	From the assumption on $p$ and the definition of $\vartheta$ in \eqref{gamma}, we see that $0 < \vartheta <k$. Moreover,
\ba \label{tildegamma}
	N-2\am-2-\frac{N-k-(p+1)\am}{p} = k - \vartheta.
	\ea

	We have
	\ba\BAL
	\int_0^{4 R} \left(\frac{\xn(B^k(x',r))}{r^{k-\vartheta}}\right)^p\frac{\dd r}{r}
	&=\sum_{n=0}^\infty\int_{2^{-n+1}R}^{2^{-n+2}R}\left(\frac{\xn(B^k(x',r))}{r^{k-\vartheta}}\right)^p\frac{\dd r}{r}\\
	&\leq \ln2\sum_{n=0}^\infty 2^{p(n-1)(k-\vartheta)}\left(\frac{\xn(B^k(x',2^{-n+2}R))}{R^{k- \vartheta}}\right)^p\\
	&\leq \ln2\left( \sum_{n=0}^\infty 2^{(n-1)(k-\vartheta)}\frac{\xn(B^k(x',2^{-n+2}R))}{R^{k- \vartheta}}\right)^p\\
	&\leq 2^{p(k-\vartheta)}(\ln2)^{-(p-1)} \left( \sum_{n=0}^\infty\int_{2^{-n+2}R}^{2^{-n+3}R}\frac{\xn(B^k(x',r))}{r^{k- \vartheta}}\frac{\dd r}{r}\right)^p\\
	&= 2^{p(k-\vartheta)}(\ln2)^{-(p-1)}\left(\int_0^{8 R}\frac{\xn(B^k(x',r))}{r^{k-\vartheta}}\frac{\dd r}{r}\right)^p.
	\EAL
	\ea
	Set
	\ba \label{Wgamma}
	\BBW_{\vartheta,{8R}}[\xn](x') := \int_0^{8 R}\frac{\xn(B^k(x',r))}{r^{k-\vartheta}}\frac{\dd r}{r} \quad \text{and} \quad
	\BBB_{k,\vartheta}[\xn](x') :=\int_{\mathbb{R}^k}\CB_{k,\vartheta}(x'-y')\, \dd \nu(y').
	\ea
	Then
	\be \label{up-intB-1} \BAL
	&\int_{B^k(0,R)}\int_{0}^R x_{1}^{N-k-1-(p+1)\am}\left(\int_{B^k(0,R)}\left(x_1+|x'-y'|\right)^{-(N-2\am-2)}\, \dd \nu(y')\right)^p\, \dd x_1\, \dd x'\\
	&\lesssim \int_{\mathbb{R}^k} \BBW_{\vartheta,{8R}}[\nu](x')^p\, \dd x' \lesssim \int_{\mathbb{R}^k}\BB_{k,\vartheta}[\xn](x')^p\, \dd x',
	\EAL
	\ee
	where in the last inequality we have used \cite[Theorem 2.3]{BNV}. Note that the assumption on $p$ ensures that \cite[Theorem 2.3]{BNV} can be applied.
	
	By \cite[Corollaries 3.6.3 and 4.1.6]{Ad}, we obtain
	\ba \label{up-Bnu-1}
	\int_{\mathbb{R}^k}\BB_{k,\vartheta}[\xn](x')^p\, \dd x'\leq C(\vartheta,k,p) \norm{\xn}_{B^{-\vartheta,p}(\mathbb{R}^k)}^p.
	\ea
	Combining \eqref{up-intB-1} and \eqref{up-Bnu-1}, we obtain the upper bound in \eqref{est-Bnu0}. \medskip

\noindent \textbf{Step 2:} We will prove the lower bound in \eqref{est-Bnu0}.
	
	Let $0<x_1<R$ and $|x'|<R$. Then by \cite[Lemma 3.1.1]{Ad}, we have
	\ba\BAL
	\int_{B^k(0,R)}\left(x_1+|x'-y'|\right)^{-(N-2\am-2)}\,\dd \nu(y')
%	&= %\int_{\mathbb{R}^k}\left(x_1+|x'-y'|\right)^{-(N-2\am-2)}\dd \nu(y')\\
%	%&=(N-2\am-2)\int_0^{\infty}\frac{\xn(B^k(x',r))}{(x_1+r)^{N-2\am%-2}}\frac{\dd r}{x_1+r} \\
	&=
	(N-2\am-2)\int_{x_1}^{\infty}\frac{\xn(B^k(x',r-x_1))}{r^{N-2\am-2}}\frac{\dd r}{r}\\
	&\geq (N-2\am-2)\int_{2x_1}^{\infty}\frac{\xn(B^k(x',\frac{r}{2}))}{r^{N-2\am-2}}\frac{\dd r}{r}\\
	&\geq C(N,\am)\int_{x_1}^{\infty}\frac{\xn(B^k(x',r))}{r^{N-2\am-2}}\frac{\dd r}{r}.
	\EAL
	\ea
	It follows that
	\ba\BAL
	&\int_{0}^R x_{1}^{N-k-1-(p+1)\am}\left(\int_{B^k(0,R)}\left(x_1+|x'-y'|\right)^{-(N-2\am-2)}\dd \nu(y')\right)^p \, \dd x_1\\
	&\gtrsim \int_{0}^R x_{1}^{N-k-1-(p+1)\am}\left(\int_{x_1}^\infty\frac{\xn(B^k(x',r))}{r^{N-2\am-2}}\frac{\dd r}{r}\right)^p \, \dd x_1\\
	&\gtrsim \int_{0}^R x_{1}^{N-k-1-(p+1)\am}\left(\int_{x_1}^{2x_1}\frac{\xn(B^k(x',r))}{r^{N-2\am-2}}\frac{\dd r}{r}\right)^p \, \dd x_1\\
	&\gtrsim \int_{0}^R\left(\frac{\xn(B^k(x',x_1))}{x_1^{k-\vartheta}}\right)^p\frac{\dd x_1}{x_1}.
	\EAL
	\ea
	For $0<r<\frac{R}{2}$, we obtain
	\bal
	\int_{0}^R\left(\frac{\xn(B^k(x',x_1))}{x_1^{k-\vartheta}}\right)^p\frac{\dd x_1}{x_1} \geq \int_{r}^{2r}\left(\frac{\xn(B^k(x',x_1))}{x_1^{k-\vartheta}}\right)^p\frac{\dd x_1}{x_1} \gtrsim \left(\frac{\xn(B^k(x',r))}{r^{k-\vartheta}}\right)^p,
	\eal
	which implies
	\bal
	\int_{0}^R\left(\frac{\xn(B^k(x',x_1))}{x_1^{k-\vartheta}}\right)^p\frac{\dd x_1}{x_1}
	\gtrsim \left(\sup_{0<r<\frac{R}{2}}\frac{\xn(B^k(x',r))}{r^{k-\vartheta}}\right)^p.
	\eal
	Set
	\bal M_{\vartheta,\frac{R}{2}}(x'):=\sup_{0<r<\frac{R}{2}}\frac{\xn(B^k(x',r))}{r^{k-\vartheta}}.
	\eal
	Then, since $\xn$ has compact support in $B(0,\frac{R}{2})$,
	\be \label{lower1} \BAL
	&\int_{B^k(0,R)}\int_{0}^R x_{1}^{N-k-1-(p+1)\am}\left(\int_{B^k(0,R)}\left(x_1+|x'-y'|\right)^{-(N-2\am-2)}\,\dd \nu(y')\right)^p \, \dd x_1\, \dd x'\\
	&\gtrsim \int_{B^k(0,R)}M_{\vartheta,\frac{R}{2}}(x')^p \, \dd x' =
	\int_{\mathbb{R}^k}M_{\vartheta,\frac{R}{2}}(x')^p\, \dd x'.
	\EAL
	\ee
By \cite[Theorem 2.3]{BNV} and \cite[Corollaries 3.6.3 and 4.1.6]{Ad},
	\be \label{lower2}
	\int_{\mathbb{R}^k}M_{\vartheta,\frac{R}{2}}(x')^p \, \dd x'\gtrsim \int_{\mathbb{R}^k}\BB_{k,\vartheta}[\xn](x')^p \, \dd x' \gtrsim \norm{\xn}_{B^{-\vartheta,p}(\mathbb{R}^k)}^p.
	\ee
	Combining \eqref{lower1}--\eqref{lower2}, we obtain the lower bound in \eqref{est-Bnu0}.
\end{proof}

\begin{theorem}\label{potest}
	Let $k\geq1,$ $ \max\left\{1,\frac{N-k-\am}{N-\am-2}\right\}< p<\frac{2+\ap}{\ap}$ and $\xn\in \mathfrak{M}^+(\partial\xO\cup \Sigma)$ with compact support in $\Sigma$. Then there exists a constant $C=C(\xO,\Sigma,\mu)>1$ such that
	\be \label{2sideKnu} \BAL
	&C^{-1}\norm{\xn}_{B^{-\vartheta,p}(\Sigma)}\leq \norm{\BBK_{\mu}[\xn]}_{L^p(\xO;\ei)}\leq C \norm{\xn}_{B^{-\vartheta,p}(\Sigma)},\EAL
	\ee
	where $\vartheta$ is given in \eqref{gamma}.
\end{theorem}
\begin{proof}
	By \eqref{cover}, there exists $\xi^j \in \Sigma$, $j=1,2,...,m_0$ (where $m_0 \in\N$ depends on $N,\Sigma$), and $\xb_1 \in (0,\frac{\xb_0}{4})$ such that
	$
	\Sigma_{\beta_1} \subset \cup_{j=1}^{m_0} V(\xi^j,\frac{\beta_0}{4})\Subset \xO.
	$	
	
	\noindent \textbf{Step 1:} We establish local 2-sided estimates.
	
	Assume $\xn\in \mathfrak{M}^+(\partial\xO\cup \Sigma)$ with compact support in $\Sigma\cap  V(\xi^j,\frac{\beta_0}{2})$ for some $j\in\{1,...,m_0\}$. We write
	\be \label{K-split1}
	\int_{ \Omega}\ei \BBK_{\mu}[\xn]^p \,\dd x
	= \int_{\Omega \setminus  V(\xi^j,\beta_0)}\ei \BBK_{\mu}[\xn]^p \,\dd x+\int_{ V(\xi^j,\beta_0) }\ei \BBK_{\mu}[\xn]^p \,\dd x.
	\ee
	On one hand, by \eqref{eigenfunctionestimates} and Proposition \ref{Martin}, we have
	\be\label{K-split12} \BAL
	&\int_{ \xO\setminus V(\xi^j,\beta_0)}\ei \BBK_{\mu}^p[\xn] \, \dd x  \\
	&\approx \int_{\Omega \setminus  V(\xi^j,\beta_0)} d(x)d_{\Sigma}(x)^{-\am} \left( \int_{\Sigma \cap V(\xi^j,\beta_0/2)}\frac{d(x)d_\Sigma(x)^{-\am}} {|x-y|^{N-2-2\am}}\dd \nu(y) \right)^p\dd x \\
	&\lesssim \nu(\Sigma \cap V(\xi^j,\beta_0/2))^p\int_{\Omega \setminus \Sigma}d_\Sigma(x)^{-(p+1)\am} \, \dd x \lesssim \nu(\Sigma \cap V(\xi^j,\beta_0/2))^p.
	\EAL \ee
	In the last estimate we have used estimate $\int_{\Omega}d_\Sigma(x)^{-(p+1)\am} \dd x \lesssim 1$ since $(1+p)\am < N-k$.

	On the other hand, again by \eqref{eigenfunctionestimates} and Proposition \ref{Martin}, we have
	\be\label{nup-0} \BAL
	&\int_{ V(\xi^j,\beta_0)}\ei \BBK_{\mu}^p[\xn]\, \dd x  \\
	&\approx \int_{V(\xi^j,\beta_0)} d(x)d_{\Sigma}(x)^{-\am} \left( \int_{\Sigma \cap V(\xi^j,\beta_0/2)}\frac{d(x)d_\Sigma(x)^{-\am}} {|x-y|^{N-2-2\am}} \dd \nu(y) \right)^p\, \dd x \\
	&\gtrsim \nu(\Sigma \cap V(\xi^j,\beta_0/2))^p\int_{V(\xi^j,\beta_0)}d_\Sigma(x)^{-(p+1)\am} \dd x  \gtrsim \nu(\Sigma \cap V(\xi^j,\beta_0/2))^p.
	\EAL \ee
	
	%{\color{blue} \begin{align}\label{20}
	%\int_{ \xO\setminus V(\xi^j,\beta_0)}\ei %\BBK_{\mu}^p[\xn]\dd x\leq %C(\xO,\xm,\Sigma,p)\xn^p(V(\xi^j,\frac{\beta_0}{2}))\leq %C(\xO,\xm,\Sigma)\int_{ V(\xi^j,\beta_0)}\ei %\BBK_{\mu}^p[\xn]\dd x.
	%\end{align}}
	Combining \eqref{K-split1}--\eqref{nup-0} yields

	\be \label{K-split2}
	\int_{ \Omega }\ei \BBK_{\mu}^p[\xn]\,\dd x
	\approx \int_{ V(\xi^j,\beta_0) }\ei \BBK_{\mu}^p[\xn]\,\dd x.
	\ee
	For any $x \in \R^N$, we write $x=(x',x'')$ where $x'=(x_1,\ldots,x_k)$ and $x''=(x_{k+1},\ldots,x_N)$, and define the $C^2$ function \bal \xF(x):=(x',x_{k+1}-\xG_{k+1}^{\xi^j}(x'),...,x_N-\xG_{N}^{\xi^j}(x')).
	\eal
	By \eqref{straigh}, $\xF:V(\xi^j,\beta_0)\to B^k(0,\xb_0)\times B^{N-k}(0,\xb_0) $ is $C^2$ diffeomorphism and
	$\xF(x)=(x',0_{\R^{N-k}})$ for $x=(x',x'') \in \Sigma$.
	In view of the proof of \cite[Lemma 5.2.2]{Ad}, there exists a measure $\overline{\xn}\in \GTM^+(\mathbb{R}^k)$ with compact support in $ B^k(0,\frac{\xb_0}{2})$ such that for any Borel $E\subset B^k(0,\frac{\xb_0}{2}) ,$ there holds
	$\overline{\xn}(E)=\xn(\xF^{-1}(E\times \{0_{\R^{N-k}}\}))$.
	
	Set $\psi=(\psi',\psi'')=\xF(x)$ then $\psi'=x'$ and $\psi''=(x_{k+1}-\xG_{k+1}^{\xi^j}(x'),...,x_N-\xG_{N}^{\xi^j}(x'))$. By \eqref{propdist}, \eqref{eigenfunctionestimates} and \eqref{Martinest1}, we have
	\bal
	&\ei(x)\approx |\psi''|^{-\am},\\
	 &K_{\mu}(x,y)\approx |\psi''|^{-\am}(|\psi''|+|\psi'-y'|)^{-(N-2\am-2)}, \;
	 \forall x\in  V(\xi^j,\beta_0)\setminus \Sigma,\;\forall y=(y',y'') \in  V(\xi^j,\beta_0)\cap \Sigma.
	\eal
	Therefore
	\ba \label{21} \BAL
	&\int_{ V(\xi^j,\beta_0) }\ei \BBK_{\mu}^p[\xn]\,\dd x\\
	&\approx \int_{B^k(0,\xb_0)}\int_{B^{N-k}(0,\xb_0)}|\psi''|^{-(p+1)\am}
	\left(\int_{B^k(0,\xb_0)}(|\psi''|+|\psi'-y'|)^{-(N-2\am-2)}\dd \overline{\nu}(y')\right)^p \, \dd\psi'' \, \dd\psi'\\
	&=C(N,k)\int_{B^k(0,\xb_0)}\int_{0}^{\xb_0}r^{N-k-1-(p+1)\am}
	\left(\int_{B^k(0,\xb_0)}(r+|\psi'-y'|)^{-(N-2\am-2)}\dd \overline{\nu}(y')\right)^p\dd r \, \dd\psi'.
	\EAL
	\ea
	
	Since $\xn\mapsto \gn\circ\xF^{-1}$ is a $C^2$ diffeomorphism between
	$\mathfrak M^+(\Sigma\cap V(\xi^j,\beta_0))\cap B^{-\vartheta,p}(\Sigma\cap V(\xi^j,\beta_0))$ and
	$\mathfrak M^+(B^k(0,\xb_0))\cap B^{-\vartheta,p}(B^k(0,\xb_0))$, using  \eqref{K-split2},\eqref{21} and
	Lemma \ref{besov}, we derive that
	\ba\BAL
	C^{-1}\norm{\xn}_{B^{-\vartheta,p}(\Sigma)}\leq \norm{\BBK_{\mu}[\xn]}_{L^p(\xO;\ei)}\leq C \norm{\xn}_{B^{-\vartheta,p}(\Sigma)},\EAL\label{22}
	\ea
	\noindent \textbf{Step 2:} We will prove global two-sided estimates.
	
	If $\xn\in \mathfrak{M}^+(\partial\xO\cup \Sigma)$ with compact support in $\Sigma,$ we may write $\xn=\sum_{j=1}^{m_0}\xn_j,$ where $\xn_j\in \mathfrak{M}^+(\partial\xO\cup \Sigma)$ with compact support in $V(\xi^j,\frac{\beta_0}{2}).$
	On one hand, by step 1, we have
	\be \label{glbK-1} \BAL
	\norm{\BBK_{\mu}[\xn]}_{L^p(\xO;\ei)}\leq \sum_{j=1}^{m_0}\norm{\BBK_{\mu}[\xn_j]}_{L^p(\xO;\ei)}\leq  C\sum_{j=1}^{m_0} \norm{\xn_j}_{B^{-\vartheta,p}(\Sigma)}
	\leq C m_0 \norm{\xn}_{B^{-\vartheta,p}(\Sigma)}.
	\EAL \ee
	On the other hand, we deduce from step 1 that
	\bal 
	\norm{\BBK_{\mu}[\xn]}_{L^p(\xO;\ei)} \geq {m_0}^{-1} \sum_{j=1}^{m_0}\norm{\BBK_{\mu}[\xn_j]}_{L^p(\xO;\ei)} 
	\geq  (Cm_0)^{-1}\sum_{j=1}^{m_0} \norm{\xn_j}_{B^{-\vartheta,p}(\Sigma)}
	\geq (Cm_0)^{-1}\norm{\xn}_{B^{-\vartheta,p}(\Sigma)}.
	\eal
	This and \eqref{glbK-1} imply \eqref{2sideKnu}. The proof is complete.
\end{proof}

Using Theorem \ref{potest}, we are ready to prove Theorem \ref{supcrK}.
\begin{proof}[\textbf{Proof of Theorem \ref{supcrK}}.]
	If $\xn$ is a positive measure which vanishes on Borel sets $E\subset \Sigma$ with $\mathrm{Cap}^{\BBR^{k}}_{\vartheta,p'}$-capacity zero, there exists an increasing sequence $\{\xn_n\}$ of positive measures in $B^{-\vartheta,p}(\Sigma)$ which converges weakly to $\xn$ (see \cite{DaM}, \cite {FDeP}). By Theorem \ref{potest}, we have that $\BBK_{\mu}[\xn_n]\in L^p(\xO\setminus \Sigma;\ei)$, hence we may apply Theorem \ref{existencesubcr} with $w=\BBK_{\mu}[\xn_n],$ $v=0$ and $g(t)=|t|^{p-1}t$ to deduce that there exists a unique nonnegative weak solution $u_n$ of \eqref{mainproblempower} with $\tr(u_n)=\xn_n.$
	
	Since $\{ \nu_n \}$ is an increasing sequence of positive measures, by Theorem \ref{linear-problem}, $\{u_n\}$ is increasing and its limit is denoted by $u$. Moreover,
	\be \label{lweakform-un}
	- \int_{\Gw}u_n L_{\xm }\zeta \, \dd x + \int_{\Gw} u_n^p \zeta \, \dd x = - \int_{\Gw} \mathbb{K}_{\xm}[\xn_n]L_{\xm }\zeta \, \dd x
	\qquad\forall \zeta \in\mathbf{X}_\xm(\xO\setminus \Sigma).
	\ee
	By taking $\zeta=\ei$ in \eqref{lweakform-un}, we obtain
	\bal \int_{\Gw} \left(\gl_\xm u_{n}+u_n^p\right)\ei \,\dd x=\gl_\xm\int_{\Gw} \BBK_{\mu}[\xn_n]\ei \,\dd x,
	\eal
	which implies that $\{u_n\}$ and $\{u_n^p\}$ are uniformly bounded in $L^1(\xO\setminus \Sigma;\ei)$. Therefore $u_n \to u$ in $L^1(\Omega;\ei)$ and in $L^p(\Omega;\ei)$. By letting $n \to \infty$ in \eqref{lweakform-un}, we deduce
	\bal \int_{\Gw}-uL_\xm\zeta \, \dd x +\int_{\Gw} u^p\zeta \, \dd x=-\int_{\Gw} \BBK_{\mu}[\xn]L_\xm\zeta \,\dd x\qquad\forall \zeta\in {\bf X}_\xm(\Gw\setminus \Sigma).
	\eal
	This means $u$ is the unique weak solution of \eqref{mainproblempower} with $\tr(u)=\xn$.
\end{proof}

Next we demonstrate Theorem \ref{supcromega}. \smallskip

\begin{proof}[\textbf{Proof of Theorem \ref{supcromega}}.] ~~
	
1. Suppose $u$ is a weak solution of \eqref{mainproblempower} with $\tr(u)=\xn$. Let $\beta>0$. Since
	\ba\label{23}
	\ei(x)\approx d(x)\quad\text{and}\quad K_{\mu}(x,y)\approx d(x)|x-y|^{-N}\quad\forall (x,y)\in (\xO\setminus \Sigma_\xb)\times \partial \xO,
	\ea
	proceeding as in the proof of \cite[Theorem 3.1]{MV-JMPA01}, we may prove that $\xn$ is absolutely continuous with respect to the Bessel capacity $\mathrm{Cap}^{\BBR^{N-1}}_{\frac{2}{p},p'}$. \medskip
	
	2. We assume that $\xn \in \GTM^+(\partial \Omega) \cap B^{-\frac{2}{p},p}(\partial\xO).$ Then by \eqref{23}, we may apply \cite[Theorem A]{MV-JMPA01} to deduce that  $\BBK_{\mu}[\xn]\in L^p(\xO\setminus \Sigma_\xb;\ei)$ for any $\xb>0$. Denote $g_n(t)=\max\{ \min\{|t|^{p-1}t,n\},-n\}$. Then by applying Theorem \ref{existencesubcr} with $w=\BBK_{\mu}[\xn],$ $v=0$ and $g=g_n$, we find that there exists a unique weak solution $v_n \in L^1(\Omega;\phi_\mu)$ of
	\be \left\{ \BAL
	- L_\gm v_n+g_n(v_n)&=0\qquad \text{in }\;\Gw\setminus \Sigma,\\
	\tr(v_n)&=\nu,
	\EAL \right. \ee
	such that $0 \leq v_n \leq \BBK_\mu[\nu]$ in $\Omega \setminus \Sigma$.
	Furthermore, by \eqref{poi5}, $\{v_n\}$ is non-increasing. Denote $v=\lim_{n \to \infty}v_n$ then $0 \leq v \leq \BBK_\mu[\nu]$ in $\Omega \setminus \Sigma$.
	
	We have
	\be \label{lweakform-un-1}
	- \int_{\Gw}v_n L_{\xm }\zeta \, \dd x + \int_{\Gw} g_n(v_n) \zeta \, \dd x = - \int_{\Gw} \mathbb{K}_{\xm}[\xn_n]L_{\xm }\zeta \, \dd x
	\qquad\forall \zeta \in\mathbf{X}_\xm(\xO\setminus \Sigma).
	\ee
	By taking $\ei$ as test function, we obtain
	\be \label{ungn} \int_{\Gw}\left(\gl_\xm v_n + g_n(v_n)\right)\ei \,\dd x=\gl_\xm\int_{\Gw} \BBK_{\mu}[\xn]\ei \,\dd x,
	\ee
	which, together with by Fatou's Lemma, implies that $v, v^p\in L^1(\xO;\ei)$ and
	\bal \int_{\Gw}\left(\gl_\xm v + v^p \right)\ei \,\dd x
	\leq \gl_\xm\int_{\Gw} \BBK_{\mu}[\xn]\ei \,\dd x.
	\eal
	 Hence $v+\BBG_\xm[v^p]$ is a nonnegative $L_\xm$ harmonic. By Representation Theorem \ref{th:Rep}, there exists a unique $\overline{\xn}\in\mathfrak{M}^+(\partial\xO\cup \Sigma)$ such that $v +\BBG_\xm[v^p]=\BBK_\mu[\overline{\xn}]$.
	Since $v \leq \BBK_\mu[\xn]$, by Proposition \ref{traceKG} (i), $\overline \nu = \tr(v) \leq \tr(\BBK_\mu[\nu]) =  \nu$ and hence $\overline{\xn}$ has compact support in $\partial\xO$.

Let	$\xz\in \mathbf{X}_\xm(\xO\setminus \Sigma)$ and $\xb>0$ be small enough such that $\xO_{4\xb}\cap \Sigma=\emptyset$ (recall that $\Omega_{\beta}$ is defined in Notations). We consider a cut-off function $\psi_{\beta} \in C^\infty(\R^N)$ such that $0\leq\psi_{\xb}\leq 1$ in $\R^N$, $\psi_\xb=1$ in $\xO_{\frac{\xb}{2}}$ and $\psi_{\xb}=0$ in $\xO\setminus \xO_\xb$. Then the function $\psi_{\beta,\xz}=\psi_{\beta}\xz \in \mathbf{X}_\xm(\xO\setminus \Sigma)$ has compact support in $\overline{\xO}_{\xb}$. Hence, by \eqref{weakfor2} and the fact that $\frac{\partial \psi_{\beta,\xz}}{\partial {\bf n}}= \frac{\partial \xz}{\partial {\bf n}}$ on $\partial \Omega$, we obtain
	
	\ba\label{24}
	\int_{\Gw}(-v L_\xm\psi_{\xb,\xz} + v^p\psi_{\xb,\xz})\,\dd x=-\int_{\partial\xO}\frac{\partial \xz}{\partial {\bf n}} \frac{1}{P_\mu(x_0,y)}\, \dd \overline{\nu}(y)=-\int_{\Gw}\BBK_{\mu}[\overline{\xn}]L_\xm\xz \,\dd x.
	\ea
	Also,
\ba \label{25}
	\int_{\Gw}(- v_nL_\xm\psi_{\xb,\xz}+g_n(v_n)\psi_{\xb,\xz})\, \dd x=-\int_{\partial\xO}\frac{\partial \xz}{\partial {\bf n}} \frac{1}{P_\mu(x_0,y)}\, \dd \nu(y)=-\int_{\Gw}\BBK_{\mu}[\xn]L_\xm\xz \, \dd x.
\ea
	
	Since $v \leq v_n\leq \BBK_{\mu}[\xn]$ and $\BBK_{\mu}[\xn]\in L^p(\xO_{4\xb};\ei)$, by letting $n\to\infty$ in \eqref{25}, we obtain by the dominated convergence theorem that
	\ba \label{26}
	\int_{\Gw}(-v L_\xm\psi_{\xb,\xz} + v^p\psi_{\xb,\xz})\, \dd x=-\int_{\Gw}\BBK_{\mu}[\xn]L_\xm\xz \, \dd x.
	\ea
	From \eqref{24} and \eqref{26}, we deduce that
	
\bal
\int_{\Gw}\BBK_{\mu}[\overline{\xn}]L_\xm\xz \, \dd x=\int_{\Gw}\BBK_{\mu}[\xn]L_\xm\xz \, \dd x,\quad \forall \xz\in \mathbf{X}_\xm(\xO\setminus \Sigma).
\eal
Since $\BBK_{\mu}[\overline{\xn}],\BBK_{\mu}[\xn]\in C^2(\xO\setminus \xS),$ by the above inequality, we can easily show that $\BBK_{\mu}[\overline{\xn}]=\BBK_{\mu}[\xn],$ which implies $\overline{\xn}=\xn$ by Proposition \ref{traceKG}. \medskip

	3. If $\xn \in \GTM^+(\partial \Omega)$ vanishes on Borel sets $E\subset \partial\xO$ with zero $\mathrm{Cap}^{\BBR^{N-1}}_{\frac{2}{p},p'}$-capacity, there exists an increasing sequence $\{\xn_n\}$  of positive measures in $B^{-\frac{2}{p},p}(\partial\xO)$ which converges to $\xn$ (see \cite{DaM}, \cite {FDeP}). Let $u_n$ be the unique weak solution of \eqref{mainproblempower} with boundary trace $\xn_n$.
	Since $\{ \nu_n \}$ is increasing, by \eqref{poi5}, $\{u_n\}$ is increasing. Moreover, $0 \leq u_n \leq \BBK_\mu[\nu_n] \leq \BBK_\mu[\nu]$.  Denote $u=\lim_{n \to \infty}u_n$. By an argument similar to the one leading to \eqref{ungn}, we obtain
	\bal \int_{\Gw}\left(\gl_\xm u_{n} + u_n^p\right)\ei \, \dd x=\gl_\xm\int_{\Gw} \BBK_{\mu}[\xn_n]\ei \, \dd x,
	\eal
	it follows that $u, u^p\in L^1(\xO;\ei)$. By the dominated convergence theorem, we derive
	\bal \int_{\Gw} \left(-u L_\xm\zeta + u^p\zeta\right)\, \dd x=-\int_{\Gw}\BBK_{\mu}[\xn]L_\xm\zeta \, \dd x\qquad\forall \zeta\in {\bf X}_\mu(\Gw\setminus \Sigma),
	\eal
	and thus $u$ is the unique weak solution of \eqref{mainproblempower}. \medskip
\end{proof}

\appendix
\renewcommand{\thesection}{\Alph{section}}
\setcounter{section}{0}
\section{A priori estimates}
%\setcounter{equation}{0}
%%%%%%%%%%%%%%%%%%%%%%%%%%%%%%%%%%%%%%%%%%%%%%%%%%%%%%%%%%%%%%%%%%%%%%%%%%%%%%%%%%%%%%%%%%%%%%%%%%%%%%%%%%%%%%%%%%%%%%%%%%%%%%%%%%%%%%%%%%%%%%%%%%%%%%%%%%%%%%%%%%%%%%%%%%%%%%
\begin{proposition}\label{barr} There exists $R_0 \in (0,\beta_0)$ such that for any $z\in \Sigma$ and $0<R\leq R_0$, there is a supersolution $w:=w_{R,z}$ of \eqref{eq:power1} in $\Gw\cap B(z,R)$ such that
\bal &w \in C(\overline\Gw\cap B(z,R)), \quad  w = 0 \text{ on } \Sigma\cap B(z,R),\\
&w(x)\to \infty \text{ as } \dist (x,F)\to 0, \text{ for any compact subset } F\subset (\xO\setminus \Sigma)\cap \prt B(z,R).
\eal
More precisely, for $\gamma \in (\am,\ap)$, $w$ can be constructed as
	\be \label{BAR1}
	w(x)= \left\{\BA{lll} \Gl(R^2-|x-z|^2)^{-b}d_\Sigma(x)^{-\gamma} \quad & \text{ if }\;\mu<H^2, \\[2mm]
	\Gl(R^2-|x-z|^2)^{-b}d_\Sigma(x)^{-H}\sqrt{\ln \left(\frac{eR}{d_\Sigma(x)}\right)} \;&\text{ if }\;\xm=H^2,
	\EA\right.
	\ee
	with $b\geq \max\{\frac{2}{p-1},\frac{N-2}{2},1\}$ and $\Gl>0$ large enough depending only on $R_0, \xg,N,b,p$ and the $C^2$ characteristic of $\Sigma.$
\end{proposition} %%%%%%%%%%%%%%%%%%%%%%%%%%%%%%%%%%%%%%%%%%%%%%%%%%%%%%%%%%%%%%PROOF%%%%%%%%%%%%%%%%%%%%%%%%%%%%%%%%%%%%%%%%%%%%%%%%%%%%%%%%%%%%%%%%%%%%%%%%%%%%%%%%%%%%%%%%%%%%%%%%%
\begin{proof} Without loss of generality, we assume $z=0 \in \Sigma$.
	%%%%%%%%%%%%%%%%%%%%%%%%%%%%%%%%%%%%%%%%%%%%%%%%%
	%-\gamma\phi_1^{\gamma-1}\Gd\phi_1-\xk d^{-2}\phi_1^{\gamma}
	%=\gamma(-\phi_1^{\gamma-1}\Gd\phi_1-\xk d^{-2}\phi_1^{\gamma})
	%-(1-\gamma)\xk d^{-2}\phi_1^{\gamma})
	%%%%%%%%%%%%%%%%%%%%%%%%%%%%%%%%%%%%%%%%%%%%%%%%%
	%%%%%%%%%%%%%%%%%%%%%%%%%%%%%%%%%%%%%%%%%%%%%%%%%%%%%%%%%%%%%%%%%%%%%%%%%%%%%%%%%%%%%%%%%%%%%%%%%%
	\medskip
	
\noindent  \textbf{Case  1: $\xm<H^2$.} Set
	\bal w(x):=\Gl(R^2-| x|^2)^{-b} d_\Sigma(x)^{-\gamma} \quad \text{for } x \in B(0,R),
	\eal
	where $\gamma>0, b$ and $\Gl>0$ will be determined later on. Then, by straightforward computation with $r=|x|$ and using \eqref{laplaciand} , we obtain

	\be\label{F3}-L_{\xm }w+w^p=\Gl(R^2-r^2)^{-b-2}d^{-\gamma-2}_\Sigma(I_1+I_2+I_3+I_4),
	\ee
	where
	\bal
	&I_1:=\Gl^{p-1}(R^2-r^2)^{-(p-1)b+2}d^{-(p-1)\gamma+2}_\Sigma,
	\\
	&I_2:=-(R^2-r^2)^{2}\left(-\xg  \eta d_\Sigma-\gamma(N-k-2-\xg)+\xm \right),
	\\
	&I_3:=-2b d^2_\Sigma\left(NR^2+(2b+2-N)r^2\right), \\
	&I_4:=4b\gamma d_\Sigma(R^2-r^2)x\nabla d_\Sigma.
	\eal
	If we choose $b\geq \frac{N-2}{2}$ then
	\be \label{F3-0} \begin{aligned}
	-I_3 \leq 4b(b+1) R^2d^2_\Sigma  \quad \text{and} \quad
	  |I_4| \leq 4b|\xg| R(R^2-r^2)d_\Sigma.
	\end{aligned} \ee
Next we choose $\gamma \in (\am,\ap),$ then $-\ap(N-k-2)+\xm<-\gamma(N-k-2-\xg)+\xm <0$. In addition, there exist $\ge_0>0$ and $\delta_0>0$ such that if $d_\Sigma\leq\gd_0$ then
\bal %\label{F3-00}
-\ap(N-k-2)+\xm<-\xg \eta d_\Sigma-\gamma(N-k-2-\xg)+\xm<-\ge_0.
\eal
It follows that if $d_\Sigma\leq\gd_0$ then
\be \label{F3-000} I_2 \geq \epsilon_0(R^2-r^2)^2.
\ee
We set
	\bal &\CA_1:=\left\{x\in\Gw\cap B_R(0):d_\Sigma(x)\leq c_1\frac{R^2-r^2}{R}\right\} \quad \text{where } c_1=\frac{\epsilon_0}{16b(|\gamma|+1)},\\
	& \CA_2:=\left\{x\in\Gw\cap B_R(0):d_\Sigma(x)\leq \gd_0^{\phantom{^4}}\right\}, \quad
	\CA_3:=\{x\in\Gw:d_\Sigma(x)\geq\gd_0\}.
	\eal
	
	In  $\CA_1 \cap \CA_2$, by \eqref{F3-0} and \eqref{F3-000}, for $b \geq \max\{ \frac{N-2}{2},1 \}$, we have
	\be \label{F3-1}
	I_2 + I_3 + I_4 \geq \frac{\xe_0(R^2-r^2)^2}{2}.
	\ee
	
	In $\CA_1^c\cap \CA_2$, $d_\Sigma \geq c_1\frac{R^2-r^2}{ R}$. If  we choose $b>\frac{2}{p-1}$, then there exists $\Gl$ large enough depending on $p,R_0,\xd_0,N,b,\gamma$ such that the following estimate holds
	\be\label{F4} \begin{aligned}
	&I_1 \geq 2\max\{  4b(b+1)R^2 d_\Sigma^2,4b |\xg| d_\Sigma R(R^2-r^2)\}.
	\end{aligned} \ee
	This, together with \eqref{F4}, yields
	\be \label{F4-0}
	I_1 + I_3 + I_4 \geq 0.
	\ee
	
	In $\CA_3$, $d_\Sigma \geq \delta_0$. Therefore, we can show that there exists $c_2>0$ depending on $N,\gamma,b,\| \eta \|_{L^\infty(\Sigma_{4\beta_0})},\gd_0,p$ such that if
	%\begin{equation}\label{F8}\BA {lll}
	$\Gl\geq c_2$
	%\EA\end{equation}
	then, in $\CA_3$,
	\be \label{F9} \begin{aligned}
	I_1 \geq 3\max\{ |\xg \eta| d_\Sigma(R^2-r^2)^2, 4d_\Sigma^2b(b+1)R^2, 4b d_\Sigma R(R^2-r^2) \}.
	\end{aligned} \ee
	It follows that
	\be \label{F10}
	I_1 + I_2 + I_3 + I_4 \geq 0.
	\ee

	Combining \eqref{F3}, \eqref{F3-000}, \eqref{F3-1}, \eqref{F4-0} and \eqref{F10}, we deduce that for $\gamma \in (0,\ap)$, $b\geq \max\{\frac{2}{p-1},\frac{N-2}{2},1\}$ and $\Gl>0$ large enough, there holds
	\be \label{F11}
	-L_{\xm }w + w^p\geq 0\qquad\text{in }\;\Gw \cap B(0,R).
	\ee
	%%%%%%%%%%%%%%%%%%%%%%%%%%%%%%%%%%%%%%%%%%%%%%%%%%%%%%%%%%%%%%%%%%%%%%%%%%%%%%%%%%%%%%%%%%%%%%%%%%
	
\noindent	\textbf{Case 2: $\xm =H^2$. } Set
\bal
w(x):=\Gl(R^2-r^2)^{-b} d_\Sigma^{-H}\left(\ln \frac{eR_0}{d_\Sigma}\right)^\frac{1}{2}, \quad \text{for } |x|<R,
\eal
	where $b$ and $\Gl$ will be determined later.
	Then, by straightforward calculations we have
%	$$\BA{ll} L_\xm\left(d_\Sigma^{-H}(\ln \frac{eR}{d_\Sigma})^\frac{1}{2}\right)= -\frac{1}{4}d_\Sigma^{-H-2}(\ln \frac{eR}{d_\Sigma})^{-\frac{3}{2}} - \frac{1}{2}gd_\Sigma^{-H-1}(\ln \frac{eR}{d_\Sigma})^\frac{1}{2}\left(2H+(\ln \frac{eR}{d_\Sigma})^{-1}\right).
%	\EA 	$$
%	Furthermore
%	$$\BA{ll}
%	\nabla (R^2-r^2)^{-b} \nabla  \left(d_\Sigma^{-H}(\ln \frac{eR}{d_\Sigma})^\frac{1}{2}\right)=-b (R^2-r^2)^{-b-1}(\ln \frac{eR}{d_\Sigma})^{-\frac{1}{2}} d^{-H-1}_\Sigma\left(2H(\ln \frac{eR}{d})+1\right)x\nabla d_\Sigma, \\
%	\Delta(R^2-r^2)^{-b} d_\Sigma^{-H} (\ln \frac{eR}{d_\Sigma})^{\frac{1}{2}} = 2b (R^2-r^2)^{-b-2} d_\Sigma^{-H} (\ln \frac{eR}{d_\Sigma})^{\frac{1}{2}} [NR^2 + (2b+2 -N)r^2].
%	\EA$$
%	The above equalities imply
%	\begin{align*}
%	&L_{\xm }w = -\Gl(R^2-r^2)^{-b-2}d_\Sigma^{-\frac{3}{2}}(\ln \frac{eR}{d_\Sigma})^{-\frac{3}{2}}\left\{ (R^2-r^2)^2 \Big[ \frac{1}{2} gd_\Sigma\Big(2H(\ln \frac{eR}{d_\Sigma})^2+(\ln \frac{eR}{d_\Sigma})\Big) + \frac{1}{4} \Big] \right. \\
%	&\left. +2b(R^2-r^2)d_\Sigma\Big[2H(\ln \frac{eR}{d_\Sigma})^2+(\ln \frac{eR}{d_\Sigma})\Big]x\nabla d_\Sigma- 2b d^2_\Sigma(\ln\frac{eR}{d_\Sigma})^2\left[NR^2+(2b+2-N)r^2\right]
%	\right\}.
%	\end{align*}
%	Therefore
	\be \label{F12}
	-L_{\xm }w +w^p =\Gl(R^2-r^2)^{-b-2}d^{-H-2}_\Sigma \left(\ln \frac{eR}{d_\Sigma}\right)^{-\frac{3}{2}}(\tilde I_1 + \tilde I_2 + \tilde I_3 + \tilde I_4),
	\ee
	where
	\bal
	&\tilde I_1: = (R^2-r^2)^{2}\left[\frac{1}{2}\eta d_\Sigma\left(2H \left(\ln \frac{eR}{d_\Sigma}\right)^2+ \left(\ln \frac{eR}{d_\Sigma}\right)\right)+\frac{1}{4}\right],\\
	&\tilde I_2: = 2b(R^2-r^2)d_\Sigma\left[2H \left(\ln \frac{eR}{d_\Sigma}\right)^2 + \left(\ln \frac{eR}{d_\Sigma}\right)\right]x\nabla d_\Sigma, \\
	&\tilde I_3: = -2b d^2_\Sigma \left(\ln\frac{eR}{d_\Sigma}\right)^2\left[NR^2+(2b+2-N)r^2\right], \\
	&\tilde I_4: =\Gl^{p-1}(R^2-r^2)^{-b(p-1)+2}d^{-H(p-1)+2}_\Sigma \left(\ln \frac{eR}{d_\Sigma}\right)^{\frac{1}{2}(p-1)+2}.
	\eal
	Notice that $\frac{eR}{d_\Sigma}\geq e$, whence
	\be \label{eRdK} (2H+1)\left(\ln \frac{eR}{d_\Sigma}\right) \leq 2H\left(\ln \frac{eR}{d_\Sigma}\right)^2+ \left(\ln \frac{eR}{d_\Sigma}\right)\leq (2H+1)\left(\ln \frac{eR}{d_\Sigma}\right)^2.
	\ee
	If we choose $b \geq \frac{N-2}{2}$ then
	\be \label{tildeI23} \BA {ll}
	|\tilde I_2| \leq 4b(b+1)(R^2-r^2)(\ln \frac{eR}{d_\Sigma})^2d_\Sigma R, \\
	|\tilde I_3| \leq 4b(b+1)(\ln \frac{eR}{d_\Sigma})^2d_\Sigma^2R^2.
	\EA \ee
	From \eqref{eRdK}, we deduce that there exist $\epsilon_0>0$ and $\gd_0>0$ such that if $d_\Sigma\leq \gd_0$ then
	\bal \frac{1}{2} \eta d_\Sigma\left(2H\left(\ln \frac{eR}{d_\Sigma}\right)^2+ \left(\ln \frac{eR}{d_\Sigma}\right)\right)+\frac{1}{4} \geq \epsilon_0.
	\eal
%	$$\BA {ll}\frac{1}{2}gd_\Sigma\left(2H(\ln \frac{eR}{d_\Sigma})^2+(\ln \frac{eR}{d_\Sigma})\right)+\frac{1}{4}\geq \epsilon_0.
%	\EA $$
   Therefore if $d_\Sigma \leq \delta_0$ then
   \be \label{tildeI1} \tilde I_1 \geq \epsilon_0(R^2-r^2)^2.
   \ee
	Denote
	\bal
	&\tilde \CA_1:=\left\{x\in\Gw\cap B_R(0):d_\Sigma(x)\leq \tilde c_1 \frac{R^2-r^2}{ R(\ln\frac{eR}{d_\Sigma})^2}\right\} \quad \text{where} \quad \tilde c_1=\frac{\epsilon_0}{16b(b+1)}, \\
	&\tilde \CA_2:=\left\{x\in\Gw\cap B_R(0):d_\Sigma(x)\leq \gd_0^{\phantom{^4}}\right\}, \quad \tilde \CA_3:=\{ x\in \Omega: d_\Sigma(x) \geq \delta_0 \}.
	\eal
	
	In $\tilde \CA_1 \cap \tilde \CA_2$, for $b \geq \max\{ \frac{N-2}{2},1 \}$, we have
	\be \label{F13-0}
	\tilde I_1 + \tilde I_2 + \tilde I_3
	\geq \frac{(R^2-r^2)^2}{16}.
	\ee
	
	In $\tilde \CA_1^c \cap \tilde \CA_2$, we have
	%\begin{equation}\label{F13}
	$d_\Sigma\geq \tilde c_1\frac{R^2-r^2}{ R(\ln\frac{eR}{d_\Sigma})^2}$.
	%\end{equation}
	If $b>\frac{2}{p-1}$, then we can choose $\Gl$ large enough depending on $p,R_0,k,\xd_0,N,b$ such that
	\bal %\label{F14}
	\tilde I_4 \geq 2\max \left\{  4b(b+1)(R^2-r^2)\left(\ln \frac{eR}{d_\Sigma}\right)^2d_\Sigma R, 4b(b+1)\left(\ln \frac{eR}{d_\Sigma}\right)^2d_\Sigma^2R^2 \right  \}.
	\eal
	This and \eqref{tildeI23} imply
	\be \label{F14-0} \tilde I_2 + \tilde I_3 + \tilde I_4 \geq 0.
	\ee
	
	In $\tilde \CA_3$, $d_\Sigma \geq \delta_0$. Similarly as in Case 1, we can choose $\Gl$ large enough depending on $p,R_0,$ $\delta_0,N,k,b$ such that
	\be \label{F15}
	\tilde I_1 + \tilde I_2 + \tilde I_3 + \tilde I_4 \geq 0.
	\ee
	
	Combining \eqref{F12}, \eqref{tildeI1}, \eqref{F13-0}, \eqref{F14-0} and  \eqref{F15}, we obtain \eqref{F11}.
\end{proof}

We recall here that $\tilde W$ has been defined in \eqref{tildeW}.

\begin{proposition}\label{prop19}
	Let $1<p<\frac{2+\am}{\am}$ if $\am>0$ or $p<\infty$ if $\am\leq0$. Assume that  $F\subsetneq \Sigma$ is a compact subset of $\Sigma$ and denote by $d_F(x)=\dist(x,F)$. There exists a constant $C=C(N,\Omega,\Sigma,\mu,p)$
	such that if $u$ is a nonnegative solution of \eqref{eq:power1} in $\Omega \setminus \Sigma$ satisfying
	\be\label{as1}
	\lim_{x\in\xO\setminus\xS,\;x\rightarrow\xi}\frac{u(x)}{\tilde W(x)}=0\qquad\forall \xi\in (\partial\xO\cup \Sigma)\setminus F,\quad  \text{locally uniformly in } \Sigma\setminus F,
	\ee
	then
	\ba \label{3.4.24}
	u(x)&\leq Cd(x)d_\Sigma(x)^{-\am}d_F(x)^{-\frac{2}{p-1}+\am}\qquad\forall x\in \xO\setminus \Sigma, \\ \label{3.4.24*}
	|\nabla u(x)|&\leq C\frac{d(x)}{\min(d(x),d_\Sigma(x))}d_\Sigma(x)^{-\am} d_F(x)^{-\frac{2}{p-1}+\am}\qquad\forall x\in \xO\setminus \Sigma.
	\ea
\end{proposition}  %%%%%%%%%%%%%%%%%%%%%%%%%%%%%%%%%%%%%%%%%%%%%%%%%%%%%%%%%%%%%%PROOF%%%%%%%%%%%%%%%%%%%%%%%%%%%%%%%%%%%%%%%%%%%%%%%%%%%%%%%%%%%%%%%%%%%%%%%%%%%%%%%%%%%%%%%%%%%%%%%%%
\begin{proof}
	The proof is in the spirit of \cite[Proposition 3.4.3]{MVbook}. Let $\xi\in \Sigma\setminus F$
	and put $d_{F,\xi} = \frac{1}{2} d_F(\xi)<1$. Denote
	\bal \xO^\xi:= \frac{1}{d_{F,\xi}}\Omega = \{y\in\mathbb{R}^N:\;d_{F,\xi}\,y\in\xO\}\quad\text{and}\quad \Sigma^\xi=\frac{1}{d_{F,\xi}}\Sigma =\{y\in\mathbb{R}^N:\;d_{F,\xi}\,y\in \Sigma\}.
	\eal
	If $u$ is a nonnegative solution of \eqref{eq:power1} in $\xO\setminus \Sigma$ then the function
	\bal
	u^\xi (y): = d_{F,\xi}^{\frac{2}{p-1}}
	u(d_{F,\xi}y),\quad y\in\xO^\xi\setminus \Sigma^\xi
	\eal
	is a nonnegative solution of
	\be \label{eq:uxi} -\xD u^\xi- \frac{\mu}{|\mathrm{dist}(y,\Sigma^\xi)|^2} u^\xi +\left(u^\xi\right)^p=0
	\ee
	in $\Omega^\xi \setminus \Sigma^\xi$.

As $d_{F, \xi} \leq 1$ the $C^2$ characteristic of $\xO$ (respectively $\Sigma$) is also a $C^2$ characteristic of $\xO^\xi$ (respectively $\Sigma^\xi$) therefore this constant $C$ can be taken to be independent of $\xi$. Let $R_0=\xb_0$ be the constant in Proposition \ref{barr}.
	Set $r_0=\frac{3R_0}{4},$ and let $w_{r_0,\xi}$ be the supersolution of \eqref{eq:uxi} in $B(\frac{1}{d_{F,\xi}}\xi,r_0)\cap(\xO^\xi\setminus \Sigma^\xi)$ constructed in Proposition \ref{barr} with $R=r_0$ and $z=\frac{1}{d_{F,\xi}}\xi$. By a similar  argument as in the proof of Lemma \ref{lem:KO-1}, we can show that
	\bal
	u^\xi(y) \leq w_{r_0,\xi}(y)\quad \forall y\in B\left(\frac{1}{d_{F,\xi}}\xi,r_0 \right)\cap(\xO^\xi\setminus \Sigma^\xi).
	\eal
	Thus $u^\xi$ is bounded from above in $B(\frac{1}{d_{F,\xi}}\xi,\frac{3R_0}{5})\cap(\xO^\xi\setminus \Sigma^\xi)$ by a constant $C$ depending only $N,k,\mu,p$ and the $C^2$ characteristic of $\Omega$ and $\Sigma$.
	
	Now we note that $u^\xi$ is a nonnegative $L_\xm$ subharmonic function and by the last inequality satisfies, for any $\gamma \in (\am,\ap)$,
\be\label{in1}
u^\xi(y) \leq C\left\{\BA{lll} d_{\Sigma^\xi}(y)^{-\gamma} \quad &\text{ if }\;\mu<H^2, \\[2mm]
	d_{\Sigma^\xi}(y)^{-H}\sqrt{\ln \left(\frac{eR}{d_{\Sigma(y)}}\right)} \;&\text{ if }\;\xm=H^2,
	\EA\right.
\ee
 for any $y\in B(\frac{1}{d_{F,\xi}}\xi,r_0)\cap(\xO^\xi\setminus \Sigma^\xi),$ where $C$ is a positive constant depending only on $R_0, \xg,N,\xb,p$ and the $C^2$ characteristic of $\Sigma$. Hence,
\bal \lim_{y\in\xO^\xi,\;y\rightarrow P}\frac{u^\xi(y)}{\tilde W^\xi(y)}=0\qquad\forall P\in B\left(\frac{1}{d_{F,\xi}}\xi,\frac{3r_0}{5}\right) \cap \Sigma^\xi,
\eal
where
\bal \tilde W^\xi(y) = 1-\eta_{\frac{\beta_0}{d_{F, \xi}}}+\eta_{\frac{\beta_0}{d_{F, \xi}}}W^\xi(y) \quad \text{in } \Omega^\xi \setminus \Sigma^\xi,
\eal
and
\bal
W^\xi(y)=\left\{ \BAL &d_{\Sigma^\xi}(y)^{-\ap}\qquad&&\text{if}\;\mu <H^2, \\
&d_{\Sigma^\xi}(y)^{-H}|\ln d_{\Sigma^\xi}(y)|\qquad&&\text{if}\;\mu =H^2,
\EAL \right. \quad x \in \Omega^\xi \setminus \Sigma^\xi.
\eal

In view of the proof of (3.14) in \cite[Lemma 3.3]{GkiNg_linear} and by \ref{in1} , we can show that there exists a constant $c>0$ depending only on $N, \xm,\xb_0 $ such that
	\be \label{sim}
	u^\xi(y) \leq c\,\dist(y,\Sigma^\xi)^{-\am}\quad \forall y\in B\left(\frac{1}{d_F(\xi)}\xi,\frac{r_0}{2}\right)\cap(\xO^\xi\setminus \Sigma^\xi) .
	\ee
	Therefore, for any $\xi \in \Sigma \setminus F$ such that $d_{F,\xi} \leq \frac{\min(\xb_0,1)}{4}$, there holds
	\be
	u(x)\leq c\,d_\Sigma(x)^{-\am}d_{F,\xi}^{-\frac{2}{p-1}+\am}\quad\forall x\in B\left(\xi,\frac{3\xb_0d_{F,\xi}}{8}\right)\cap (\xO\setminus \Sigma).\label{sim2}
	\ee
	
	Take $x \in \Omega \setminus \Sigma$. If $x\in \xO\setminus \Sigma_{\frac{\xb_0}{2}}$ then  \eqref{3.4.24} follows easily from \eqref{KO-1}. It remains to deal with the case $x\in \Sigma_{\frac{\xb_0}{2}}$. We will consider the following cases.
	
\noindent	\textbf{Case 1:} $ d_F(x) < \frac{4+\xb_0}{2+\xb_0}$.
	If $d_\Sigma(x)\leq \frac{\xb_0}{8+2\xb_0}d_F(x)$
%	\begin{equation} \label{sim30} d_\Sigma(x)\leq \frac{\xb_0}{4+\xb_0}d_F(x)
%	\end{equation}
	then let $\xi$ be the unique point in $\Sigma\setminus F$ such that $|x-\xi|=d_\Sigma(x).$ Then we have
	\be \label{sim3a}
	d_{F,\xi} = \frac{1}{2}d_F(\xi)\leq \frac{1}{2}(d_\Sigma(x)+d_F(x)) \leq \frac{2+\xb_0}{4+\xb_0}d_F(x)<1,
	\ee
	and $d_F(x)\leq \frac{2(8+2\xb_0)}{8+\xb_0}d_{F,\xi}$.
%	\be \label{sim3}
%	d_F(x)\leq \frac{4+\xb_0}{2}d_{F,\xi}.
%	\ee
	Therefore $d_\Sigma(x) \leq \frac{\xb_0}{4}d_{F,\xi}$.
%	\begin{equation} \label{sim3b}
%	d_\Sigma(x) \leq \frac{\xb_0}{2}d_{F,\xi}.
%	\end{equation}
This, combined with \eqref{sim2}, \eqref{sim3a} and the fact that $p<\frac{2+\am}{\am}$, yields
	\bal %\label{sim4}
	u(x)\leq C d_\Sigma(x)^{-\am}d_{F,\xi}^{-\frac{2}{p-1}+\am} \leq Cd_\Sigma(x)^{-\am}d_F(x)^{-\frac{2}{p-1}+\am}.
	\eal
	
%	If $d_\Sigma(x)\leq \frac{\xb_0}{2}d_{F,\xi}$ then by \eqref{sim3} and since $p<\frac{2+\am}{\am}$, we have
%	$$u(x)\leq Cd_\Sigma(x)^{-\am}d_F(x)^{-\frac{2}{p-1}+\am}.$$
%	If $d_\Sigma(x)>\frac{\xb_0}{2}d_{F,\xi},$ then by \eqref{ko} and \eqref{sim3}, there holds
%	
%	$$u(x)\leq C(\xb_0,p)d_{F,\xi}^{-\frac{2}{p-1}}.$$
%	
%	We set $v(y)=u(d_{F,\xi}\,y).$ Then $v$ satisfies
%	
%	$$-\xD v- \frac{\mu}{|\mathrm{dist}(y,\Sigma^\xi)|^2}v+d_{F,\xi}^{2}\left|v\right|^p=0\qquad\mathrm{in}\;\;\xO^\xi\setminus \Sigma^\xi.$$
%	Set $x_0=d_{F,\xi}\frac{\xb_0}{2}\frac{x-\xi}{|x-\xi|}+\xi$ then $x_0 \in B(\xi,d_{F,\xi}\frac{\xb_0}{2})$. Since $$\mu\mathrm{dist}(y,\Sigma^\xi)^{-2} +d_{F,\xi}^{2}v^{p-1}<C(\xb_0,p), \quad  \forall y\in B(x_0,d_F(\xi)\frac{\xb_0}{2}),$$
%	by the Harnack inequality, we obtain
%	$$u(y)\leq C(\xb_0,p) u(x_0),\quad\forall y\in B(x_0,d_F(\xi)\frac{\xb_0}{4}).$$
%	As a consequence, we may apply the Harnack inequality at most $m=m(\xb_0)$ times from to connect $x_0$ to $x$ so that
%	$$u(x)\leq C^{k}(\xb_0,p)u(x_0)\leq  Cd^{-\am}_\Sigma(x)d_{F,\xi}^{-\frac{2}{p-1}+\am} \leq Cd^{-\am}_\Sigma(x)d_F(x)^{-\frac{2}{p-1}+\am}.$$
	
	If   $d_\Sigma(x)> \frac{\xb_0}{8+2\xb_0}d_F(x)$ then by \eqref{ko} and the assumption $p<\frac{2+\am}{\am}$, we obtain
	\bal u(x) \leq Cd_\Sigma(x)^{-\frac{2}{p-1}} \leq Cd_\Sigma(x)^{-\am}d_F(x)^{-\frac{2}{p-1}+\am}.
	\eal
	Thus (\ref{3.4.24}) holds for every $x\in \Sigma_{\frac{\xb_0}{2}}$ such that $d_F(x) < \frac{4+\xb_0}{2 + \xb_0}$. \medskip
	
\noindent	\textbf{Case 2:} $d_F(x) \geq \frac{4+\xb_0}{2 + \xb_0}$.
	Let $\xi$ be the unique point in $\Sigma\setminus F$ such that $|x-\xi|=d_\Sigma(x)$. Since $u$ is an $L_\mu$-subharmonic function in $B(\xi,\frac{\xb_0}{4}) \cap (\Omega \setminus \Sigma)$.

By \eqref{as1} and \cite[Lemma 3.3 and estimate (2.10)]{GkiNg_linear}, we deduce that
	\bal
	u(x)\leq C d_\Sigma(x)^{-\am} \leq  Cd_\Sigma(x)^{-\am} d_F(x)^{-\frac{2}{p-1}+\am}\qquad\forall x\in B\left(\xi,\frac{\xb_0}{2}\right)\cap (\Omega \setminus \Sigma).
	\eal
In view of the proof of \ref{sim}, we may show that $C$ depends only on $\xb_0, \xg,N,\xb,p$ and the $C^2$ characteristic of $\Sigma.$	
	
	(ii) Let $x_0\in\xO\setminus \Sigma$. Put $\ell=\dist(x_0,\Omega \setminus \Sigma)=\min\{d(x_0),d_\Sigma(x_0)\}$ and
	\bal (\Omega \setminus \Sigma)^{\ell}:=\frac{1}{\ell}(\Omega \setminus \Sigma)= \{y\in\mathbb{R}^N:\;\ell y\in\xO\setminus \Sigma\}, \quad d_{(\Omega \setminus \Sigma)^\ell}(y):=\mathrm{dist}(y,\partial (\Omega \setminus \Sigma)^\ell).
	\eal
	If $x\in B(x_0,\frac{\ell}{2})$ then $y=\ell^{-1}x$ belongs to $B(y_0,\frac{1}{2}),$ where $y_0=\ell^{-1}x_0$. Also we have that
	$\frac{1}{2}\leq d_{(\Omega \setminus \Sigma)^\ell}(y)\leq \frac{3}{2}$ for each $y\in B(y_0,\frac{1}{2})$. Set $v(y)=u(\ell y)$ for $y\in B(y_0,\frac{1}{2})$ then $v$ satisfies
	\bal
	-\xD v- \frac{\xm}{d_{(\Omega \setminus \Sigma)^\ell}^2} v + \ell^2 \left|v\right|^p=0\qquad\mathrm{in}\;\;B(y_0,\frac{1}{2}).
	\eal
	By standard elliptic estimate we have
	\bal
	\sup_{y\in B(y_0,\frac{1}{4})}|\nabla v(y)|\leq C\left(\sup_{y\in B(y_0,\frac{1}{3})}|v(y)|+\sup_{y\in B(y_0,\frac{1}{3})}\ell^2|v(y)|^p\right),
	\eal
This, together with the equality $\nabla v(y)=\ell \nabla u(x)$, estimate \eqref{3.4.24} and the assumption on $p$, implies	
	\bal
	|\nabla u(x_0)|&\leq C\ell^{-1} \left(d(x_0) d_\Sigma^{-\am}(x_0)d_F(x_0)^{-\frac{2}{p-1}+\am} + \ell^2 d(x_0)^pd_\Sigma(x_0)^{-\am p}d_F(x_0)^{p\left(-\frac{2}{p-1}+\am\right)}\right) \\
	&\leq  C\frac{d(x_0)}{\min\{d(x_0),d_\Sigma(x_0)\}} d_\Sigma(x_0)^{-\am}d_F(x_0)^{-\frac{2}{p-1}+\am} \left[1+ \left( \frac{d_\Sigma(x_0)}{d_F(x_0)} \right)^{2 - (p-1)\am} \right] \\
	&\leq C\frac{d(x_0)}{\min\{d(x_0),d_\Sigma(x_0)\}} d_\Sigma(x_0)^{-\am}d_F(x_0)^{-\frac{2}{p-1}+\am}.
	\eal
Therefore estimate \eqref{3.4.24*} follows since $x_0$ is an arbitrary point. The proof is complete.
\end{proof}
%%%%%%%%%%%%%%%%%%%%%%%%%%%%%%%%

%%%%%%%%%%%%%%%%%%%%%%%%%%%%%%%%%%%%%%%%%%%%%%

\end{document}